\documentclass[reqno]{amsart}
\usepackage[left=1in,right=1in,top=1in,bottom=1in]{geometry}
\usepackage{microtype}
\usepackage{hyperref}

\usepackage{amsmath}             
\usepackage{amsfonts}             
\usepackage{amsthm}               
\usepackage{amsbsy}
\usepackage{amssymb}
\usepackage{amsthm}
\usepackage[nobysame]{amsrefs}

\usepackage{enumerate}

\newtheorem{thm}{Theorem}[section]
\newtheorem{lem}[thm]{Lemma}
\newtheorem{prop}[thm]{Proposition}
\newtheorem{cor}[thm]{Corollary}

\newtheorem{ques}[thm]{Question}

\theoremstyle{definition}
\newtheorem{defn}[thm]{Definition}

\newtheorem{rem}[thm]{Remark}
\newtheorem{exam}[thm]{Example}

\newcommand{\bC}{{\mathbb{C}}}
\newcommand{\bN}{{\mathbb{N}}}
\newcommand{\bR}{{\mathbb{R}}}

\newcommand{\A}{{\mathcal{A}}}
\newcommand{\B}{{\mathcal{B}}}

\newcommand{\J}{{\mathcal{J}}}

\newcommand{\M}{{\mathcal{M}}}

\newcommand{\X}{{\mathcal{X}}}
\newcommand{\Y}{{\mathcal{Y}}}

\newcommand{\vX}{\mathbf{X}}
\newcommand{\vY}{\mathbf{Y}}
\newcommand{\vS}{\mathbf{S}}
\newcommand{\vT}{\mathbf{T}}

\renewcommand{\phi}{\varphi}

\newcommand{\fA}{{\mathfrak{A}}}

\newcommand{\fM}{{\mathfrak{M}}}

\newcommand{\qand}{\quad\text{and}\quad}
\newcommand{\qqand}{\qquad\text{and}\qquad}

\newcommand{\alg}{\mathrm{alg}}

\newcommand{\Tr}{\mathrm{Tr}}

\newcommand{\diag}{\mathrm{diag}}

\newcommand{\sa}{\mathrm{sa}}

\newcommand{\set}[1]{\left\{#1\right\}}
\newcommand{\ang}[1]{\left\langle#1\right\rangle}
\newcommand{\paren}[1]{\left(#1\right)}
\newcommand{\sq}[1]{\left[#1\right]}

\newcommand{\slr}{\set{\ell, r}}

\usepackage{tikz}
\usetikzlibrary{shapes,snakes,calc,arrows}
\usetikzlibrary{decorations.pathreplacing,shapes.geometric}
\usetikzlibrary{calc,positioning}
\tikzset{Box/.style={very thick, rounded corners}}
\tikzset{marked/.style={star, star point height = .75mm, star points =5, fill=black,minimum size=2mm, inner sep=0mm} }
\tikzset{verythickline/.style = {line width=7pt}}
\tikzset{thickline/.style = {line width=5pt}}
\tikzset{medthick/.style = {line width=3pt}}
\tikzset{med/.style = {line width=2pt}}
\tikzset{count/.style = {fill=white,circle,draw,thin, inner sep=2pt}}
\tikzset{rcount/.style = {fill=white,rectangle,draw,thin,inner sep=2pt, rounded corners}}
\tikzset{cpr/.style = {draw,fill=white,rectangle,thin, rounded corners}}

\definecolor{ggreen}{HTML}{7FDD99}

\begin{document}

\nocite{*}

\title[Bi-Free Entropy: Non-Microstate]{Analogues of Entropy in Bi-Free Probability Theory: Non-Microstate}

\author{Ian Charlesworth}
\address{Department of Mathematics, University of California, Berkeley, California, 94720, USA}
\email{ilc@math.berkeley.edu}

\author{Paul Skoufranis}
\address{Department of Mathematics and Statistics, York University, 4700 Keele Street, Toronto, Ontario, M3J 1P3, Canada}
\email{pskoufra@yorku.ca}

\date{\today}
\thanks{The research of the second author was supported in part by NSERC (Canada) grant RGPIN-2017-05711.}

\begin{abstract}
	In this paper, we extend the notion of non-microstate free entropy to the bi-free setting.  Using a diagrammatic approach involving bi-non-crossing diagrams, bi-free difference quotients are constructed as analogues of the free partial derivations.  Adjoints of bi-free difference quotients are discussed and used to define bi-free conjugate variables.  Notions of bi-free Fisher information, non-microstate bi-free entropy, and non-microstate bi-free entropy dimension are defined and known properties in the free setting are extended to the bi-free setting.
\end{abstract}

\maketitle

\section{Introduction}
\label{sec:Intro}

In a series of revolutionary papers \cites{V1993, V1994, V1996, V1997, V1998-2, V1999}, Voiculescu generalized the notions of entropy and Fisher's information to the free probability setting.
In particular, \cite{V1998-2} introduced a non-microstate notion of free entropy, in contrast to the microstates-based approach pioneered in \cite{V1994}.
The non-microstates approach to entropy takes its inspiration from Fisher information in probability and studies the behaviour of non-commutative distributions under infinitesimal perturbations by free Brownian motion,through tracial formulae related to the free difference quotients.
Non-microstate free entropy and the techniques developed to study it led to many advances in free probability theory with ramifications to the study of von Neumann algebras.  For example, these techniques were used to demonstrate specific type II$_1$ factors are non-$\Gamma$ \cite{D2010}, to establish free monotone transport \cite{GS2014}, and to show the absence of atoms in free product distributions \cites{CS2014, MSW2017}.

Recently in \cite{V2014} Voiculescu extended the notion of free probability to simultaneously study the left and right actions of algebras on reduced free product spaces.
This so-called bi-free probability has attracted the attention of many researchers and has had numerous developments (see \cites{BBGS2017, C2016, CNS2015-1, CNS2015-2, S2016-1, S2016-2, S2016-3, S2016-4, HW2016} for example).
The interest surrounding bi-free probability is the possibility to extend the techniques of free probability to solve problems pertaining to pairs of von Neumann algebras, such as a von Neumann algebra and its commutant, or the tensor product of von Neumann algebras.

One important development in bi-free probability theory was the diagrammatical and combinatorial approach using bi-non-crossing partitions developed in \cites{CNS2015-1, CNS2015-2}.
As a diagrammatical view of the free conjugate variables is possible using non-crossing partitions, in this paper we extend this diagrammatical view using \cites{CNS2015-1, CNS2015-2} to develop a notion of non-microstate bi-free entropy.
In our sister paper \cite{CS2017} a notion of microstate bi-free entropy is developed.

In addition to this introduction, this paper contains seven sections which are organized as follows.  In Section \ref{sec:DiffQuot} the notion of bi-free difference quotients is introduced.  The left and right bi-free difference quotients are motivated via a diagrammatical view of the free difference quotients and are obtained by connecting nodes to the bottom of bi-non-crossing diagrams.  In particular, in the bi-partite case where all left and right operators commute, the bi-free difference quotients may be viewed as partial derivatives.  Using the bi-free difference quotients, the notions of left and right conjugate variables are introduced.

In Section \ref{sec:Adjoints} adjoints of the bi-free difference quotients are analyzed.  One important fact from \cite{V1998-2} is that a free conjugate variable exists if and only if $1 \otimes 1$ is in the domain of the adjoint of the corresponding free difference quotient.  In the bi-free setting things are more complicated due to the lack of traciality.  It is demonstrated that a bi-free conjugate variable exists if and only if $1 \otimes 1$ is in the domain of the adjoint of a `flipped' bi-free difference quotient; that is, an analogue of the bi-free difference quotient where nodes are connected to the top of diagrams.  In addition, it is demonstrated that large portions of the generating algebras are contained in the domain of the adjoint of these `flipped' bi-free difference quotients, but it remains unknown whether these adjoints are densely defined.

In Section \ref{sec:Properties} additional properties of bi-free conjugate variables are examined.  In particular, most of the properties of the free conjugate variables exhibited in \cite{V1998-2} hold for the bi-free conjugate variables.

In Section \ref{sec:Fisher} the relative bi-free Fisher information is defined (see Definition \ref{defn:bi-fisher}).  In addition, all properties of the relative Free information exhibited in \cite{V1998-2} are extended to the bi-free setting.

In Section \ref{sec:Entropy} we define the non-microstate bi-free entropy (see Definition \ref{defn:entropy}) as an integral of the Fisher information of perturbations by the independent bi-free Brownian motion.  The non-microstate bi-free entropy of every self-adjoint bi-free central limit distribution is computed and agrees with the microstate bi-free entropy as seen by \cite{CS2017}.  Furthermore, natural properties desired for an entropy theory are demonstrated for the non-microstate bi-free entropy and a lower bound based on the non-microstate free entropy of the system obtained by modifying all right variables to be left variables is obtained.

In Section \ref{sec:Entropy-Dimension} we define the non-microstate bi-free entropy dimension.  In particular, known properties and bounds of the non-microstate free entropy dimension are extended to the bi-free setting and it is demonstrated that the bi-free entropy dimension of a bi-free central limit distribution pair equals the dimension of the support of its joint distribution. 

Finally we analyze the question of when bi-free Fisher information being additive implies bi-freeness in Section \ref{sec:Additive-Bi-Free-Fisher-Info} and discuss several open questions  in Section \ref{sec:Ques}.

\subsection{Notation}
Throughout the paper, $\vX$ and $\vY$ will denote tuples of left operators $(X_1, \ldots, X_n)$ and right operators $(Y_1, \ldots, Y_m)$ respectively of possible different length. When it is necessary to specify their lengths we will tend to denote the length of $\vX$ by $n$ and that of $\vY$ by $m$.
By $\hat\vX_i$ we denote the tuple $(X_1, \ldots, X_{i-1}, X_{i+1}, \ldots, X_n)$.
The notation $B\ang{\vX}$ will denote the non-commutative free algebra generated by $B$ and the elements of $\vX$.

\section{Bi-Free Difference Quotients and Conjugate Variables}
\label{sec:DiffQuot}

In this section we will introduce the notions of bi-free difference quotients and bi-free conjugate variables.  We begin by motivating the bi-free difference quotient by analyzing various interpretations of the free difference quotient and free conjugate variables.

\begin{defn}[\cite{V1998-2}] 
	\label{defn:free-diff-quotient}
	Let $B$ be a unital algebra and let $\A = B\langle X \rangle$ be the non-commutative free algebra generated by $B$ and a variable $X$.  The \emph{free derivation corresponding to $X$} (also known as the \emph{free difference quotient in $X$}) is the linear map $\partial_{X} : \A \to \A \otimes \A$ such that
	\begin{align*}
		\partial_{X}(X) &= 1 \otimes 1, \\ 
		\partial_X(b) &= 0 \text{ for all }b \in B, \text{ and}\\
		\partial_{X}(Z_1Z_2) &= \partial_{X}(Z_1) Z_2 + Z_1 \partial_{X}(Z_2)\text{ for all }Z_1, Z_2 \in \A
	\end{align*}
	where $\A \otimes \A$ is viewed as an $\A$-bimodule via
	\[
		Z_1(P \otimes Q)Z_2 = Z_1P \otimes QZ_2.
	\]
\end{defn}

\begin{defn}[\cite{V1998-2}] 
	\label{defn:free-conjugate-variables}
	Let $\fM$ be a von Neumann algebra, $\tau : \fM \to \bC$ be a tracial state on $\fM$, $X \in \fM$ a self-adjoint operator, $B$ a subalgebra of $\fM$ with no algebraic relations with $X$, and $\A = B\langle X \rangle$.
	Let $L_2(\A, \tau)$ denote the GNS Hilbert space of $\A$ with respect to $\tau$ defined by the sesquilinear form $\langle Z_1, Z_2\rangle_{L_2(\A, \tau)} = \tau(Z_2^*Z_1)$ so that there is a left-action of $\A$ on $L_2(\A, \tau)$.
	Consequently $Z\zeta$ is a well-defined element of $L_2(\A, \tau)$ for all $\zeta \in L_2(\A, \tau)$ and $Z \in \A$.
	Define $\tau(Z\zeta) = \langle Z\zeta, 1\rangle_{L_2(\A, \tau)}$ (where $1 \in \A$ is viewed as an element of $L_2(\A, \tau)$).

	The \emph{conjugate variable of $X$ relative to $B$ with respect to $\tau$} is the unique element $\xi \in L_2(\A, \tau)$ (if it exists) such that
	\[
		\tau(Z \xi) = (\tau \otimes \tau)(\partial_{X}(Z))
	\]
	for all $Z \in \A$ (where $\partial_{X}(Z)$ represents computing the free difference quotient algebraically as defined in Definition \ref{defn:free-diff-quotient} and evaluating at elements of $\fM$).  We use $\J(X : B)$ to denote $\xi$ provided $\xi$ exists.
\end{defn}

\begin{rem}
	\label{rem:free-diff-quot-diagram-view}
	Alternatively, the relation between the free difference quotient and conjugate variables may be seen diagrammatically. To begin, under the notation of Definition \ref{defn:free-conjugate-variables}, notice that if $X_1, \ldots, X_k \in B \cup \{X\}$ then
	\begin{align*}
		\tau( X_{1} \cdots X_{k} \xi) &= (\tau \otimes \tau)(\partial_{X}(X_{1} \cdots X_{k})) \\
		&= \sum_{X_q = X} \tau(X_{1} \cdots X_{{q-1}}) \tau(X_{{q+1}} \cdots X_{k}).
	\end{align*}
	This may be viewed diagrammatically as listing $X_{1},\ldots, X_{k}, \xi$ along a horizontal line, drawing all pictures connecting $\xi$ to any $X_{q}$ where $X_q = X$, taking the trace of each component of the diagram, multiplying the results, and then summing over all such diagrams.
	
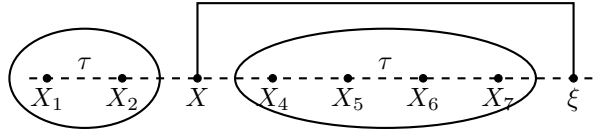
\begin{align*}
		\begin{tikzpicture}[baseline]
			\draw[thick, dashed] (.25,0) -- (-7.25, 0);
			\draw[thick] (-5, 0) -- (-5,1) -- (0,1) -- (0, 0);
			\draw[thick] (-2.5,0) ellipse (2cm and .66cm);
			\draw[thick] (-6.5,0) ellipse (1cm and .66cm);
			\node[above] at (-2.5, 0) {$\tau$};
			\node[above] at (-6.5, 0) {$\tau$};
			\node[below] at (0, 0) {$\xi$};
			\draw[black, fill=black] (0,0) circle (0.05);	
			\node[below] at (-1, 0) {$X_{7}$};
			\draw[black, fill=black] (-1,0) circle (0.05);	
			\node[below] at (-2, 0) {$X_{6}$};
			\draw[black, fill=black] (-2,0) circle (0.05);	
			\node[below] at (-3, 0) {$X_{5}$};
			\draw[black, fill=black] (-5,0) circle (0.05);
			\node[below] at (-4, 0) {$X_{4}$};
			\draw[black, fill=black] (-4,0) circle (0.05);	
			\node[below] at (-5, 0) {$X$};
			\draw[black, fill=black] (-3,0) circle (0.05);
			\node[below] at (-6, 0) {$X_{2}$};
			\draw[black, fill=black] (-6,0) circle (0.05);	
			\node[below] at (-7, 0) {$X_{1}$};
			\draw[black, fill=black] (-7,0) circle (0.05);
		\end{tikzpicture}
	\end{align*}
\end{rem}

To generalize this to the bi-free setting, we will examine an analogue of the above using bi-non-crossing diagrams.
To begin, suppose $B_\ell$ and $B_r$ are unital $*$-algebras, and let $\A = (B_\ell \vee B_r) \langle X, Y\rangle$ for two variables $X$ and $Y$, where $B_\ell \vee B_r$ denotes the unital algebra generated by $B_\ell$ and $B_r$.  One should think of $X$ and $Y$ as being self-adjoint operators, elements of $B_\ell \langle X \rangle$ as being left operators, and elements of $B_r\langle Y\rangle$ as being right operators.  Note that we do not assume we are in the bi-partite setting; that is, we do not assume that elements of $B_\ell \langle X \rangle$ commute with elements of $B_r\langle Y\rangle$.

\begin{defn}
	\label{defn:left-bi-free-diff-quot}
	The \emph{left bi-free difference quotient corresponding to $X$ with respect to $(B_\ell, B_r \langle Y \rangle)$} is the map $\partial_{\ell, X} : \A \to \A \otimes \A$ defined as follows.  Equipping $\A\otimes \A$ with the multiplication $(Z_1 \otimes Z_2) \cdot (W_1 \otimes W_2) = Z_1W_1 \otimes Z_2W_2$, let $T_\ell : \A \to \A \otimes \A$ be the algebra homomorphism defined by
	\[
		T_\ell(x) = 1 \otimes x \qqand T_\ell(y) = y \otimes 1
	\]
	for all $x \in B_\ell \langle X\rangle$ and $y \in B_r \langle Y \rangle$.  Note $T_\ell$ is $*$-preserving when $\A$ is equipped with an involution and $\A \otimes \A$ is equipped with the canonical involution on a tensor product. 	Define $C : \A \otimes \A \to \A$ by
	\[
		C(Z_1 \otimes Z_2) = Z_1Z_2
	\]
	for all $Z_1,Z_2 \in \A$.  Note that $C$ is a homomorphism when $\A \otimes \A$ is equipped with the multiplication $(Z_1 \otimes Z_2) \cdot (W_1 \otimes W_2) = Z_1W_1 \otimes W_2Z_2$ (that is, one uses the opposite multiplication on the second tensor component).
	Then $\partial_{\ell, X} := (C \otimes 1)  \circ  (1\otimes T_\ell)   \circ  \partial_{X}$ where $\partial_X$ is the free derivation for $X$ with respect to $(B_\ell \vee B_r)\langle Y \rangle$.  In particular, $\partial_{\ell, X}$ is not a derivation but a composition of homomorphisms (with differing multiplications) with a derivation.  Also note $C \otimes 1$ is $*$-preserving on the range of $(1\otimes T_\ell)   \circ  \partial_{X}$ provided $B_\ell\langle X \rangle$ and $B_r\langle Y \rangle$ commute with each other.
\end{defn}

\begin{exam}
	To see the diagrammatic view of $\partial_{\ell, X}$, consider the following example. For $x_1, x_2 \in B_\ell$ and $y_1, y_2, y_3 \in B_r \langle Y \rangle$, Definition \ref{defn:left-bi-free-diff-quot} yields
	\begin{align*}
		\partial_{\ell, X}(y_1 X y_1 x_1 y_2 X y_3  y_1x_2) &= ((C \otimes 1)  \circ  (1\otimes T_\ell))\left(y_1 \otimes y_1 x_1 y_2 X y_3  y_1x_2 + y_1 X y_1 x_1 y_2 \otimes y_3  y_1x_2\right)  \\
		&= (C \otimes 1) \left(y_1 \otimes y_1y_2y_3  y_1 \otimes  x_1 X x_2 + y_1 X y_1 x_1 y_2 \otimes y_3  y_1\otimes x_2\right)\\
		& = y_1 y_1 y_2 y_3 y_1 \otimes x_1Xx_2 + y_1Xy_1x_1y_2y_3y_1 \otimes x_2.
	\end{align*}
	This can be observed by drawing $y_1, X, y_1, x_1, y_2, X, y_3, y_1, x _2$ as one would in a bi-non-crossing diagram (i.e. drawing two vertical lines and placing the variables on these lines starting at the top and going down with left variables on the left line and right variables on the right line), drawing all pictures connecting the centre of the bottom of the diagram to any $X$, taking the product of the elements starting from the top and going down in each of the two isolated components of the diagram, and taking the tensor of the two components with the one isolated on the right in the tensor.

	
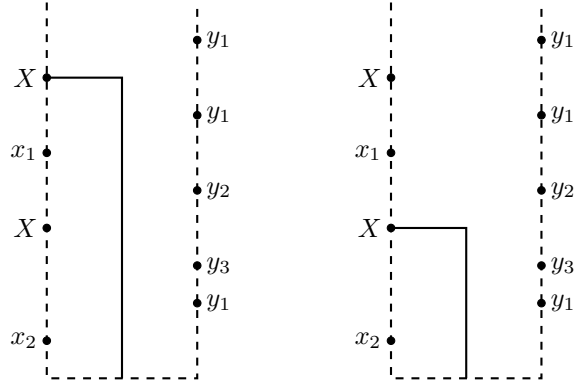
\begin{align*}
		\begin{tikzpicture}[baseline]
			\draw[thick, dashed] (-1,4.5) -- (-1,-.5) -- (1,-.5) -- (1,4.5);
			\draw[thick] (0,-.5) -- (0,3.5) -- (-1,3.5);
			\node[left] at (-1, 3.5) {$X$};
			\draw[black, fill=black] (-1,3.5) circle (0.05);
			\node[left] at (-1, 2.5) {$x_1$};
			\draw[black, fill=black] (-1,2.5) circle (0.05);
			\node[left] at (-1, 1.5) {$X$};
			\draw[black, fill=black] (-1,1.5) circle (0.05);	
			\node[left] at (-1, 0) {$x_2$};
			\draw[black, fill=black] (-1,0) circle (0.05);
			\node[right] at (1, 4) {$y_1$};
			\draw[black, fill=black] (1,4) circle (0.05);	
			\node[right] at (1, 3) {$y_1$};
			\draw[black, fill=black] (1,3) circle (0.05);
			\node[right] at (1, 2) {$y_2$};
			\draw[black, fill=black] (1,2) circle (0.05);
			\node[right] at (1, 1) {$y_3$};
			\draw[black, fill=black] (1,1) circle (0.05);
			\node[right] at (1, .5) {$y_1$};
			\draw[black, fill=black] (1,.5) circle (0.05);
		\end{tikzpicture}
		\qquad  \qquad
		\begin{tikzpicture}[baseline]
			\draw[thick, dashed] (-1,4.5) -- (-1,-.5) -- (1,-.5) -- (1,4.5);
			\draw[thick] (0,-.5) -- (0,1.5) -- (-1,1.5);
			\node[left] at (-1, 3.5) {$X$};
			\draw[black, fill=black] (-1,3.5) circle (0.05);
			\node[left] at (-1, 2.5) {$x_1$};
			\draw[black, fill=black] (-1,2.5) circle (0.05);
			\node[left] at (-1, 1.5) {$X$};
			\draw[black, fill=black] (-1,1.5) circle (0.05);	
			\node[left] at (-1, 0) {$x_2$};
			\draw[black, fill=black] (-1,0) circle (0.05);
			\node[right] at (1, 4) {$y_1$};
			\draw[black, fill=black] (1,4) circle (0.05);	
			\node[right] at (1, 3) {$y_1$};
			\draw[black, fill=black] (1,3) circle (0.05);
			\node[right] at (1, 2) {$y_2$};
			\draw[black, fill=black] (1,2) circle (0.05);
			\node[right] at (1, 1) {$y_3$};
			\draw[black, fill=black] (1,1) circle (0.05);
			\node[right] at (1, .5) {$y_1$};
			\draw[black, fill=black] (1,.5) circle (0.05);
		\end{tikzpicture}
	\end{align*}
\end{exam}

\begin{rem}
	First note $\partial_{\ell, X} : \A \to \A \otimes B_\ell\langle X \rangle$.
	Furthermore, it is elementary to see that the $\partial_{\ell, X}|_{B_\ell \langle X \rangle} = \partial_{X}$.  Thus $\partial_{\ell, X}$ is an extension of the free difference quotient to accommodate right variables. 
\end{rem}

\begin{rem}
	Note that although $\partial_{X}$ does not behave well with respect to commutation of variables, $\partial_{\ell, X}$ does provided the commutation is between left and right variables.  Indeed first notice that if $y \in B_r \langle Y\rangle$ then
	\[
		\partial_{\ell, X}(Z_1 X  y  Z_2) = \partial_{\ell, X}(Z_1 y  X  Z_2)
	\]
	for all $Z_1, Z_2 \in \A$.  Furthermore, if $x \in B_\ell$ is such that $[x, y] = 0$ then 
	\[
		 \partial_{\ell, X}(Z_1 x  y  Z_2)=  \partial_{\ell, X}(Z_1 y  x  Z_2)
	\]
	for all $Z_1, Z_2 \in \A$.
	Thus although we have defined $\partial_{\ell, X}$ assuming that $X, Y, B_\ell$, and $B_r$ share no algebraic relations, $\partial_{\ell, X}$ is well-defined under the above commutation relations.
	In particular $\partial_{\ell, X}$ is well-defined with respect to the relations contained in bi-partite systems.
\end{rem}

\begin{rem}
	\label{rem:left-bi-free-diff-quot-bi-partite}
	The reason that $\partial_{\ell, X}$ is called a difference quotient can be most easily seen in the bi-partite setting.  Indeed suppose that $[x,y] = 0$ for all $x \in B_\ell\langle X \rangle$ and $y \in B_r\langle Y \rangle$.  Then $(B_\ell \vee B_r)\langle X, Y\rangle$ is naturally isomorphic to the algebra $B_r\langle Y \rangle \otimes B_\ell\langle X \rangle$.  In this case
	\[
		\partial_{\ell, X} : B_r\langle Y \rangle \otimes B_\ell\langle X \rangle \to (B_r\langle Y \rangle \otimes B_\ell\langle X \rangle) \otimes B_\ell\langle X \rangle = B_r \langle Y\rangle \otimes (B_\ell\langle X \rangle \otimes B_\ell\langle X \rangle)
	\]
	and, with respect to this decomposition, $\partial_{\ell, X} = id \otimes \partial_X$.
	Thus, if $B_\ell = B_r = \bC$, if we identify  $B_r\langle Y \rangle \otimes B_\ell\langle X \rangle$ with polynomials in the commuting variables $X$ and $Y$, and if we associate $B_r\langle Y \rangle \otimes B_\ell\langle X \rangle \otimes B_\ell\langle X \rangle$ with polynomials in commuting variables $Y, X_1$, and $X_2$, we see that
	\[
		\partial_{\ell, X}(X^nY^m) = \frac{X_1^n - X_2^n}{X_1-X_2} Y^m.
	\] 
	Thus $\partial_{\ell, X}$ really is a partial derivative in the left variable.
\end{rem}

Now we repeat on the right.

\begin{defn}
	\label{defn:right-bi-free-diff-quot}
	The \emph{right bi-free difference quotient with respect to $Y$ corresponding to $(B_\ell \langle X \rangle, B_r)$} is the map $\partial_{r, Y} : \A \to \A \otimes \A$ defined as follows: equipping $\A\otimes \A$ with the multiplication $(Z_1 \otimes Z_2) \cdot (W_1 \otimes W_2) = Z_1W_1 \otimes Z_2W_2$, let $T_r : \A \to \A \otimes \A$ to be the homomorphism such that
	\[
		T_r(x) =   x \otimes 1\qqand T_r(y) = 1 \otimes   y
	\]
	for all $x \in B_\ell\langle X \rangle$ and $y \in B_r\langle Y \rangle$, and let $C$ be as in Definition \ref{defn:left-bi-free-diff-quot}.  Note $T_r$ is a $*$-preserving when $\A$ is equipped with an involution.  Then $\partial_{r, Y} = (C \otimes 1)  \circ  (1 \otimes T_r)   \circ  \partial_{Y}$ where $\partial_Y$ is the free derivation of $Y$ with respect to $(B_\ell \vee B_r)\langle X\rangle$.  In particular, $\partial_{r, Y}$  is not a derivation but a composition of homomorphisms (with differing multiplications) with a derivation.  Also note $C \otimes 1$ is $*$-preserving on the range of $(1 \otimes T_r)   \circ  \partial_{Y}$ provided $B_\ell\langle X \rangle$ and $B_r\langle Y \rangle$ commute with each other.
\end{defn}

\begin{exam}
	To see the diagrammatic view of $\partial_{r, Y}$, consider the following example.  For $x_1, x_2, x_3 \in B_\ell\langle X \rangle$ and $y_1, y_2 \in B_r$, Definition \ref{defn:right-bi-free-diff-quot} yields
	\begin{align*}
		\partial_{r, Y}&(Y x_1 Y x_2 y_1 x_1 y_2  Y x_3)\\
		&= \left((C \otimes 1)  \circ  (1 \otimes T_r)\right)\left(1 \otimes x_1 Y x_2 y_1 x_1 y_2  Y x_3 + Y x_1 \otimes x_2 y_1 x_1 y_2  Y x_3 + Y x_1 Y x_2 y_1 x_1 y_2  \otimes x_3\right) \\
		&= (C \otimes 1) \left(1 \otimes x_1x_2  x_1x_3  \otimes y_1 Y y_2 Y  + Y x_1 \otimes x_2x_1x_3   \otimes  y_1  y_2  Y  + Y x_1 Y x_2 y_1 x_1 y_2  \otimes x_3 \otimes 1  \right) \\
		&= x_1x_2 x_1x_3 \otimes Y y_1y_2Y  + Yx_1x_2x_1x_3 \otimes y_1y_2Y  + Yx_1Yx_2y_1x_1y_2x_3 \otimes 1
	\end{align*}
	This can be observed by drawing  $Y, x_1, Y, x_2, y_1, x_1, y_2,  Y, x_3$ as one would in a bi-non-crossing diagram  (i.e. drawing two vertical lines and placing the variables on these lines starting at the top and going down with left variables on the left line and right variables on the right line), drawing all pictures connecting the centre of the bottom of the diagram to any $Y$, taking the product of the elements starting from the top and going down in each of the two isolated components of the diagram, and taking the tensor of the two components with the one isolated on the right of the tensor.

	
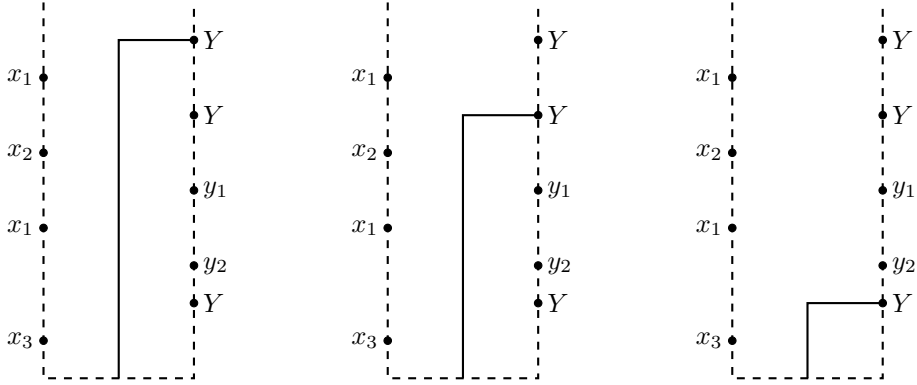
\begin{align*}
		\begin{tikzpicture}[baseline]
			\draw[thick, dashed] (-1,4.5) -- (-1,-.5) -- (1,-.5) -- (1,4.5);
			\draw[thick] (0,-.5) -- (0,4) -- (1,4);
			\node[left] at (-1, 3.5) {$x_1$};
			\draw[black, fill=black] (-1,3.5) circle (0.05);
			\node[left] at (-1, 2.5) {$x_2$};
			\draw[black, fill=black] (-1,2.5) circle (0.05);
			\node[left] at (-1, 1.5) {$x_1$};
			\draw[black, fill=black] (-1,1.5) circle (0.05);	
			\node[left] at (-1, 0) {$x_3$};
			\draw[black, fill=black] (-1,0) circle (0.05);
			\node[right] at (1, 4) {$Y$};
			\draw[black, fill=black] (1,4) circle (0.05);	
			\node[right] at (1, 3) {$Y$};
			\draw[black, fill=black] (1,3) circle (0.05);
			\node[right] at (1, 2) {$y_1$};
			\draw[black, fill=black] (1,2) circle (0.05);
			\node[right] at (1, 1) {$y_2$};
			\draw[black, fill=black] (1,1) circle (0.05);
			\node[right] at (1, .5) {$Y$};
			\draw[black, fill=black] (1,.5) circle (0.05);
		\end{tikzpicture}
		\qquad  \qquad
		\begin{tikzpicture}[baseline]
			\draw[thick, dashed] (-1,4.5) -- (-1,-.5) -- (1,-.5) -- (1,4.5);
			\draw[thick] (0,-.5) -- (0,3) -- (1,3);
			\node[left] at (-1, 3.5) {$x_1$};
			\draw[black, fill=black] (-1,3.5) circle (0.05);
			\node[left] at (-1, 2.5) {$x_2$};
			\draw[black, fill=black] (-1,2.5) circle (0.05);
			\node[left] at (-1, 1.5) {$x_1$};
			\draw[black, fill=black] (-1,1.5) circle (0.05);	
			\node[left] at (-1, 0) {$x_3$};
			\draw[black, fill=black] (-1,0) circle (0.05);
			\node[right] at (1, 4) {$Y$};
			\draw[black, fill=black] (1,4) circle (0.05);	
			\node[right] at (1, 3) {$Y$};
			\draw[black, fill=black] (1,3) circle (0.05);
			\node[right] at (1, 2) {$y_1$};
			\draw[black, fill=black] (1,2) circle (0.05);
			\node[right] at (1, 1) {$y_2$};
			\draw[black, fill=black] (1,1) circle (0.05);
			\node[right] at (1, .5) {$Y$};
			\draw[black, fill=black] (1,.5) circle (0.05);
		\end{tikzpicture}
		\qquad  \qquad
		\begin{tikzpicture}[baseline]
			\draw[thick, dashed] (-1,4.5) -- (-1,-.5) -- (1,-.5) -- (1,4.5);
			\draw[thick] (0,-.5) -- (0,.5) -- (1,.5);
			\node[left] at (-1, 3.5) {$x_1$};
			\draw[black, fill=black] (-1,3.5) circle (0.05);
			\node[left] at (-1, 2.5) {$x_2$};
			\draw[black, fill=black] (-1,2.5) circle (0.05);
			\node[left] at (-1, 1.5) {$x_1$};
			\draw[black, fill=black] (-1,1.5) circle (0.05);	
			\node[left] at (-1, 0) {$x_3$};
			\draw[black, fill=black] (-1,0) circle (0.05);
			\node[right] at (1, 4) {$Y$};
			\draw[black, fill=black] (1,4) circle (0.05);	
			\node[right] at (1, 3) {$Y$};
			\draw[black, fill=black] (1,3) circle (0.05);
			\node[right] at (1, 2) {$y_1$};
			\draw[black, fill=black] (1,2) circle (0.05);
			\node[right] at (1, 1) {$y_2$};
			\draw[black, fill=black] (1,1) circle (0.05);
			\node[right] at (1, .5) {$Y$};
			\draw[black, fill=black] (1,.5) circle (0.05);
		\end{tikzpicture}
	\end{align*}
\end{exam}

\begin{rem}
\label{rem:right-bi-free-diff-quot-bi-partite}
Clearly $\partial_{r, Y}$ shares many properties with $\partial_{\ell, X}$.  Indeed first note $\partial_{r, Y} : \A \to \A \otimes B_r\langle Y \rangle$ and $\partial_{r, Y}|_{B_r \langle Y \rangle} = \partial_{Y}$.  Thus $\partial_{r,Y}$ is an extension of the free partial derivations to accommodate left variables.  Furthermore, similar arguments show that  $\partial_{r, Y}$ is well-behaved with respect to the commutation of left and right operators.  Finally, in the case that $[x,y] = 0$ for all $x \in B_\ell\langle X \rangle$ and $y \in B_r\langle Y \rangle$ so $(B_\ell \vee B_r)\langle X, Y\rangle$ is naturally isomorphic to the algebra $B_\ell\langle X \rangle \otimes B_r\langle Y \rangle$, we see that
	\[
		\partial_{r, Y} : B_\ell\langle X \rangle \otimes B_r\langle Y \rangle \to  (B_\ell\langle X \rangle \otimes B_r\langle Y \rangle) \otimes B_r\langle Y \rangle = B_\ell\langle X \rangle \otimes (B_r\langle Y \rangle \otimes B_r\langle Y \rangle)
	\]
	and, with respect to this decomposition, $\partial_{r, Y} = id \otimes \partial_Y$.
	Thus, if $B_\ell = B_r = \bC$, if we identify  $B_\ell\langle X \rangle \otimes B_r\langle Y \rangle$ with polynomials in the commuting variables $X$ and $Y$, and if we associate $B_\ell\langle X \rangle \otimes B_r\langle Y \rangle \otimes B_r\langle Y \rangle$ with polynomials in commuting variables $X, Y_1$, and $Y_2$, we see that
	\[
		\partial_{r, Y}(X^nY^m) = X^n\frac{Y_1^m - Y_2^m}{Y_1-Y_2}.
	\] 
	Thus $\partial_{r, Y}$ is really a partial derivative in the right variable.
\end{rem}

\begin{rem}
	It is not difficult to verify that the bi-free difference quotients behave well with respect to composition.  In particular
	\begin{align*}
		(\partial_{\ell, X} \otimes id) \circ \partial_{\ell, X} &= (id \otimes \partial_{\ell, X}) \circ \partial_{\ell, X} \\
		(\partial_{r, Y} \otimes id) \circ \partial_{r, Y} &= (id \otimes \partial_{r, Y}) \circ \partial_{r, Y} \\
		(\partial_{\ell, X} \otimes id) \circ \partial_{r, Y} &= \Theta_{(1), (2,3)} \circ (id \otimes \partial_{r, Y}) \circ \partial_{\ell, X} \\
	\end{align*}
	where $\Theta_{(1), (2,3)} : (B_\ell \vee B_r)\langle X, Y\rangle^{\otimes 3} \to (B_\ell \vee B_r)\langle X, Y\rangle^{\otimes 3}$ is defined by
	\[
		\Theta_{(1), (2,3)}(Z_1 \otimes Z_2 \otimes Z_3) = Z_1 \otimes Z_3 \otimes Z_2.
	\]
\end{rem}

The following shows that the bi-free difference quotients truly behaves like partial derivatives on polynomials.

\begin{prop}
	\label{prop:zero-diff-quot-equals-scalar}
	Let $\A = \bC\ang{\vX, \vY}/Z$ where
	\[
	Z = \mathrm{span}\left(\left\{ [X_i, Y_j] \, \mid \, \forall \, i,j\right\}  \right)
	\]
		and define $\Theta_{(1,2)} : \A \otimes \A \to \A \otimes \A$ by $\Theta_{(1,2)}(Z_1\otimes Z_2) = Z_2\otimes Z_1$.
	Let $\partial_{\ell, X_i}$ denote the left bi-free difference quotient of $X_i$ with respect to $\left(\bC\ang{\hat\vX_i}, \bC\ang{\vY}\right)$ and take $\partial_{r, Y_i}$ similarly on the right.  Then, when $\A$ is equipped with the multiplication $(Z_1 \otimes Z_2) \cdot (W_1 \otimes W_2) = Z_1W_1 \otimes Z_2W_2$, for any $P \in \A$, 
	\[
		\sum^n_{i=1} \partial_{\ell, X_i}(P) (X_i \otimes 1) - (1 \otimes X_i) \partial_{\ell, X_i}(P) - \Theta_{1,2}\left( \sum^m_{j=1}   (\partial_{r, Y_j}(P) (Y_j \otimes 1) - (1 \otimes Y_j)  \partial_{r, Y_j}(P)    \right)= P \otimes 1 - 1 \otimes P.
	\]
	In particular, if $P \in \A$ is such that $\partial_{\ell, X_i}(P) = 0 = \partial_{r, Y_j}(P)$ for all $i$ and $j$, then $P$ is a scalar.
\end{prop}

\begin{proof}
By linearity and commutativity of $X_i$ and $Y_j$ for all $i,j$, it suffices to consider the case that $P = X_{i_1} \cdots X_{i_p} Y_{j_1}\cdots Y_{j_q}$.  Then it is easy via commutativity of $X_i$ and $Y_j$ for all $i,j$ to see that
\begin{align*}
\sum^n_{i=1} \partial_{\ell, X_i}(P) (X_i \otimes 1) & = \sum^p_{k=1} X_{i_1} \cdots X_{i_{k-1}} X_{i_k} Y_{j_1}\cdots Y_{j_q} \otimes X_{i_{k+1}} \cdots X_{i_p}, \\
\sum^n_{i=1} (1 \otimes X_i) \partial_{\ell, X_i}(P) & = \sum^p_{k=1} X_{i_1} \cdots X_{i_{k-1}} Y_{j_1}\cdots Y_{j_q} \otimes  X_{i_k} X_{i_{k+1}} \cdots X_{i_p}, \\
 \sum^m_{j=1}   \partial_{r, Y_j}(P) (Y_j \otimes 1)&= \sum^q_{k=1} X_{i_1} \cdots X_{i_p} Y_{j_1} \cdots Y_{j_{k-1}} Y_{j_k} \otimes Y_{j_{k+1}} \cdots Y_{j_q},\text{ and} \\
  \sum^m_{j=1}  (Y_j \otimes 1) \partial_{r, Y_j}(P) &= \sum^q_{k=1} X_{i_1} \cdots X_{i_p} Y_{j_1} \cdots Y_{j_{k-1}}  \otimes Y_{j_k}Y_{j_{k+1}} \cdots Y_{j_q}.
\end{align*}
Thus
\begin{align*}
\sum^n_{i=1} \partial_{\ell, X_i}(P) (X_i \otimes 1) - (1 \otimes X_i) \partial_{\ell, X_i}(P) &= P \otimes 1 - Y_{j_1}\cdots Y_{j_q} \otimes X_{i_1} \cdots X_{i_p} \\
 \sum^m_{j=1}   \partial_{r, Y_j}(P) (Y_j \otimes 1) - (1 \otimes Y_j)  \partial_{r, Y_j}(P)  &= P \otimes 1 - X_{i_1} \cdots X_{i_p} \otimes Y_{j_1}\cdots Y_{j_q}
\end{align*}
Hence the result follows.
\end{proof}

\begin{rem}
	Unfortunately the conclusion of Proposition \ref{prop:zero-diff-quot-equals-scalar} fails in the non-bi-partite setting.  Indeed consider $\A = \bC \langle X, Y\rangle$ with no relations between $X$ and $Y$.  If $P = XY - YX$ then
	\[
		\partial_{\ell, X}(P) = 0 = \partial_{r, Y}(P).
	\]
	Thus, for non-bi-partite systems, there can be non-scalar operators with zero bi-free difference quotients.
\end{rem}

In order to develop bi-free analogues of conjugate variables, we note a cumulant approach to the free conjugate variables from \cite{NSS2002}.  Under the notation of Definition \ref{defn:free-conjugate-variables}, recall that $Z \zeta$ is a well-defined element of $L_2(\A, \tau)$ for all $\zeta \in L_2(\A, \tau)$ and $Z \in \A$.  Consequently we can define the free cumulants of $\zeta \in L_2(\A, \tau)$ with elements $Z_1, \ldots, Z_k \in \A$ via
\[
	\kappa(Z_1, Z_2, \ldots, Z_k, \zeta) = \sum_{\pi \in NC(k+1)} \tau_\pi(Z_1, \ldots, Z_k, \zeta) \mu_{NC}(\pi, 1_{k+1}),
\]
where $NC(k+1)$ denotes the non-crossing partitions on $k+1$ elements, $1_n$ denotes the full partition, $\mu_{NC}$ denotes the M\"{o}bius function on the set of non-crossing partitions, and
\[
	\tau_\pi(Z_1, \ldots, Z_k, \zeta) = \prod_{V \in \pi} \tau\left( \prod_{q \in V} Z_q\right)
\]
where $Z_{k+1} = \zeta$ and the product is performed in increasing order.  Note via M\"{o}bius inversion
\[
	\tau(Z_1 \cdots Z_k \zeta) = \sum_{\pi \in NC(k+1)}\kappa_\pi(Z_1, Z_2, \ldots, Z_k, \zeta)
\]
where
\[
	\kappa_\pi(Z_1, \ldots, Z_k, \zeta) = \prod_{V \in \pi} \kappa\left((Z_1, \ldots, Z_{k+1})|_{V}\right).
\]

Using this notion, we have the following characterization of the free conjugate variables which trivially follows by the M\"{o}bius inversion formula.

\begin{cor}
	Under the notation and assumptions of Definition \ref{defn:free-conjugate-variables}, an element $\xi \in L_2(\A,\tau)$ is the conjugate variable of $X$ with respect to $B$ if and only if
	\begin{align*}
		\kappa_1(\xi) &= 0 \\
		\kappa_2(X, \xi) &= 1 \\
		\kappa_2(b, \xi) &= 0 \text{ for all }b \in B\\
		\kappa_k(Z_1,\ldots, Z_k, \xi) &= 0 \text{ for all }k > 2 \text{ and } Z_1, \ldots, Z_k \in B \cup \set{X}.
	\end{align*}
\end{cor}

Using the above cumulant view of conjugate variables, it is not difficult to develop a bi-free analogue.
To begin, let $\fA$ be a unital C$^*$-algebra and $\varphi : \fA \to \bC$ a state on $\fA$.
We will call $(\fA, \varphi)$ a \emph{C$^*$-non-commutative probability space}.
Note we will assume neither that $\varphi$ is tracial nor faithful on $\fA$ as these properties need not occur in most bi-free systems (see \cite{BBGS2017}*{Theorem 6.1} and \cite{R2017} respectively).
Let $L_2(\fA, \varphi)$ denote the GNS Hilbert space induced from the sesquilinear form $\langle Z_1, Z_2 \rangle_{L_2(\fA, \varphi)} = \varphi(Z_2^*Z_1)$.  Thus there is a left action of $\fA$ on $L_2(\fA, \varphi)$ so that $Z \zeta$ is a well-defined element of $L_2(\fA, \varphi)$ for all $\zeta \in L_2(\fA, \varphi)$ and $Z \in \fA$.  We define $\varphi(Z \zeta) = \langle Z \zeta, 1\rangle_{L_2(\fA, \varphi)}$ (where $1 \in \fA$ is viewed as an element of $L_2(\fA, \tau)$).  

Let $(B_\ell, B_r)$ be a pair of unital subalgebras of $\fA$ that specify left and right operators of $\fA$.  If $\zeta \in L_2(\A,\tau)$, $k \in \bN$, $\chi : \{1,\ldots, k, k+1\} \to \slr$ is such that $\chi(k+1) = \ell$, and $Z_1, \ldots, Z_{k} \in \fA$ are such that $Z_p \in B_{\chi(q)}$, we define the \emph{$\chi$-bi-free cumulant of $Z_1, \ldots, Z_k, \xi$} to be
\[
	\kappa_{\chi}(Z_1, \ldots, Z_{k}, \zeta) = \sum_{\pi \in BNC(\chi)} \varphi_{\pi}(Z_1, \ldots, Z_{k}, \zeta) \mu_{BNC}(\pi, 1_\chi)
\]
where $BNC(\chi)$ denotes the bi-non-crossing partitions with respect to $\chi$, $1_\chi$ denotes the full partition, $\mu_{BNC}$ denotes the M\"{o}bius function on the set of bi-non-crossing partitions (see Remark \ref{rem:partial-mobius-inversion}), and
\[
	\varphi_\pi(Z_1, \ldots, Z_k, \zeta) = \prod_{V \in \pi} \varphi\left( \prod_{q \in V} Z_q\right)
\]
where $Z_{k+1} = \zeta$ and the product is performed in increasing order.   Note via M\"{o}bius inversion
\[
	\varphi(Z_1 \cdots Z_k \zeta) = \sum_{\pi \in BNC(\chi)}\kappa_\pi(Z_1, Z_2, \ldots, Z_k, \zeta)
\]
where
\[
	\kappa_\pi(Z_1, \ldots, Z_k, \zeta) = \prod_{V \in \pi} \kappa_{\chi|_V}\left((Z_1, \ldots, Z_{k+1})|_{V}\right).
\]

Note that we have specified that the entry $\zeta$ is inserted into is treated as a left variable.  Alternatively if $\chi' : \{1, \ldots, k+1\} \to \slr$ is such that $\chi'(k+1) = r$ and $\chi'(p) = \chi(p)$ for all $p \neq k+1$ and we define
\[
	\kappa_{\chi'}(Z_1, \ldots, Z_{k}, \xi) = \sum_{\pi \in BNC(\chi')} \tau_{\pi}(Z_1, \ldots, Z_{k}, \xi) \mu_{BNC}(\pi, 1_{\chi'}),
\]
then it is elementary to see that
\[
	\kappa_{\chi}(Z_1, \ldots, Z_{k}, \xi) = \kappa_{\chi'}(Z_1, \ldots, Z_{k}, \xi)
\]
as there is a bijection between $BNC(\chi)$ to $BNC(\chi')$ obtained by changing the side of the last node which preserves lattice structure.  To summarize, as we have seen throughout the theory of bi-free probability, the first operator to act (which is the last one in any list) can be treated as either a left or as a right and the moment/cumulant formulae do not change.

Using the above, we may now define notions of bi-free conjugate variables. 
\begin{defn}
	\label{defn:bi-free-conjugate variables}
	Let $(\fA, \varphi)$ be a C$^*$-non-commutative probability space, and let $X, Y \in \fA$ be self-adjoint operators.
	Let $B_\ell$ and $B_r$ be unital, self-adjoint subalgebras of $\fA$ such that $X$ and $Y$ satisfy no polynomial relations in $B_\ell\vee B_r$ other than possibly commuting with $B_r$ and $B_\ell$ respectively.
	Denote $\A_X =(B_\ell \vee B_r)\langle X \rangle$ and $\A_Y =(B_\ell \vee B_r)\langle Y \rangle$.
	An element $\xi \in L_2(\A_X, \varphi)$ is said to be a \emph{left bi-free conjugate variable of $X$ with respect to $(B_\ell, B_r)$} and an element $\eta \in L_2(\A_Y, \varphi)$ is said to be a \emph{right bi-free conjugate variable of $Y$ with respect to $(B_\ell, B_r)$} if
	\begin{align*}
		&\kappa_{\ell}(\xi) = 0 & & \kappa_\ell(\eta) = 0\\
		&\kappa_{{\ell, \ell}}(X, \xi) = 1 & & \kappa_{{r, \ell}}(Y, \eta) = 1 \\
		&\kappa_{{\ell, \ell}}(x, \xi)  =  0 \text{ for all }x \in B_\ell & &\kappa_{{\ell, \ell}}(x, \eta)  =  0 \text{ for all }x \in B_\ell  \\
		&\kappa_{{r, \ell}}(y, \xi)  =  0 \text{ for all }y \in B_r& &\kappa_{{r, \ell}}(y, \eta)  =  0 \text{ for all }y \in B_r  \\
		&\kappa_\chi(Z_1, \ldots, Z_{k}, \xi)  = 0  & &\kappa_\chi(Z'_1, \ldots, Z'_{k}, \eta)  = 0 
	\end{align*}
	for all $k \geq 2$, $\chi : \set{1, \ldots, k+1} \to \slr$, and $Z_1, \ldots, Z_k \in \A_X$ and $Z'_1, \ldots, Z'_k \in \A_Y$ where $Z_p \in B_{\ell} \langle X \rangle$ and $Z'_p \in B_\ell$ when $\chi(p) = \ell$, and $Z_p \in B_r$ and $Z'_p \in B_r\langle Y \rangle$ when $\chi(p) = r$.
\end{defn}

\begin{rem}
	By the comments preceding Definition \ref{defn:bi-free-conjugate variables}, it does not matter whether we take $\chi(k+1)$ to be $\ell$ or $r$ as both cumulants are the same, although we may prefer to treat $\xi$ as a left variable and $\eta$ as a right variable.
	There is some subtlety here in that $\xi$ may be a mixture of left and right variables and so should not really be thought of as being either left or right (see, for example, the semicircular case in Example~\ref{exam:conju-of-semis}).
\end{rem}

\begin{rem}
	Due to the moment-cumulant formulae, the values of the cumulants specified in Definition \ref{defn:bi-free-conjugate variables} automatically specify the values of
	\[
		\varphi(Z \xi) = \langle \xi, Z^* \rangle_{L_2(\A_X, \varphi)} \qqand \varphi(Z' \eta) = \langle \eta, Z'^*\rangle_{L_2(\A_Y, \varphi)}
	\]
	for all $Z \in \A_X$ and $Z' \in \A_Y$.  Therefore, by density of an algebra in its $L_2$-space, there is at most one left bi-free conjugate variable for $X$ and at most one right bi-free conjugate variable for $Y$.  As such we will use
	\[
		\J_\ell(X : (B_\ell, B_r)) \qqand \J_r(Y : (B_\ell, B_r))
	\]
	to denote the left bi-free conjugate variable for $X$ with respect to $(B_\ell, B_r)$ and the right bi-free conjugate variable for $Y$ with respect to $(B_\ell, B_r)$, respectively, should they exist.
\end{rem}

\begin{rem}
	\label{rem:bi-conjugate-variable-via-moment-formula}
	It is not difficult using the moment-cumulant formulae to see that $\J_\ell(X : (\B_\ell, B_r))$ exists if and only if there exists an element $\xi \in L_2(\A, \varphi)$ such that
	\[
		\varphi(Z \xi) = (\varphi \otimes \varphi)(\partial_{\ell, X}(Z))
	\]
	for all $Z \in (B_\ell \vee B_r)\ang{X}$, in which case $\J_\ell(X : (\B_\ell, B_r)) = \xi$.  A similar result holds for right bi-free conjugate variables.
	In particular, both views of the free conjugate variables have a consistent interpretation for our bi-free conjugate variables.
\end{rem}

\begin{exam}
	\label{exam:conju-of-semis}
	Let $(S, T)$ be a self-adjoint bi-free central limit distribution with respect to a state $\varphi$ such that $\varphi(S^2) = \varphi(T^2) = 1$ and $\varphi(ST) = \varphi(TS) = c \in (-1,1)$ (see \cite{V2014}*{Section 7}). Then
	\[
		\J_\ell(S : (\bC, \bC\langle T\rangle )) = \frac{1}{1-c^2}(S - cT).
	\]
	To see this via cumulants, let $\xi = \frac{1}{1-c^2}(S - cT)$.  Clearly $\varphi(\xi) = 0$.  Furthermore,
	\[
		\kappa_{\ell, \ell}(S, \xi) = \varphi(S\xi) = \frac{1}{1-c^2}\left( \varphi(S^2) - c \varphi(TS)\right) = \frac{1}{1-c^2}(1 - c^2) = 1
	\]
	and
	\[
		\kappa_{r, \ell}(T, \xi) = \varphi(T \xi) = \frac{1}{1-c^2} \left( \varphi(ST) - c \varphi(T^2) \right) = \frac{1}{1-c^2}(c - c) = 0.
	\]
	Finally, all higher order cumulants involving $\xi$ vanish as bi-free cumulants of order at least three with entries in $S$ and $T$ vanish (and thus so to do those involving $S^n$, $T^m$, and $\xi$ by the $(\ell, r)$-cumulant expansion formula from \cite{CNS2015-2}*{Theorem 9.1.5}) and due to the fact that it does not matter whether the last entry in a cumulant expression is treated as a left or as a right operator.
	
	Alternatively, we can derive our expression for $\J_\ell(S : (\bC, \bC\langle T\rangle ))$ using moments.  To see this, it suffices by linearity and commutativity to show for all $n,m \in \bN \cup \{0\}$ that
	\[
	\varphi(S^n T^m \J_\ell(S : (\bC, \bC\langle T\rangle ))) = (\varphi \otimes \varphi)(\partial_{\ell, S}(S^n T^m)).
	\]
	Using the moment-cumulant formula together with the knowledge of the bi-free cumulants for bi-free central limit distributions, we see that
	\begin{align*}
	\varphi(S^n T^m S) &= \sum^{n-1}_{i=0} \varphi(S^i T^m) \varphi(S^{n-i-1}) + \sum^{m-1}_{j=0} c \varphi(S^n T^{j}) \varphi(T^{m-j-1}) \text{ and}\\
	\varphi(S^n T^m T) &= \sum^{n-1}_{i=0} c\varphi(S^i T^m) \varphi(S^{n-i-1}) + \sum^{m-1}_{j=0} \varphi(S^n T^{j}) \varphi(T^{m-j-1}).
	\end{align*}
Hence it follows that
\[
	\varphi\left( S^n T^m \left( \frac{1}{1-c^2}(S - cT)  \right)\right) =  \sum^{n-1}_{i=0} \varphi(S^i T^m) \varphi(S^{n-i-1}) = (\varphi \otimes \varphi)(\partial_{\ell, S}(S^n T^m)),
\]
as desired.

	A similar argument shows that $$\J_r(T:(\bC\ang{S}, \bC)) = \frac{1}{1-c^2}(T-cS).$$
\end{exam}

\begin{exam}
	\label{exam:bi-free-conjugate-independence}
	Under the notation and assumptions of Definition \ref{defn:bi-free-conjugate variables}, suppose that $B_\ell\langle X \rangle$ and $B_r \langle Y \rangle$ are classically independent with respect to $\varphi$; that is, $B_\ell\langle X \rangle$ and $B_r \langle Y \rangle$ commute and $\varphi(xy) = \varphi(x) \varphi(y)$ for all $x \in B_\ell \langle X\rangle$ and $y \in B_r\langle Y \rangle$.
	Then $L_2(\A, \varphi) = L_2(B_\ell \langle X \rangle, \varphi) \otimes  L_2(B_r\langle Y \rangle, \varphi)$, and it is not difficult to see based on Remark \ref{rem:bi-conjugate-variable-via-moment-formula} that $\J_\ell(X : (B_\ell, B_r\langle Y \rangle))$ exists if and only if $\J(X : B_\ell)$ exists in which case
	\[
		\J_\ell(X : (B_\ell, B_r\langle Y \rangle)) = \J(X : B_\ell) \otimes 1.
	\]
	Similarly $\J_r(Y : (B_\ell\langle X \rangle, B_r))$ exists if and only if $\J(Y : B_r)$ exists in which case
	\[
		\J_r(Y : (B_\ell\langle X \rangle, B_r)) = 1 \otimes \J(Y : B_r).
	\]
\end{exam}

As a generalization of the above, the bi-free conjugate variables for a bi-partite system can be described via their joint distribution.

\begin{prop}
	\label{prop:conjugate-variable-integral-description-bi-partite}
	Let $(X, Y)$ be a pair of commuting self-adjoint operators in a C$^*$-non-commutative probability space.  Let $\mu_{X,Y}$ denote the joint distribution of $(X, Y)$ and suppose $\mu_{X,Y}$ is absolutely continuous with respect to the two-dimensional Lebesgue measure with density $f(x,y) \in L_3(\bR^2, d\lambda_2)$.  Thus $L_2(\alg(X, Y), \varphi) = L_2(\bR^2, f(x,y) \, d\lambda_2)$ and the distributions of $X$ and $Y$ are absolutely continuous with respect to the one-dimensional Lebesgue measure with distributions
	\[
		f_X(x) = \int_\bR f(x,y) \, dy \qqand f_Y(y) = \int_\bR f(x,y) \, dx
	\]
	respectively.  Let
	\begin{align*}
		D &= \mathrm{supp}(\mu_{X, Y}), \\
		D_X &= \mathrm{supp}(\mu_{ Y}), \text{ and} \\
		D_Y &= \mathrm{supp}(\mu_{Y}).
	\end{align*}

	For $\epsilon > 0$ let
	\[
		g_\epsilon(x) = \int_\bR \frac{x-s}{(x-s)^2 + \epsilon^2} f_X(s) \, ds \qqand G_\epsilon(x,y) = \int_\bR \frac{x-s}{(x-s)^2 + \epsilon^2} f(s,y)  \, ds.
	\]
	Suppose $h_X, \xi \in L_2(\bR^2, f(x,y) \, d\lambda_2)$ are such that
	\[
		h_X(x,y) = \lim_{\epsilon \to 0+} g_\epsilon(x) \qqand \xi(x,y) = \lim_{\epsilon \to 0+} \frac{f_X(x)G_\epsilon (x,y)}{f(x,y)} 1_{\{(x,y) \, \mid f(x,y) \neq 0\}}
	\]
	with the limits being in $L_2(\bR^2, f(x,y) \, d\lambda_2)$ (in particular, $h_X$ is, up to a factor of $\pi$, the Hilbert transform of $f_X$).  If $D = D_X \times D_Y$ (up to sets of $\lambda_2$-measure zero) then
	\[
		\J_\ell(X : (\bC, \bC\langle Y\rangle)) = h_X(x,y) + \xi(x,y).
	\]
	The analogous result holds for $\J_r(Y : (\bC\langle X \rangle, \bC))$.
\end{prop}

\begin{proof}
	Note $D$, $D_X$, and $D_Y$ are compact sets.  By the theory of the Hilbert transform (see \cite{SW1971}) $g_\epsilon$ converges in $L_3(\bR, d\lambda)$ to $\pi$ times the Hilbert transform of $f_X$.  Since $h_X$ and $f_X$ are in $L_3(\bR, d\lambda)$, we infer that $h_X \in L_2(\bR, f_X(x) \, d\lambda(x))$ and $g_\epsilon$ converges to $h_X$ in  $L_2(\bR, f_X(x) \, d\lambda(x))$. 

	Let $\A = \alg(X, Y) = \text{span}\{X^nY^m \, \mid \, n,m \in \bN \cup \{0\}\}$.  Thus a vector $\eta \in L_2(\A, \varphi)$ is $\J_\ell(X : (\bC, \bC\langle Y\rangle))$ if and only if
	\[
		\varphi(X^nY^m \eta) = (\varphi \otimes \varphi)(\partial_{\ell, X}(X^n Y^m))
	\]
	for all $n,m \in \bN \cup \{0\}$.  

	Notice for all $n,m \in \bN \cup \{0\}$ that
	\begin{align*}
		(\varphi \otimes \varphi)(\partial_{\ell, X}(X^n Y^m)) &= \sum^{n-1}_{k=0} \varphi(X^k) \varphi(X^{n-k-1}Y^m) \\
		&= \sum^{n-1}_{k=0}  \iint_D s^k f(s,t) \, ds \, dt \iint_D x^{n-k-1} y^m f(x,y) \, dx \, dy \\ 
		&= \sum^{n-1}_{k=0}  \iint_D  \iint_D s^k  x^{n-k-1} y^m f(s,t)f(x,y) \, ds \, dt  \, dx \, dy \\ 
		&= \iint_D  \iint_D \frac{x^n-s^n}{x-s} y^mf(s,t) f(x,y) \, ds \, dt  \, dx \, dy \\ 
		&= \lim_{\epsilon \to 0+}\iint_D  \iint_D \frac{(x-s)(x^n-s^n)}{(x-s)^2 + \epsilon^2} y^mf(s,t) f(x,y) \, ds \, dt  \, dx \, dy
	\end{align*}
	as $\mu_{X,Y}$ is a compactly supported probability measure.  Furthermore, notice
	\begin{align*}
		&\lim_{\epsilon \to 0+}\iint_D  \iint_D \frac{(x-s)x^n}{(x-s)^2 + \epsilon^2} y^mf(s,t) f(x,y) \, ds \, dt  \, dx \, dy \\
		&=\lim_{\epsilon \to 0+}\iint_D  \left(\iint_D \frac{x-s}{(x-s)^2 + \epsilon^2} f(s,t) \, dt \, ds \right) x^n y^m f(x,y)   \, dx \, dy \\
		&=\lim_{\epsilon \to 0+}\iint_D  \left(\int_{D_X} \frac{x-s}{(x-s)^2 + \epsilon^2} f_X(s) \, ds \right) x^n y^m f(x,y)  \, dx \, dy \\
		&=\iint_D  h_X(x,y) x^n y^m f(x,y)  \, dx \, dy 
	\end{align*}
	and
	\begin{align*}
		&\lim_{\epsilon \to 0+}\iint_D  \iint_D \frac{(x-s)(-s^n)}{(x-s)^2 + \epsilon^2} y^mf(s,t) f(x,y) \, ds \, dt  \, dx \, dy \\
		&=\lim_{\epsilon \to 0+}\iint_D  \iint_D \frac{(s-x)s^n}{(s-x)^2 + \epsilon^2} y^mf(s,t) f(x,y) \, ds \, dt  \, dx \, dy\\
		&=\lim_{\epsilon \to 0+}\iint_D s^n f(s,t) \left( \iint_D \frac{(s-x)}{(s-x)^2 + \epsilon^2} y^m f(x,y)  \, dx \, dy\right) \, dt \, ds \\
		&= \lim_{\epsilon \to 0+} \int_{D_X} \int_{D_Y} s^n y^m f_X(s) G_\epsilon(s,y)  \, dy \, ds \\
		&= \lim_{\epsilon \to 0+} \iint_{D} s^n y^m f_X(s) G_\epsilon(s,y)  \, dy \, ds \\
		&= \lim_{\epsilon \to 0+} \iint_D x^n y^m \frac{f_X(x) G_\epsilon(x,y)}{f(x,y)} f(x,y)  \, dy \, dx \\
		&= \iint_D x^ny^m \xi(x,y) f(x,y) \, dy \, dx.
	\end{align*}
	Therefore
	\[
		(\varphi \otimes \varphi)(\partial_{\ell, X}(X^n Y^m)) = \varphi(X^nY^m h_X) + \varphi(X^nY^m \xi) = \varphi(X^nY^m(h_X + \xi))
	\]
	as desired.
\end{proof}

\begin{rem}
	\label{rem:formula-for-conjugate-variables-in-the-bi-partite-situation}
	Note that $h_X$ from Proposition \ref{prop:conjugate-variable-integral-description-bi-partite} is equal to $\frac{1}{2} \J(X : \bC)$ by \cite{V1998-2}*{Proposition 3.5}.  Furthermore, heuristically, if
	\[
		H_X(x,y) = \int_\bR \frac{f(s,y)}{x-s} \, dx
	\]
	(that is, $H_X$ is, up to a factor of $\pi$, the pointwise Hilbert transform of $x \mapsto f(x,y)$), then
	\[
		\J_\ell(X : (\bC, \bC\langle Y\rangle)) = h_X(x) + \frac{f_X(x) H_X(x,y)}{f(x,y)} 1_{\{(x,y) \, \mid f(x,y) \neq 0\}}.
	\]
	Thus $\J_\ell(X : (\bC, \bC\langle Y\rangle))$ looks like half of $\J(X : \bC)$ plus a mixing term.  

	In the case that $(X, Y)$ are classically independent, we see that $f(x,y) = f_X(x)f_Y(y)$ so $H_X(x,y) = h_X(x) f_Y(y)$ and
	\[
		\frac{f_X(x) H_X(x,y)}{f(x,y)} 1_{\{(x,y) \, \mid f(x,y) \neq 0\}} = h_X(x).
	\]
	Hence $\J_\ell(X : (\bC, \bC\langle Y\rangle)) =  \J(X : \bC)$ which is consistent with Example \ref{exam:bi-free-conjugate-independence}.
\end{rem}

\begin{rem}
	\label{rem:conjugate-variables-to-free-conjugate}
	Based on Proposition \ref{prop:conjugate-variable-integral-description-bi-partite}, it is not surprising that the existence of the bi-free conjugate variables implies the existence of the free conjugate variables.  Indeed, under the assumptions and notation of Definition \ref{defn:bi-free-conjugate variables} suppose $\xi = \J_\ell(X : (B_\ell, B_r))$ exists.  If $P : L_2(\A, \varphi) \to L_2(B_\ell\langle X\rangle, \varphi)$ is the orthogonal projection onto $L_2(B_\ell\langle X\rangle, \varphi)$, then it is elementary to see that $P(\xi) = \J(X : B_\ell)$.  A similar result holds for the right bi-free conjugate variables.
\end{rem}

\begin{rem}
	In relation to Proposition \ref{prop:conjugate-variable-integral-description-bi-partite}, it is natural to ask whether the converse holds; that is, if the conjugate variables exist for a bi-partite pair, does the formula for the conjugate variables from Proposition \ref{prop:conjugate-variable-integral-description-bi-partite} hold, and must it be the case that $D = D_X \times D_Y$?
	Note this latter condition can be interpreted as that there is not too much degeneracy between the variables (i.e. if the support of the distribution is not a product, the two variables are more closely related).

To analyze this question, first note that if the conjugate variables exist then by Remarks \ref{rem:formula-for-conjugate-variables-in-the-bi-partite-situation} and \ref{rem:conjugate-variables-to-free-conjugate} we must have that $h_X$ exists.   By performing the same computations  in the proof of Proposition \ref{prop:conjugate-variable-integral-description-bi-partite}, we find that
	$$\varphi\paren{X^nY^mJ_\ell(X:(\bC,\bC\ang{Y}))} = \varphi\paren{X^nY^mh_X(x,y)} + \lim_{\epsilon \to 0+} \int_{D_X}\int_{D_Y} x^ny^mf_X(x)G_\epsilon(x,y)\,dy\,dx.$$
	Thus $f_X(x)G_\epsilon(x, y)$ converges weakly to $f(x,y)\paren{J_\ell(X:(\bC, \bC\ang{Y}))(x, y) - h_X(x)}$ in  $L_2(\bR^2, \lambda_2)$.  Hence for almost every $(x, y) \notin D$, either $f_X(x) = 0$ or $\lim_{ \epsilon \to 0+} G_\epsilon(x, y) = 0$.  

	In an attempt to show that $D = D_X \times D_Y$, note clearly $D \subseteq D_X \times D_Y$.  Suppose we can find a $y_0 \in \bR$ so that $f_Y(y_0) > 0$, and $f_X(x) > 0$ but $f(x, y_0) = 0$ for all $x$ in a set $S$ of positive Lebesgue measure.
	Note the function $$z \mapsto \int_\bR \frac{f(x, y_0)}{z-x}\,dx$$ is holomorphic on the upper half plane and satisfies
	$$\lim_{\epsilon\to 0+} \int_\bR\frac{f(x, y_0)}{x_0+i\epsilon - x}\,dx = -i\pi f(x_0, y_0) + \lim_{\epsilon\to 0+}G_\epsilon(x_0, y_0)$$
	for all $x_0 \in S$.  Hence this holomorphic function tends to zero as $z$ tends non-tangentially to any $x_0 \in S$.  Therefore, if it was the case that $S$ was second category and dense in an open interval then the Lusin-Privalov Theorem \cite{LP1925} would imply the holomorphic function is zero in the upper half plane and thus $f(x, y_0)$ would be identically zero.  Repeating on the right would then yield $D = D_X \times D_Y$.  From this we can conclude that any bi-partite pairs that have conjugate variables outside of Proposition \ref{prop:conjugate-variable-integral-description-bi-partite} are pathological.
\end{rem}

Remark \ref{rem:conjugate-variables-to-free-conjugate} demonstrates a connection between the free and bi-free conjugate variables.  In the tracially bi-partite setting (where all left operators commute with all right operators and the restriction of the state to both the left algebra and the right algebra is tracial), this connection runs deeper.

\begin{lem}
	\label{lem:converting-rights-to-lefts}
	Under the assumptions and notation of Definition \ref{defn:bi-free-conjugate variables} suppose that $(\bC\ang{\vX}, \bC\ang{\vY})$ is tracially bi-partite, with $\vX$ an $n$-tuple, and $\vY$ $m$-tuple of self-adjoint operators.
	Suppose further that $\vX$ and $\vY$ satisfy no relations other than $[X_i, Y_j] = 0$ for each $i$ and $j$.

	Assume that there exists another C$^*$-non-commutative probability space $(\A_0, \tau_0)$ and tuples of self-adjoint operators $\vX', \vY'$ such that $\tau_0$ is tracial on $\A_0$ and
	\[
		\varphi(X_{i_1} \cdots X_{i_p} Y_{j_1} \cdots Y_{j_q}) = \tau(X'_{i_1} \cdots X'_{i_p} Y'_{j_q} \cdots Y'_{j_1})
	\]
	for all $p,q \in \bN \cup \{0\}$ and $i_1, \ldots, i_p \in \{1,\ldots, n\}$ and $j_1, \ldots, j_q \in \{1, \ldots, m\}$.
	Then there is an isometric map $\Psi : L_2(\A, \varphi) \to L_2(\A_0, \tau_0)$ such that
	\[
		\Psi(X_{i_1} \cdots X_{i_p} Y_{j_1} \cdots Y_{j_q}) = X'_{i_1} \cdots X'_{i_p} Y'_{j_q} \cdots Y'_{j_1}
	\]
	for all $p,q \in \bN \cup \{0\}$, and for all $i_1, \ldots, i_p \in \{1,\ldots, n\}$ and $j_1, \ldots, j_q \in \{1, \ldots, m\}$.  Furthermore, if $P : L_2(\A_0, \tau_0) \to  \Psi(L_2(\A, \varphi))$ is the orthogonal projection onto $\Psi(L_2(\A, \varphi))$, then
		\[\J_\ell\paren{X_i : \paren{\bC\ang{\hat\vX_i}, \bC\ang{\vY}}} = \Psi^{-1}\paren{P\paren{\J\paren{X'_i : \bC\ang{\hat\vX'_i, \vY'}}}}\]
provided $\J\paren{X'_i : \bC\ang{\hat\vX'_i, \vY'}}$ exists.  A similar result holds on the right.
\end{lem}

\begin{proof}
	For notational simplicity, let
	\[
		\xi_i = \J_\ell\paren{X_i : \paren{\bC\ang{\hat\vX_i}, \bC\ang{\vY}}} \qqand
		\xi'_i = \J\paren{X'_i : \bC\ang{\hat\vX'_i, \vY'}},
	\]
provided they exist.
Now $\A = \bC\ang{\vX, \vY}$ is generated by monomials of the form $p(\vX)q(\vY)$ and admits no relations other than commutation between $X$'s and $Y$'s, so we may define $\Psi$ as desired on $\A$ without issue.

We claim that $\Psi$ extends to an isometry.
To see this, it suffices to verify that it preserves inner products between monomials.  Suppose $p_1(\vX)q_1(\vY)$ and $p_2(\vX)q_2(\vY)$ are two monomials.
Let $q_1'$ and $q_2'$ be obtained from $q_1$ and $q_2$ by reversing the order of the variables (so that, e.g., $\Psi(q_1(\vY)) = q_1'(\vY')$)
Notice that
\begin{align*}
	\tau_0\paren{\Psi\paren{p_1(\vX)q_1(\vY)}^*\Psi\paren{p_2(\vX)q_2(\vY)}}
	&=\tau_0\paren{q_1'(\vY')^*p_1(\vX')^*p_2(\vX')q_2'(\vY')}\\
	&= \tau_0\paren{p_1(\vX')^*p_2(\vX')q_2'(\vY')q_1'(\vY')^*}\\
	&= \varphi\paren{p_1(\vX)^*p_2(\vX)q_1(\vY)^*q_2(\vY)}\\
	&= \varphi\paren{q_1(\vY)^*p_1(\vX)^*p_2(\vX)q_2(\vY)} \\
	&= \varphi\left( (p_1 (\vX)q_1(\vY))^* (p_2(\vX) q_2(\vY)) \right).
\end{align*}
Here we have used the definition of $\Psi$, the fact that $\tau_0$ is tracial, the relation between $\tau_0$ and $\varphi$, and the fact that the elements of $\vX$ commute with those in $\vY$.
Hence $\Psi$ is an isometry and thus extends to a well-defined isometry from $L_2(\A, \varphi)$ to $L_2(\A_0, \tau_0)$.

Suppose that $\xi'_i$ exists.
To see that $\xi_i$ exists and $\xi_i = \Psi^{-1}(P(\xi'_i))$, we will demonstrate that $\Psi^{-1}(P(\xi'_i))$ satisfies the appropriate moment formula described in Remark \ref{rem:bi-conjugate-variable-via-moment-formula} to be the bi-free conjugate variable.
Once again let $p(\vX)q(\vY)$ be a monomial with $p(\vX) = X_{i_1}\cdots X_{i_k}$, and let $q'$ be obtained from $q$ by reversing its letters.
Then
\begin{align*}
	\varphi\paren{p(\vX)q(\vY)\Psi^{-1}(P(\xi'_i))}
	&= \ang{ \Psi^{-1}(P(\xi'_i)), q(\vY)^*p(\vX)^* }_{\varphi}\\
	&= \ang{ P(\xi'_i),  p(\vX')^*q'(\vY')^* }_{\tau_0}\\
	&= \ang{  \xi'_i, p(\vX')^*q'(\vY')^* }_{\tau_0}\\
	&= \tau_0\paren{ q'(\vY')p(\vX) \xi'_i } \\
	&= \sum^k_{j=1} \delta_{i, i_j}     \tau_0(q'(\vY')X_{i_1}'\cdots X_{i_{j-1}}')\tau_0(X_{i_{j+1}}'\cdots X_{i_k}')\\
	&= \sum^k_{j=1} \delta_{i, i_j}     \tau_0(X'_{i_1} \cdots X'_{i_{j-1}} q'(\vY')) \tau_0(X'_{i_{j+1}} \cdots X'_{i_k})\\
	&= \sum^k_{j=1} \delta_{i, i_j}     \varphi(X_{i_1} \cdots X_{i_{j-1}} q(\vY)) \varphi(X_{i_{j+1}} \cdots X_{i_k}).
\end{align*}
Hence $P(\xi'_i) = \xi_i$. 
\end{proof}

We note there are several instances where the hypotheses of Lemma \ref{lem:converting-rights-to-lefts} are satisfied.  Indeed if $(\fM, \tau)$ is a tracial von Neumann algebra, $X_0, Y_0 \in \fM$ are self-adjoint, and $L_2(\fM, \tau)$ denotes the GNS representation of $\fM$ with respect to $\tau$, then $\B(L_2(\fM, \tau))$, the bounded linear operators on $L_2(\fM, \tau)$, may be equipped with the state $\varphi : \B(L_2(\fM, \tau)) \to \bC$ defined by
\[
\varphi(T) = \tau(T(1))
\]
for all $T \in \B(L_2(\fM, \tau))$.  If $X$ and $Y$ denoted left and right multiplication by $X_0$ and $Y_0$ respectively, and $X'$ and $Y'$ denote left multiplication by $X_0$ and $Y_0$ respectively, then the hypotheses of Lemma \ref{lem:converting-rights-to-lefts} are satisfied.

\begin{rem}
	Before we conclude this section by demonstrating an important property of bi-free cumulants, we note a portion of the diagrammatic view of conjugate variables in the free probability setting that is not observed in the bi-free setting due to the lack of traciality.  Under the notation of Definition \ref{defn:free-conjugate-variables}, we note since $\tau$ is tracial that there are left and right actions of $\A$ on $L_2(\A, \tau)$.  Consequently, for $\zeta \in L_2(\A, \tau)$ and $Z_1, Z_2 \in \A$, the element $Z_1 \zeta Z_2$ makes sense as an element of $L_2(\A, \tau)$ and we may define $\tau(Z_1 \zeta Z_2) = \langle Z_1 \zeta Z_2, 1 \rangle_{L_2(\A, \tau)}$.  Hence, if  $X_1, \ldots, X_k \in B \cup \{X\}$ then, due to traciality, for all $p$ 
	\begin{align*}
		\tau(X_{p} \cdots X_{k} \xi X_{1} \cdots X_{p-1}) &= \tau( X_{1} \cdots X_{k} \xi) \\
		&= (\tau \otimes \tau)(\partial_{X}(X_{1} \cdots X_{k})) \\
		&= \sum_{X_q = X} \tau(X_{1} \cdots X_{{q-1}}) \tau(X_{{q+1}} \cdots X_{k}).
	\end{align*}
	This can be viewed diagrammatically via an extension of the view of Remark \ref{rem:free-diff-quot-diagram-view} where we sum over the encapsulated region and the non-encapsulated region.
	
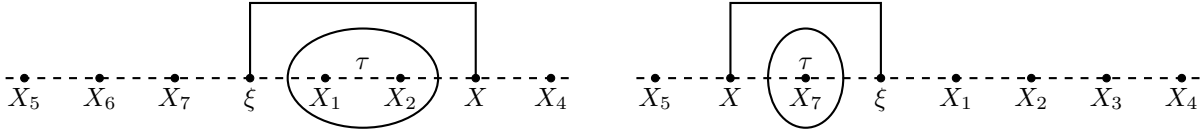
\begin{align*}
		\begin{tikzpicture}[baseline]
			\draw[thick, dashed] (-.25,0) -- (7.25, 0);
			\draw[thick] (3, 0) -- (3,1) -- (6,1) -- (6, 0);
			\draw[thick] (4.5,0) ellipse (1cm and .66cm);
			\node[above] at (4.5, 0) {$\tau$};
			\node[below] at (0, 0) {$X_{5}$};
			\draw[black, fill=black] (0,0) circle (0.05);	
			\node[below] at (1, 0) {$X_{6}$};
			\draw[black, fill=black] (1,0) circle (0.05);	
			\node[below] at (2, 0) {$X_{7}$};
			\draw[black, fill=black] (2,0) circle (0.05);	
			\node[below] at (3, 0) {$\xi$};
			\draw[black, fill=black] (5,0) circle (0.05);
			\node[below] at (4, 0) {$X_{1}$};
			\draw[black, fill=black] (4,0) circle (0.05);	
			\node[below] at (5, 0) {$X_{2}$};
			\draw[black, fill=black] (3,0) circle (0.05);
			\node[below] at (6, 0) {$X$};
			\draw[black, fill=black] (6,0) circle (0.05);	
			\node[below] at (7, 0) {$X_{4}$};
			\draw[black, fill=black] (7,0) circle (0.05);
		\end{tikzpicture}
		\qquad
		\begin{tikzpicture}[baseline]
			\draw[thick, dashed] (-.25,0) -- (7.25, 0);
			\draw[thick] (1, 0) -- (1,1) -- (3,1) -- (3, 0);
			\draw[thick] (2,0) ellipse (.5cm and .66cm);
			\node[above] at (2, 0) {$\tau$};
			\node[below] at (0, 0) {$X_{5}$};
			\draw[black, fill=black] (0,0) circle (0.05);	
			\node[below] at (1, 0) {$X$};
			\draw[black, fill=black] (1,0) circle (0.05);	
			\node[below] at (2, 0) {$X_{7}$};
			\draw[black, fill=black] (2,0) circle (0.05);	
			\node[below] at (3, 0) {$\xi$};
			\draw[black, fill=black] (5,0) circle (0.05);
			\node[below] at (4, 0) {$X_{1}$};
			\draw[black, fill=black] (4,0) circle (0.05);	
			\node[below] at (5, 0) {$X_{2}$};
			\draw[black, fill=black] (3,0) circle (0.05);
			\node[below] at (6, 0) {$X_{3}$};
			\draw[black, fill=black] (6,0) circle (0.05);	
			\node[below] at (7, 0) {$X_{4}$};
			\draw[black, fill=black] (7,0) circle (0.05);
		\end{tikzpicture}
	\end{align*}

	The bi-free analogues developed will not have such a diagrammatic interpretation.
	The main reason for this is that if $\varphi$ is not tracial then it is unclear how to make sense of $Z_1 \zeta Z_2$ as an element of $L_2(\A, \varphi)$ for all $\zeta \in L_2(\A, \varphi)$ and $Z_1, Z_2 \in \A$.
	More specifically if $L_2(\A, \varphi)$ is the GNS Hilbert space given by the left action of $\A$ on itself with respect to the sesquilinear form $\langle Z_1, Z_2 \rangle =\varphi(Z_2^*Z_1)$, then, in general, there need not be a bounded right action of $\A$ on $L_2(\A, \varphi)$.
	Of course there are certain circumstances where such an action occurs, but we do not desire to restrict ourselves to that setting.

	Another thought would be perhaps it is only necessary to have left and right actions of certain elements of $L_2(\A, \varphi)$.
	For example, we are always in the situation that $\A$ is generated by two unital algebras, say $B_\ell$ and $B_r$.
	Thus, as every instance currently studied in bi-free probability requires `left objects' to come from the left algebras, one might think of trying to define a left bi-free conjugate variable as an element of $L_2(B_\ell, \varphi)$.
	In specific cases, such as the tracially bi-partite setting, it is possible to make sense of $Z_1 \zeta Z_2$ as an element of $L_2(\A, \varphi)$ for all $\zeta \in L_2(B_\ell, \varphi)$ and $Z_1, Z_2 \in \A$.
	However, several complications arise when using this definition.
	For example generalizing results such as \cite{V1998-2}*{Proposition 3.6} (which says conjugate variables are preserved under adding a free algebra) fail due to the lack of knowledge of the behaviour of the expectation of elements of $B_r$ onto the $L_2(B_\ell, \varphi)$.
\end{rem}

To conclude this section, we will demonstrate an interesting fact that both further supports the idea of the bi-free conjugate variables being defined using the last entries of bi-free cumulants and will be used in subsequent sections.  In particular, we demonstrate that a cumulant involving a product of left and right entries in the final entry may be expanded as a sum of specific cumulants where the left and right entries in the cumulant are separated.    

To begin, given two partitions $\pi, \sigma \in BNC(\chi)$, let $\pi \vee \sigma$ denote the smallest element of $BNC(\chi)$ greater than $\pi$ and $\sigma$.  Given $p,q \in \bN$ with $p < q$, a $\chi : \{1,\ldots, p\} \to \slr$, and a $\chi' : \set{p, \ldots, q} \to \slr$, define $\widehat{\chi} : \set{1, \ldots, q} \to \slr$ via
\[
	\widehat{\chi}(k) = \begin{cases}
		\chi(k) & \text{if } k < p \\
		\chi'(k) & \text{if }k \geq p
	\end{cases}.
\]
We may embed $BNC(\chi)$ into $BNC(\widehat{\chi})$ via $\pi \mapsto \widehat{\pi}$ where $p+1, p+2, \ldots, q$ are added to the block of $\pi$ containing $p$.  It is not difficult to see that $\hat{\pi}$ will be non-crossing as the new nodes $p, \ldots, q$ occur at the bottom of the diagram and so form an interval in the ordering induced by $\widehat{\chi}$.
Alternatively, this map can be viewed as an analogue of the map on non-crossing partitions from \cite{NSBook}*{Notation 11.9} after applying $s^{-1}_\chi$ (where $s_\chi$ is the permutation that sends $\{1,\ldots, n\}$ to elements of $\chi^{-1}(\{\ell\})$ in increasing order followed by elements of $\chi^{-1}(\{r\})$ in decreasing order).

It is easy to see that $\widehat{1_\chi} = 1_{\widehat{\chi}}$, 
\[
	\widehat{0_\chi} = \set{\{1\}, \{2\}, \ldots, \{p-1\}, \{p, p+1, \ldots, q\}},
\]
and $\pi \mapsto \widehat{\pi}$ is injective and preserves the partial ordering on $BNC$.  Furthermore the image of $BNC(\chi)$ under this map is
\[
	\widehat{BNC}(\chi) = \left[\widehat{0_\chi}, \widehat{1_\chi}\right] = \left[\widehat{0_\chi}, 1_{\widehat{\chi}}\right] \subseteq BNC(\widehat{\chi}).
\]
\begin{rem}
	\label{rem:partial-mobius-inversion}
	Recall that since $\mu_{BNC}$ is the M\"{o}bius function on the lattice of bi-non-crossing partitions, we have for each $\sigma,\pi \in BNC(\chi)$ with $\sigma \leq \pi$ that
	\[
		\sum_{\substack{ \rho \in BNC(\chi) \\ \sigma \leq \rho \leq \pi  }} \mu_{BNC}(\rho, \pi) =  \left\{
			\begin{array}{ll}
				1 & \mbox{if } \sigma = \pi  \\
				0 & \mbox{otherwise }
		\end{array} \right. .
	\]
Since the lattice structure is preserved under the map defined above, we see that $\mu_{BNC}(\sigma, \pi) = \mu_{BNC}(\widehat{\sigma}, \widehat{\pi})$.

	It is also easy to see that the partial M\"{o}bius inversion from \cite{NSBook}*{Proposition 10.11} holds in the bi-free setting; that is, if $f, g : BNC(\chi) \to \bC$ are such that 
	\[
		f(\pi) = \sum_{\substack{\sigma \in BNC(\chi) \\ \sigma \leq \pi}} g(\sigma)
	\]
	for all $\pi \in BNC(\chi)$, then for all $\pi, \sigma \in BNC(\chi)$ with $\sigma \leq \pi$, we have the relation
	\[
		\sum_{\substack{\rho \in BNC(\chi) \\ \sigma \leq \rho \leq \pi }} f(\rho) \mu_{BNC}(\rho, \pi) = \sum_{\substack{\omega \in BNC(\chi) \\ \omega \vee \sigma = \pi }} g(\omega).
	\]
\end{rem}

Following the spirit of \cite{NSBook}*{Theorem 11.12}, we now describe how the bi-free cumulants involving products of operators in terms in the last entry behave.
\begin{lem}
	\label{lem:bottom-can-always-be-expanded}
	Let $(\A, \varphi)$ be a C$^*$-non-commutative probability space, $p,q \in \bN$ with $p < q$, $\chi : \{1,\ldots, p\} \to \slr$, and  $\chi' : \set{p, \ldots, q} \to \slr$.
	If $\pi \in BNC(\chi)$ and $Z_k \in \A$ for all $k \in \{1,\ldots, q\}$, then
	\[
		\kappa_\pi\left(Z_1, \ldots, Z_{p-1}, Z_pZ_{p-1} \cdots Z_{q}\right) =  \sum_{\substack{\sigma \in BNC(\widehat{\chi})\\ \sigma \vee \widehat{0_\chi} = \widehat{\pi}}} \kappa_\sigma(Z_1, \ldots, Z_q).
	\]
	In particular, taking $\pi = 1_\chi$, we have
	\[
		\kappa_{\chi}\left(Z_1, \ldots, Z_{p-1}, Z_pZ_{p-1} \cdots Z_{q}\right)  = \sum_{\substack{\sigma \in BNC(\widehat{\chi})\\ \sigma \vee \widehat{0_\chi} = 1_{\widehat{\chi}}}} \kappa_\sigma(Z_1, \ldots, Z_q).
	\]
\end{lem}

\begin{proof}
	Notice
	\begin{align*}
		\kappa_\pi\left(Z_1, \ldots, Z_{p-1}, Z_pZ_{p-1} \cdots Z_{q}\right) &= \sum_{\substack{\rho \in BNC(\chi) \\ \rho \leq \pi}} \varphi_\rho\left(Z_1, \ldots, Z_{p-1}, Z_pZ_{p-1} \cdots Z_{q}\right) \mu_{BNC}(\rho, \pi) \\
		& = \sum_{\substack{\rho \in BNC(\chi)\\  \rho \leq \pi}} \varphi_{\widehat{\rho}}(Z_1, \ldots, Z_q) \mu_{BNC}(\widehat{\rho}, \widehat{\pi}) \\
		& = \sum_{\substack{\sigma \in BNC(\widehat{\chi}) \\ \widehat{0_\chi} \leq \sigma \leq \widehat{\pi}}} \varphi_{\sigma}(Z_1, \ldots, Z_q) \mu_{BNC}(\sigma, \widehat{\pi}) \\
		& = \sum_{\substack{\sigma \in BNC(\widehat{\chi})\\ \sigma \vee \widehat{0_\chi} = \widehat{\pi}}} \kappa_\sigma(Z_1, \ldots, Z_q)
	\end{align*}
	with the last line following from Remark \ref{rem:partial-mobius-inversion}.
\end{proof}

With Lemma \ref{lem:bottom-can-always-be-expanded} we can now extend the vanishing of mixed cumulants to allow products of left and right operators in the last entry of a cumulant expression.

\begin{prop}
	\label{prop:vanishing-of-mixed-cumulants-with-mixed-bottom}
	Let $(\A, \varphi)$ be a C$^*$-non-commutative probability space and let $\{(A_{k,\ell}, A_{k,r})\}_{k \in K}$ be bi-free pairs of algebras in $\A$.  If $q \geq 2$, if $\chi : \{1,\ldots,q\} \to \slr$, if $\omega : \{1,\ldots, q\} \to K$, if $Z_p \in A_{\omega(p), \chi(p)}$ for all $p < q$, and if $Z_q \in \alg(A_{\omega(q), \ell}, A_{\omega(q), r})$, then
	\[
		\kappa_{\chi}(Z_1, \ldots, Z_q) = 0
	\]
	unless $\omega$ is constant.
\end{prop}

\begin{proof}
	By linearity, it suffices to consider $Z_q$ a product of elements from $A_{\omega(q),\ell}$ and $A_{\omega(q), r}$.
	Lemma \ref{lem:bottom-can-always-be-expanded} then implies $\kappa_{\chi}(Z_1, \ldots, Z_q) $ is a sum of products of $(\ell, r)$-cumulants involving $\{(A_{k,\ell}, A_{k,r})\}_{k \in K}$ where only left elements occur in left entries and right elements occur in right entries.
	As at least one cumulant in each product is mixed by the $\sigma \vee \widehat{0_\chi} = 1_{\widehat{\chi}}$ assumption, the result follows from \cite{CNS2015-2}*{Theorem 4.3.1}.
\end{proof}

\section{Adjoints of Bi-Free Difference Quotients}
\label{sec:Adjoints}

One essential tool in the theory of free conjugate variables is the ability to express the conjugate variables using adjoints of the derivations.
Specifically, given $X$ and a unital self-adjoint algebra $B$ with no algebraic relations, it is possible to view $\partial_X$ as a densely defined, unbounded operator from $L_2(B\langle X \rangle, \tau)$ to $L_2(B\langle X \rangle, \tau) \otimes L_2(B\langle X \rangle, \tau)$ and thus $\partial_X^*$, the adjoint of $\partial_X$, makes sense.
This led to the original definition of conjugate variable in \cite{V1998-2}: $\J(X : B)$ is defined when $1\otimes1 \in \mathrm{dom}(\partial_X^*)$, in which case $\J(X : B) := \partial_X^*(1\otimes1)$.
This characterization is essential for many analytical arguments.

In the bi-free setting, things (unsurprisingly) become more complicated.  Under the notation and assumptions of Definition \ref{defn:bi-free-conjugate variables}, it is not apparent that $1 \otimes 1 \in \mathrm{dom}(\partial_{\ell, X}^*)$ is equivalent to the existence of $\J_\ell(X : (B_\ell, B_r \langle Y \rangle))$ due to complications with adjoints.  However, as taking the adjoint of a product of operators corresponds to vertically flipping a bi-non-crossing diagram, there is a corresponding flipped version of $\partial_{ \ell, X}$ that will play the role of $\partial_{X}$ when it comes to adjoints.  Again these definitions are purely algebraic and we substitute elements of $(\fA, \varphi)$ later.

\begin{defn}
	\label{defn:left-bi-free-diff-quot-flipped}
	Let $B_\ell$ and $B_r$ be unital self-adjoint algebras and let $\A = (B_\ell \vee B_r) \langle X, Y\rangle$ for two variables $X$ and $Y$. The \emph{flipped left bi-free difference quotient of $X$ relative to $(B_\ell, B_r\langle Y \rangle)$} is the map $\hat{\partial}_{\ell, X} : \A \to \A \otimes \A$ defined as follows: equipping $\A \otimes \A$ with the multiplication given by $(Z_1 \otimes Z_2) \cdot (W_1 \otimes W_2) = Z_1 W_1 \otimes Z_2 W_2$, define $\hat{T}_\ell : \A \to \A \otimes \A$ to be the ($*$-)homomorphism such that
	\[
		\hat{T}_\ell(x) =  x \otimes 1 \qqand \hat{T}_\ell(y) = 1 \otimes y
	\]
	for all $x \in B_\ell\langle X \rangle$ and $y \in B_r\langle Y \rangle$ and let $C : \A \otimes \A \to \A$ be as in Definition \ref{defn:left-bi-free-diff-quot}.
	(Notice that $\hat{T}_\ell = T_r$ from Definition \ref{defn:left-bi-free-diff-quot}.) Then $\hat{\partial}_{\ell, X} = (1 \otimes C)  \circ  (\hat{T}_\ell \otimes 1)   \circ  \partial_{X}$ where $\partial_X$ is the free derivation of $X$ relative to $(B_\ell \vee B_r)\langle Y \rangle$.  Thus $\hat{\partial}_{\ell, X}$ is not a derivation but a composition of homomorphisms (using different multiplicative structures) with a derivation.
\end{defn}

\begin{exam}
	To see the diagrammatic view of $\hat{\partial}_{\ell, X}$, consider the following example.  For $x_1, x_2 \in B_\ell$ and $y_1, y_2, y_3 \in B_r \langle Y \rangle$, Definition \ref{defn:left-bi-free-diff-quot-flipped} yields
	\begin{align*}
		\hat{\partial}_{\ell, X}(y_1 Xy_1 x_1 y_2 Xy_3  y_1x_2) &= ((1 \otimes C)  \circ  (\hat{T}_\ell \otimes 1))(y_1 \otimes y_1 x_1 y_2 Xy_3  y_1x_2 + y_1 Xy_1 x_1 y_2 \otimes y_3  y_1x_2)  \\
		&= (1 \otimes C) (1 \otimes y_1 \otimes y_1 x_1 y_2 Xy_3  y_1x_2 +    Xx_1   \otimes y_1 y_1 y_2 \otimes y_3  y_1x_2)   \\
		&= 1 \otimes y_1 y_1 x_1 y_2 Xy_3  y_1x_2 +  Xx_1 \otimes y_1 y_1  y_2 y_3  y_1x_2
	\end{align*}
	This can be observed by drawing  $y_1, X, y_1, x_1, y_2, X, y_3,  y_1, x_2$ as one would in a bi-non-crossing diagram (i.e. drawing two vertical lines and placing the variables on these lines starting at the top and going down with left variables on the left line and right variables on the right line), drawing all pictures connecting the centre of the top of the diagram to any $X$, taking the product of the elements starting from the top and going down in each of the two isolated components of the diagram.
	
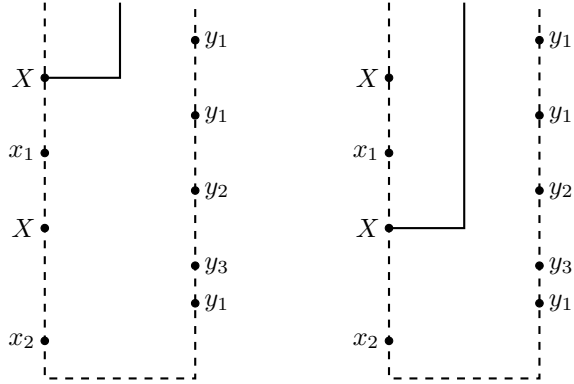
\begin{align*}
		\begin{tikzpicture}[baseline]
			\draw[thick, dashed] (-1,4.5) -- (-1,-.5) -- (1,-.5) -- (1,4.5);
			\draw[thick] (0,4.5) -- (0,3.5) -- (-1,3.5);
			\node[left] at (-1, 3.5) {$X$};
			\draw[black, fill=black] (-1,3.5) circle (0.05);
			\node[left] at (-1, 2.5) {$x_1$};
			\draw[black, fill=black] (-1,2.5) circle (0.05);
			\node[left] at (-1, 1.5) {$X$};
			\draw[black, fill=black] (-1,1.5) circle (0.05);	
			\node[left] at (-1, 0) {$x_2$};
			\draw[black, fill=black] (-1,0) circle (0.05);
			\node[right] at (1, 4) {$y_1$};
			\draw[black, fill=black] (1,4) circle (0.05);	
			\node[right] at (1, 3) {$y_1$};
			\draw[black, fill=black] (1,3) circle (0.05);
			\node[right] at (1, 2) {$y_2$};
			\draw[black, fill=black] (1,2) circle (0.05);
			\node[right] at (1, 1) {$y_3$};
			\draw[black, fill=black] (1,1) circle (0.05);
			\node[right] at (1, .5) {$y_1$};
			\draw[black, fill=black] (1,.5) circle (0.05);
		\end{tikzpicture}
		\qquad  \qquad
		\begin{tikzpicture}[baseline]
			\draw[thick, dashed] (-1,4.5) -- (-1,-.5) -- (1,-.5) -- (1,4.5);
			\draw[thick] (0,4.5) -- (0,1.5) -- (-1,1.5);
			\node[left] at (-1, 3.5) {$X$};
			\draw[black, fill=black] (-1,3.5) circle (0.05);
			\node[left] at (-1, 2.5) {$x_1$};
			\draw[black, fill=black] (-1,2.5) circle (0.05);
			\node[left] at (-1, 1.5) {$X$};
			\draw[black, fill=black] (-1,1.5) circle (0.05);	
			\node[left] at (-1, 0) {$x_2$};
			\draw[black, fill=black] (-1,0) circle (0.05);
			\node[right] at (1, 4) {$y_1$};
			\draw[black, fill=black] (1,4) circle (0.05);	
			\node[right] at (1, 3) {$y_1$};
			\draw[black, fill=black] (1,3) circle (0.05);
			\node[right] at (1, 2) {$y_2$};
			\draw[black, fill=black] (1,2) circle (0.05);
			\node[right] at (1, 1) {$y_3$};
			\draw[black, fill=black] (1,1) circle (0.05);
			\node[right] at (1, .5) {$y_1$};
			\draw[black, fill=black] (1,.5) circle (0.05);
		\end{tikzpicture}
	\end{align*}
\end{exam}

\begin{rem}
	\label{rem:left-connection-between-two-diff-quot}
	Note that $\hat{\partial}_{\ell, X}$ shares the same properties and remarks that were demonstrated for $\partial_{\ell, X}$ in Section \ref{sec:DiffQuot}.  Indeed it is straightforward to check that $\hat{\partial}_{\ell, X}(Z) = \paren{\partial_{\ell, X}(Z^*)}^\star$ where we interpret $(A\otimes B)^\star$ as $B^*\otimes A^*$.
	From this it follows that
	\[
		(\varphi \otimes \varphi)(\hat{\partial}_{\ell, k}(Z^*)^*) = (\varphi \otimes \varphi)(\hat{\partial}_{\ell, k}(Z^*)^\star) = (\varphi \otimes \varphi)(\partial_{\ell, k}(Z)) .
	\]
	Moreover, $\hat{\partial}_{\ell, X}|_{B_\ell\ang{X}} = \partial_X = \partial_{\ell, X}|_{B_\ell\ang{X}}$.
	The reason $\partial_{\ell, X}$ was used over $\hat{\partial}_{\ell, X}$ in the definition of the left bi-free conjugate variables was the connection between $\partial_{\ell, X}$ and the bottom of bi-non-crossing diagrams which enabled the establishment of bi-free conjugate variables via cumulants.
\end{rem}

Similarly, we have the following on the right.

\begin{defn}
	\label{defn:right-bi-free-diff-quot-flipped}
	Let $B_\ell$ and $B_r$ are unital self-adjoint algebras and let $\A = (B_\ell \vee B_r) \langle X, Y\rangle$ for two variables $X$ and $Y$.  The \emph{flipped right bi-free difference quotient of $Y$ relative to $(B_\ell \langle X \rangle, B_r)$} is the map $\hat{\partial}_{r, Y} : \A \to \A \otimes \A$ is defined as follows:   equipping $\A \otimes \A$ with the multiplication given by $(Z_1 \otimes Z_2) \cdot (W_1 \otimes W_2) = Z_1 W_1 \otimes Z_2 W_2$, define $\hat{T}_r : \A \to \A \otimes \A$ to be the ($*$-)homomorphism such that
	\[
		\hat{T}_r(x) =  1 \otimes  x \qqand \hat{T}_r(y) = y \otimes 1
	\]
	for all $x \in B_\ell \langle X \rangle$ and $y \in B_r \langle Y \rangle$,
	and let $C : \A \otimes \A \to \A$ be as in Definition \ref{defn:left-bi-free-diff-quot}.
	(Notice that $\hat{T}_\ell = T_r$ from Definition \ref{defn:right-bi-free-diff-quot}.) Then $\hat{\partial}_{r, Y} = (1 \otimes C)  \circ  ( \hat{T}_r \otimes 1)   \circ  \partial_{Y}$ where $\partial_Y$ is the free derivation of $Y$ with respect to $(B_\ell \vee B_r)\langle X \rangle$.  Thus $\hat{\partial}_{r,  Y}$ is not a derivation but a composition of homomorphisms  (using different multiplicative structures) with a derivation.
\end{defn}

\begin{exam}
	To see the diagrammatic view of $\hat{\partial}_{\ell, X}$, consider the following example.  For $x_1, x_2, x_3 \in B_\ell \langle X \rangle$ and $y_1, y_2 \in B_r$, Definition \ref{defn:right-bi-free-diff-quot-flipped} implies that
	\begin{align*}
		\hat{\partial}_{r, Y} & (Yx_1 Yx_2 y_1x_1 y_2  Yx_3) \\
		&= ((1 \otimes C)  \circ  ( \hat{T}_r \otimes 1) )  (1 \otimes x_1 Yx_2 y_1x_1 y_2  Yx_3 + Yx_1 \otimes x_2 y_1x_1 y_2  Yx_3 + Yx_1 Yx_2 y_1x_1 y_2  \otimes x_3    )\\
		&=(1 \otimes C) (1 \otimes 1 \otimes x_1 Yx_2 y_1x_1 y_2  Yx_3 + Y \otimes x_1 \otimes x_2 y_1x_1 y_2  Yx_3 + Y Y y_1y_2   \otimes x_1x_2 x_1 \otimes x_3    )   \\
		&=1 \otimes x_1 Yx_2 y_1x_1 y_2  Yx_3 + Y   \otimes x_1x_2 y_1 x_1 y_2  Yx_3 + Y Y  y_1  y_2   \otimes   x_1x_2x_1x_3
	\end{align*}
	This can be observed by drawing  $Y, x_1, Y, x_2, y_1, x_1, y_2,  Y, x_3$ as one would in a bi-non-crossing diagram, drawing all pictures connecting the top of the diagram to any $Y$, taking the product of each component of the diagram, and taking the tensor of the two components with the one isolated on the left.
	
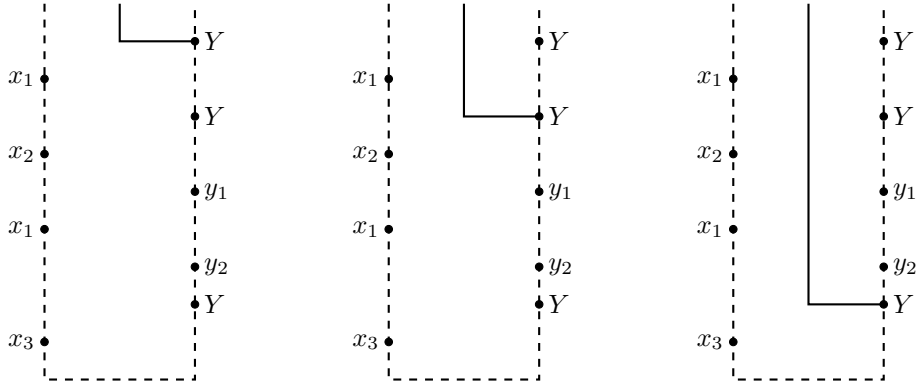
\begin{align*}
		\begin{tikzpicture}[baseline]
			\draw[thick, dashed] (-1,4.5) -- (-1,-.5) -- (1,-.5) -- (1,4.5);
			\draw[thick] (0,4.5) -- (0,4) -- (1,4);
			\node[left] at (-1, 3.5) {$x_1$};
			\draw[black, fill=black] (-1,3.5) circle (0.05);
			\node[left] at (-1, 2.5) {$x_2$};
			\draw[black, fill=black] (-1,2.5) circle (0.05);
			\node[left] at (-1, 1.5) {$x_1$};
			\draw[black, fill=black] (-1,1.5) circle (0.05);	
			\node[left] at (-1, 0) {$x_3$};
			\draw[black, fill=black] (-1,0) circle (0.05);
			\node[right] at (1, 4) {$Y$};
			\draw[black, fill=black] (1,4) circle (0.05);	
			\node[right] at (1, 3) {$Y$};
			\draw[black, fill=black] (1,3) circle (0.05);
			\node[right] at (1, 2) {$y_1$};
			\draw[black, fill=black] (1,2) circle (0.05);
			\node[right] at (1, 1) {$y_2$};
			\draw[black, fill=black] (1,1) circle (0.05);
			\node[right] at (1, .5) {$Y$};
			\draw[black, fill=black] (1,.5) circle (0.05);
		\end{tikzpicture}
		\qquad  \qquad
		\begin{tikzpicture}[baseline]
			\draw[thick, dashed] (-1,4.5) -- (-1,-.5) -- (1,-.5) -- (1,4.5);
			\draw[thick] (0,4.5) -- (0,3) -- (1,3);
			\node[left] at (-1, 3.5) {$x_1$};
			\draw[black, fill=black] (-1,3.5) circle (0.05);
			\node[left] at (-1, 2.5) {$x_2$};
			\draw[black, fill=black] (-1,2.5) circle (0.05);
			\node[left] at (-1, 1.5) {$x_1$};
			\draw[black, fill=black] (-1,1.5) circle (0.05);	
			\node[left] at (-1, 0) {$x_3$};
			\draw[black, fill=black] (-1,0) circle (0.05);
			\node[right] at (1, 4) {$Y$};
			\draw[black, fill=black] (1,4) circle (0.05);	
			\node[right] at (1, 3) {$Y$};
			\draw[black, fill=black] (1,3) circle (0.05);
			\node[right] at (1, 2) {$y_1$};
			\draw[black, fill=black] (1,2) circle (0.05);
			\node[right] at (1, 1) {$y_2$};
			\draw[black, fill=black] (1,1) circle (0.05);
			\node[right] at (1, .5) {$Y$};
			\draw[black, fill=black] (1,.5) circle (0.05);
		\end{tikzpicture}
		\qquad  \qquad
		\begin{tikzpicture}[baseline]
			\draw[thick, dashed] (-1,4.5) -- (-1,-.5) -- (1,-.5) -- (1,4.5);
			\draw[thick] (0,4.5) -- (0,.5) -- (1,.5);
			\node[left] at (-1, 3.5) {$x_1$};
			\draw[black, fill=black] (-1,3.5) circle (0.05);
			\node[left] at (-1, 2.5) {$x_2$};
			\draw[black, fill=black] (-1,2.5) circle (0.05);
			\node[left] at (-1, 1.5) {$x_1$};
			\draw[black, fill=black] (-1,1.5) circle (0.05);	
			\node[left] at (-1, 0) {$x_3$};
			\draw[black, fill=black] (-1,0) circle (0.05);
			\node[right] at (1, 4) {$Y$};
			\draw[black, fill=black] (1,4) circle (0.05);	
			\node[right] at (1, 3) {$Y$};
			\draw[black, fill=black] (1,3) circle (0.05);
			\node[right] at (1, 2) {$y_1$};
			\draw[black, fill=black] (1,2) circle (0.05);
			\node[right] at (1, 1) {$y_2$};
			\draw[black, fill=black] (1,1) circle (0.05);
			\node[right] at (1, .5) {$Y$};
			\draw[black, fill=black] (1,.5) circle (0.05);
		\end{tikzpicture}
	\end{align*}
\end{exam}

\begin{rem}
	Note that $\hat{\partial}_{r,Y}$ and $\partial_{r, Y}$ share the same relation as $\hat{\partial}_{\ell,X}$ and $\partial_{\ell, X}$: we have $\hat{\partial}_{r, Y}(Z) = \paren{\partial_{r, Y}(Z^*)}^\star$, and $\hat{\partial}_{r,Y}|_{B_r \langle Y \rangle} = \partial_{Y} = \partial_{r, Y}|_{B_r \langle Y \rangle}$ and, under the assumptions of Definition \ref{defn:bi-free-conjugate variables}, 
	\[
		(\varphi \otimes \varphi)(\hat{\partial}_{r, Y}(Z^*)^*) = (\varphi \otimes \varphi)(\partial_{r, Y}(Z))
	\]
	for all $Z \in (B_\ell \vee B_r)\langle X, Y\rangle$.
\end{rem}

Using the flipped bi-free difference quotients, we obtain a characterization of bi-free conjugate variables using adjoints of maps.

\begin{thm}
	Under the notation and assumptions used in Definition \ref{defn:bi-free-conjugate variables}, for $\xi \in L_2(\A, \varphi)$ the following are equivalent:
	\begin{enumerate}
		\item $\xi = \J_\ell(X : (B_\ell, B_r \langle Y\rangle))$.
		\item Viewing $\hat{\partial}_{\ell, X} : (B_\ell \vee B_r)\langle X, Y \rangle \to (B_\ell \vee B_r)\langle X, Y \rangle \otimes (B_\ell \vee B_r)\langle X, Y \rangle$ as a densely defined unbounded operator from $L_2(\A, \varphi)$ to $L_2(A, \varphi) \otimes L_2(\A, \varphi)$, we have $1 \otimes 1 \in \mathrm{dom}(\hat{\partial}_{\ell, X}^*)$ and $\hat{\partial}_{\ell, X}^*(1 \otimes 1) = \xi$.
	\end{enumerate}
	A similar results holds for the right bi-free conjugate variables.
\end{thm}

\begin{proof}
	Notice
	\begin{align*}
		\langle 1 \otimes 1, \hat{\partial}_{\ell, X}(Z)\rangle_{\varphi \otimes \varphi} &= (\varphi \otimes \varphi)(\hat{\partial}_{\ell, X}(Z)^*) = (\varphi \otimes \varphi)(\partial_{\ell, X}(Z^*)).
	\end{align*}
	for all $Z \in (B_\ell \vee B_r)\langle X, Y \rangle$.  Furthermore, the defining formula for $\xi = \J_\ell(X : (B_\ell, B_r \langle Y\rangle))$ is that
	\[
		(\varphi \otimes \varphi)(\partial_{\ell, X}(Z^*)) = \varphi(Z^*\xi) = \langle \xi, Z\rangle_{\varphi \otimes \varphi}
	\]
	for all $Z \in (B_\ell \vee B_r)\langle X, Y \rangle$.  Hence the result follows.
\end{proof}

One of the essential reasons why knowing $1 \otimes 1 \in \mathrm{dom}(\partial^*_{X})$ is so important is \cite{V1998-2}*{Corollary 4.2} which states that if $1 \otimes 1 \in \mathrm{dom}(\partial^*_X)$ then $B\langle X \rangle \otimes B\langle X \rangle \in \mathrm{dom}(\partial^*_X)$ and thus $\partial_X$ is pre-closed.  Thus it is natural to ask whether we have a similar result for $\hat\partial^*_{\ell, X}$ and $\hat\partial^*_{r, Y}$.

To begin, notice that
\begin{align*}
	\hat\partial_{\ell, X} &: (B_\ell \vee B_r)\langle X, Y \rangle \to B_\ell\langle X\rangle \otimes (B_\ell \vee B_r)\langle X, Y \rangle \quad \text{and} \\
	\hat\partial_{r, Y} &: (B_\ell \vee B_r)\langle X, Y \rangle \to B_r\langle Y\rangle \otimes (B_\ell \vee B_r)\langle X, Y \rangle
\end{align*}
so the potential domains for $\partial^*_{\ell, X}$ and $\partial^*_{r, Y}$ are $B_\ell\langle X\rangle \otimes (B_\ell \vee B_r)\langle X, Y \rangle$ and $B_r\langle Y\rangle \otimes (B_\ell \vee B_r)\langle X, Y \rangle$ respectively.  To show a good portion of these algebras are in the domains, we note the following.

\begin{lem}
	Let $B_\ell$ and $B_r$ are unital self-adjoint algebras and let $\A = (B_\ell \vee B_r) \langle X, Y\rangle$ for two variables $X$ and $Y$.  
	For all $C, C_1, C_2 \in B_\ell\langle X\rangle$, $D, D_1, D_2 \in B_r\langle Y\rangle$, and $M \in  (B_\ell \vee B_r)\langle X, Y \rangle$, 
	\begin{align*}
		\hat{\partial}_{\ell, X}(C M) &= \hat{\partial}_{\ell, X}(C)(1 \otimes M) + (C \otimes 1) \hat{\partial}_{\ell, X}(M) \\
		\hat{\partial}_{\ell, X}(D_1 MD_2) &= (1 \otimes D_1)\hat{\partial}_{\ell, X}(M)(1 \otimes D_2) \\
		\hat{\partial}_{r, Y}(D M) &= \hat{\partial}_{r, Y}(D)(1\otimes M) + (D \otimes 1) \hat{\partial}_{r, Y}(M)\\
		\hat{\partial}_{r, Y}(C_1 MC_2) &= (1 \otimes C_1)\hat{\partial}_{r, Y}(M)(1 \otimes C_2)
	\end{align*}
	where, in $\A \otimes \A$, $(Z_1 \otimes Z_2)(W_1\otimes W_2) = Z_1W_1 \otimes Z_2W_2$.
\end{lem}

\begin{proof}
	The result trivially follows from the definitions of $\partial_{\ell, X}$ and $\partial_{r, Y}$.
\end{proof}

\begin{prop}
	\label{prop:computing-domains-of-adjoints}
	Under the notation and assumptions of Definition \ref{defn:bi-free-conjugate variables}, consider $\hat{\partial}_{\ell, X} : (B_\ell \vee B_r)\langle X, Y \rangle \to (B_\ell \vee B_r)\langle X, Y \rangle \otimes (B_\ell \vee B_r)\langle X, Y \rangle$ as a densely defined unbounded operator from $L_2(\A, \varphi)$ to $L_2(A, \varphi) \otimes L_2(\A, \varphi)$.  Suppose $\eta \in \mathrm{dom}(\hat{\partial}^*_{\ell, X})$.  Then
	\[
		(C \otimes 1) \eta, (1 \otimes D)\eta \in \mathrm{dom}(\hat{\partial}^*_{\ell, X})
	\]
	for all $C  \in B_\ell\langle X\rangle$ and  $D \in B_r\langle Y\rangle$.  In particular, we have
	\begin{align*}
		\hat{\partial}^*_{\ell, X}((C \otimes 1) \eta) &= C\hat{\partial}_{\ell, X}^*(\eta)   -  (\varphi \otimes id)(\hat{\partial}_{\ell, X}(C^*)^*\eta)\\
		\hat{\partial}^*_{\ell, X}((1 \otimes D)\eta) &= D\partial^*_{X, \ell}(\eta).
	\end{align*}
	Analogous results hold on the right for $\hat{\partial}^*_{r, Y}$.
\end{prop}

\begin{proof}
	Let $p \in (B_\ell \vee B_r)\langle X, Y \rangle$.  Then
	\begin{align*}
		\langle (C \otimes 1) \eta, \hat{\partial}_{\ell, X}(p)\rangle_{L_2(\A, \varphi) \otimes L_2(\A, \varphi)} &= \langle  \eta, (C^* \otimes 1)\hat{\partial}_{\ell, X}(p)\rangle_{L_2(\A, \varphi) \otimes L_2(\A, \varphi)} \\
		&= \langle  \eta, \hat{\partial}_{\ell, X}(C^* p)  - \hat{\partial}_{\ell, X}(C^*)(1 \otimes p)\rangle_{L_2(\A, \varphi) \otimes L_2(\A, \varphi)} \\
		&= \langle  \eta, \hat{\partial}_{\ell, X}(C^* p)\rangle   - \langle \eta, \hat{\partial}_{\ell, X}(C^*)(1 \otimes p)\rangle_{L_2(\A, \varphi) \otimes L_2(\A, \varphi)} \\
		&= \langle \hat{\partial}_{\ell, X}^*(\eta), C^* p\rangle   - \langle \hat{\partial}_{\ell, X}(C^*)^*\eta, (1 \otimes p)\rangle_{L_2(\A, \varphi)} \\
		&= \langle C\hat{\partial}_{\ell, X}^*(\eta), p\rangle   - \langle (\varphi \otimes id)(\hat{\partial}_{\ell, X}(C^*)^*\eta), p\rangle_{L_2(\A, \varphi)} \\
		&= \langle C\hat{\partial}_{\ell, X}^*(\eta)   -  (\varphi \otimes id)(\hat{\partial}_{\ell, X}(C^*)^*\eta), p\rangle_{L_2(\A, \varphi)}
	\end{align*}
	Hence the first claim follows. Furthermore
	\begin{align*}
		\langle (1 \otimes D) \eta, \hat{\partial}_{\ell, X}(p)\rangle_{L_2(\A, \varphi) \otimes L_2(\A, \varphi)} &= \langle \eta,(1 \otimes D^*) \hat{\partial}_{\ell, X}(p)\rangle_{L_2(\A, \varphi) \otimes L_2(\A, \varphi)} \\
		&= \langle \eta, \hat{\partial}_{\ell, X}(D^*p)\rangle_{L_2(\A, \varphi) \otimes L_2(\A, \varphi)} \\
		&= \langle \hat{\partial}^*_{X, \ell}(\eta), D^* p \rangle_{L_2(\A, \varphi)} \\
		&= \langle D\hat{\partial}^*_{X, \ell}(\eta),  p \rangle_{L_2(\A, \varphi)}.
	\end{align*}
	Hence the second claim follows.  The results for the flipped right bi-free difference quotient are similar.
\end{proof}

\begin{cor}
	\label{cor:domains}
	Under the notation and assumptions of Definition \ref{defn:bi-free-conjugate variables}, if $1 \otimes 1 \in \mathrm{dom}(\hat{\partial}^*_{\ell, X})$ then
	\[
		B_\ell\langle X\rangle \otimes B_r \langle Y \rangle \in \mathrm{dom}(\hat{\partial}^*_{\ell, X}).
	\]
	Similarly, if $1 \otimes 1 \in \mathrm{dom}(\hat{\partial}^*_{r, Y})$ then $B_r \langle Y\rangle \otimes B_\ell \langle X \rangle \in \mathrm{dom}(\hat{\partial}^*_{r, Y})$.
\end{cor}

Of course, Corollary \ref{cor:domains} leaves a large question open.

\begin{ques}
	\label{ques:domains}
	If $1 \otimes 1 \in \mathrm{dom}(\hat{\partial}^*_{\ell, X})$ must it be true that
	\[
		B_\ell\langle X\rangle \otimes (B_\ell \vee B_r) \langle X, Y \rangle \subseteq \mathrm{dom}(\hat{\partial}^*_{\ell, X})?
	\]
	If so, we would have a similar result on the right.
\end{ques}

In regards to Question \ref{ques:domains}, the proof in \cite{V1998-2}*{Corollary 4.2} breaks down due to the lack of traciality.
The answer to Question \ref{ques:domains} is also not clear even in the simplest non-trivial setting where traciality does occur.
Indeed suppose $B_\ell = B_r = \bC$, $[X, Y] = 0$, and $1 \otimes 1 \in \mathrm{dom}(\hat{\partial}^*_{\ell, X})$.
If we desired to show that $1 \otimes X^p$ is in the domain of $\hat{\partial}^*_{\ell, X}$ for all $p \in \bN$ (which will then imply the domain of $\hat{\partial}^*_{\ell, X}$ contains all of $\bC \langle X \rangle \otimes \bC\langle Y \rangle$ by Proposition \ref{prop:computing-domains-of-adjoints} since $X$ and $Y$ commute), it suffices to show that for all $n,m \in \bN \cup \{0\}$ that there exists a $\zeta_p \in L_2(\A, \varphi)$ so that
\[
	\langle 1 \otimes X^p, \hat{\partial}_{\ell, X}(X^n Y^m) \rangle_{L_2(\A, \varphi) \otimes L_2(\A, \varphi)} = \langle \zeta_p, X^nY^m \rangle_{L_2(\A, \varphi)}.
\]
Naturally we would proceed by induction on $p$.  So suppose $1 \otimes X^{p-1}$ is in the domain of $\hat{\partial}_{\ell, X}^*$.  Then
\begin{align*}
	\langle 1 \otimes X^p, & \hat{\partial}_{\ell, X}(X^n Y^m) \rangle_{L_2(\A, \varphi) \otimes L_2(\A, \varphi)}\\
	&= \left\langle 1 \otimes X^{p}, \sum^{n-1}_{k=0} X^{k} \otimes X^{n-1-k} Y^m \right\rangle_{L_2(\A, \varphi) \otimes L_2(\A, \varphi)} \\
	&= \left\langle 1 \otimes X^{p-1}, \sum^{n-1}_{k=0} X^{k} \otimes X^{n-k} Y^m \right\rangle_{L_2(\A, \varphi) \otimes L_2(\A, \varphi)} \\
	&= \left\langle 1 \otimes X^{p-1}, \hat{\partial}_{\ell, X}(X^{n+1} Y^m) - X^{n+1} \otimes Y^m       \right\rangle_{L_2(\A, \varphi) \otimes L_2(\A, \varphi)}\\ 
	&= \langle X \hat{\partial}_{\ell, X}^*(1 \otimes X^{p-1}), X^n Y^m  \rangle_{L_2(\A, \varphi) \otimes L_2(\A, \varphi)} - \varphi(X^{n}) \varphi(X^{p-1}Y^m).
\end{align*}
Thus the existence of $\zeta_p$ for all $p$ is equivalent to the existence of $\zeta'_p \in L_2(\A, \varphi)$ for all $p$ so that
\[
	\langle \zeta'_p, X^nY^m \rangle_{L_2(\A, \varphi)} = \varphi(X^{n}) \varphi(X^{p-1}Y^m)
\]
for all $n,m \in \bN \cup \{0\}$.  

Clearly if $X$ and $Y$ are classically independent, then $\zeta'_p = \varphi(X^{p-1})$ would work.   More generally, if the joint distribution of $(X, Y)$ is given by the Lebesgue absolutely continuous measure $f(x,y) \, dx \, dy$ with support $D$ and the distributions of $X$ and $Y$ are given by the Lebesgue absolutely continuous measures $f_X(x) \, dx$ and $f_Y(y) \, dy$ respectively with supports $D_X$ and $D_Y$ respectively such that $D = D_X \times D_Y$, then notice for all $n,m \in \bN \cup \{0\}$ that
\begin{align*}
	\varphi(X^{n}) \varphi(X^{p-1}Y^m) &= \iint_D \iint_D x^n s^{p-1} t^m f(s,t) f(x,y) \, ds \, dt \, dx \, dy \\
	&= \int_{D_X} \int_{D_Y} x^n t^m \mathbb{E}\left[X^{p-1} \, | \, Y= t\right] f_Y(t)  f_X(x) \, dt \, dx \\
	&= \iint_D x^n t^m \mathbb{E}\left[X^{p-1} \, | \, Y = t\right] f_Y(t) f_X(x)  \, dx \, dt \\
	&= \iint_D x^n t^m \frac{\mathbb{E}\left[X^{p-1} \, | \, y = t\right] f_Y(t)  f_X(x)}{f(x,t)} f(x,t)  \, dx \, dt \\
\end{align*}
where
\[
	\mathbb{E}\left[X^{p-1} \, | \, Y = t\right] = \int_{D_X} \frac{s^{p-1} f(s,t)}{f_Y(t)} \, ds.
\]
Hence $\zeta_p$ exists if and only if
\[
	\frac{\mathbb{E}\left[X^{p-1} \, | \, Y = y\right] f_Y(y) f_X(x)}{f(x,y)}
\]
is an element of $L_2(\bR^2, f(x,y) \, dx \, dy)$.  In particular, notice that
\[
	\zeta'_1 = \frac{f_X(x)f_Y(y)}{f(x,y)} \qqand \zeta'_p = \mathbb{E}\left[X^{p-1} \, | \, Y = y\right]  \zeta'_1 \text{ for all }p.
\]

Of course $1 \otimes 1 \in \mathrm{dom}(\hat{\partial}^*_{\ell, X})$ implies that
\[
	\frac{f_X(x) H_X(x,y)}{f(x,y)} 1_{\{(x,y) \, \mid \, f(x,y) \neq 0\}}
\]
is an element of $L_2(\bR^2, f(x,y) \, dx \, dy)$ by Remark \ref{rem:formula-for-conjugate-variables-in-the-bi-partite-situation}.  We believe there is more difficulty in the later being in $L_2(\bR^2, f(x,y) \, dx \, dy)$ than the former so we expect the domain to be dense in this setting.  However, it is not clear that $1 \otimes 1 \in \mathrm{dom}(\hat{\partial}^*_{\ell, X})$ implies $\zeta'_1 \in L_2(\bR^2, f(x,y) \, dx \, dy)$.

For a specific example of the above situation, by \cite{HW2016}*{Example 3.4}, if $(X, Y)$ is a self-adjoint bi-free central limit distribution with variance 1 and covariance $c \in (-1, 1)$, then the joint distribution of $(X, Y)$ is given by the measure $\mu_c$ on $[-2,2]^2$ defined by
\[
	d\mu_c = \frac{1-c^2}{4\pi^2} \frac{\sqrt{4-x^2}\sqrt{4-y^2}}{(1-c^2)^2 - c(1+c^2)xy + c^2(x^2 + y^2)} \, dx \, dy.
\]
Therefore
\[
	\zeta'_1 = \frac{(1-c^2)^2 - c(1+c^2)XY + c^2(X^2 + Y^2)}{1-c^2}
\]
which is clearly an element of $L_2(\bR^2,\mu_c)$ as it is a polynomial.  
To show the existence of $\zeta'_p$ for other $p$, note elementary calculus can be used to show that for a fixed $c \in (-1, 1)$ and $p \in \bN$ that there exists $0 < k_{c,p}  < \infty$ such that
\[
\left| \frac{1-c^2}{2\pi} \frac{x^{p-1}\sqrt{4-x^2}}{(1-c^2)^2 - c(1+c^2)xy + c^2(x^2 + y^2)} \right| \leq k_{c,p}
\]
for all $(x,y) \in [-2,2]^2$ as the minimal value of $(1-c^2)^2 - c(1+c^2)xy + c^2(x^2 + y^2)$ is obtained at $(x,y) = \pm (2,2)$ and is strictly positive.  Hence we obtain that $\mathbb{E}\left[X^{p-1} \, | \, Y = y\right]$ is a bounded function and hence $\zeta'_p = \mathbb{E}\left[X^{p-1} \, | \, Y = y\right] \zeta'_1 \in L_2(\bR^2,\mu_c)$.

\begin{rem}
	Due to the lack of an answer to Question \ref{ques:domains} and the anti-symmetry of Corollary \ref{cor:domains}, it may be useful in the future to flip the tensors in the definition of $\hat{\partial}_{r, Y}$ so that $\hat{\partial}^*_{\ell, X}$ and $\hat{\partial}^*_{r, Y}$ have a common domain (which is a nice algebra).
\end{rem}

\section{Properties of Bi-Free Conjugate Variables}
\label{sec:Properties}

In this section, we will examine the behaviour of the bi-free conjugate variables under several operations.  As bi-free conjugate variables are generalizations of the free conjugate variables, we can expect only to extend known properties to the bi-free setting.  We will use a cumulant approach to the proofs as opposed to the moment approach used in \cite{V1998-2}.  This is done out of ease of working with cumulants.  In most cases, the moment proofs from \cite{V1998-2} generalize, using \cite{C2016} whenever an `alternating centred moment vanish' is required.

We begin with the following which immediately follows from the linearity of cumulants.

\begin{lem}
	\label{lem:deriv-conjugate-variable-scaling}
	Under the assumptions and notation of Definition \ref{defn:bi-free-conjugate variables}, if
	\[
		\xi = \J_\ell(X : (B_\ell, B_r))   
	\]
	exists then for all $\lambda \in \bR \setminus \{0\}$
	\[
		\xi' = \J_\ell(\lambda X : (B_\ell, B_r))   
	\]
	exists and is equal to $\frac{1}{\lambda} \xi$.

	A similar results holds for the right bi-free conjugate variables.
\end{lem}

\begin{lem}
	\label{lem:deriv-conjugate-variable-projecting}
	Under the assumptions and notation of Definition \ref{defn:bi-free-conjugate variables}, if $C_\ell \subseteq B_\ell$ and $C_r \subseteq B_r$ are self-adjoint subalgebras and if
	\[
		\xi = \J_\ell(X : (B_\ell, B_r))   
	\]
	exists   then 
	\[
		\xi' = \J_\ell(X : (C_\ell, C_r))   
	\]
	exists.  In particular, if $P : L_2(\A, \varphi) \to L_2( (C_\ell \vee C_r)\langle X \rangle, \varphi)$ is the orthogonal projection onto the codomain, then $\xi' = P(\xi)$.

	A similar results holds for the right bi-free conjugate variables.
\end{lem}

\begin{proof}
	Since $\varphi(ZP(\xi)) = \varphi(Z\xi)$ for all $Z \in (C_\ell \vee C_r)\langle X \rangle$, it follows that for all $\chi : \{1,\ldots, p\} \to \slr$ with $\chi(p) = \ell$ and for all $Z_k \in (C_\ell \vee C_r)\langle X \rangle$ with $Z_k \in C_\ell \langle X\rangle$ if $\chi(k) = \ell$ and $Z_k \in C_r $ if $\chi(k) = r$ that
	\[
		\kappa_\chi(Z_1, \ldots, Z_{p-1}, P(\xi)) = \kappa_\chi(Z_1, \ldots, Z_{p-1}, \xi).
	\]
	Hence the result follows.
\end{proof}

The following generalizes \cite{V1998-2}*{Proposition 3.6}.
\begin{prop}
	\label{prop:deriv-bi-free-affecting-conjugate-variables}
	Under the assumptions and notation of Definition \ref{defn:bi-free-conjugate variables}, if $(C_\ell, C_r)$ is a pair of unital, self-adjoint subalgebras of $\fA$ such that
	\[
		(B_\ell  \langle X \rangle, B_r ) \qqand (C_\ell, C_r)
	\]
	are bi-free with respect to $\varphi$, then
	\[
		\xi =  \J_\ell(X : (B_\ell, B_r))
	\]
	exists if and only if
	\[
		\xi' =  \J_\ell(X : (B_\ell \vee C_\ell, B_r \vee C_r))
	\]
	exists, in which case they are equal.

	A similar results holds for the right bi-free conjugate variables.
\end{prop}

\begin{proof}
	If $\xi'$ exists, then Lemma \ref{lem:deriv-conjugate-variable-projecting} implies that $\xi$ exists.

	Conversely, suppose the left bi-free conjugate variable $\xi$ exists.  Hence $\xi$ is an $L_2$-limit of elements from $(B_\ell \vee B_r) \langle X\rangle$.  Since the bi-free cumulants are $L_2$-continuous in each entry, it follows that any bi-free cumulant involving at the end $\xi$ and at least one element of $C_\ell$ or $C_r$ must be zero by Proposition \ref{prop:vanishing-of-mixed-cumulants-with-mixed-bottom} as
	\[
		(B_\ell  \langle X \rangle, B_r\rangle) \qqand (C_\ell, C_r)
	\]
	are bi-free.  Therefore, as $L_2((B_\ell \vee B_r)\langle X \rangle, \varphi) \subseteq L_2((B_\ell \vee B_r \vee C_\ell \vee C_r)\langle X \rangle, \varphi)$,  it easily follows that $\xi = \xi'$.
\end{proof}

The following generalizes  \cite{V1998-2}*{Proposition 3.7}.
\begin{prop}
	\label{prop:deriv-sums-affecting-conjugate-variables}
	Let $\vX,\vX'$ be $n$-tuples of self-adjoint operators, let $\vY,  \vY'$ be $m$-tuples of self-adjoint operators, and let $B_\ell$, $B_r$, $C_\ell$, $C_r$ be unital, self-adjoint subalgebras of a C$^*$-non-commutative probability space $(\A, \varphi)$ such that
	\[
		(B_\ell  \langle \vX \rangle, B_r \langle \vY\rangle) \qqand (C_\ell \langle \vX' \rangle, C_r\langle \vY'\rangle)
	\]
	are bi-free and each pair contains no algebraic relations other than possibly elements of the left algebra commuting with elements of the right algebra.  If
	\[
		\xi = \J_\ell\left(X_1 : (B_\ell\langle \hat{\vX}_1 \rangle, B_r \langle \vY \rangle)   \right)
	\]
	exists then
	\[
		\eta = \J_\ell(X_1+X'_1 : ((B_\ell \vee C_\ell) \langle \widehat{(\vX + \vX')}_1\rangle, (B_r \vee C_r) \langle \vY + \vY' \rangle ))
	\]
	exists.  Moreover, if $P$ is the orthogonal projection of $L_2(\fA, \varphi)$ onto $L_2(( (B_\ell \vee C_\ell) \vee (B_r \vee C_r)) \langle \vX + \vX', \vY + \vY'\rangle, \varphi)$, then
	\[
		\eta = P(\xi).
	\]

	A similar results holds for the right bi-free conjugate variables.
\end{prop}

\begin{proof}
	Suppose $\xi$ exists and let $\A = ( (B_\ell \vee C_\ell) \vee (B_r \vee C_r)) \langle \vX + \vX', \vY + \vY'\rangle$.  Since $\varphi(ZP(\xi)) = \varphi(Z\xi)$ for all $Z \in \A$, it follows for all $\chi : \{1,\ldots, p\} \to \slr$ with $\chi(p) = \ell$ and for all $Z_k \in\A$ with $Z_k \in (B_\ell \vee C_\ell) \langle \vX + \vX' \rangle$ if $\chi(k) = \ell$ and $Z_k \in (B_r \vee C_r) \langle  \vY + \vY'\rangle$ if $\chi(k) = r$ that
	\[
		\kappa_\chi(Z_1, \ldots, Z_{p-1}, P(\xi)) = \kappa_\chi(Z_1, \ldots, Z_{p-1}, \xi).
	\]
	Thus any $(\ell, r)$-cumulants involving terms of the form $B_\ell$, $C_\ell$, $B_r$, $C_r$, $X_i + X'_i$, and $Y_j + Y'_j$ and a $\xi$ at the end may be expanded using linearity to involve only terms of the form $B_\ell$, $C_\ell$, $B_r$, $C_r$, $X_i$, $X'_i$, $Y_j$, and $Y'_j$ with a $\xi$ at the end.  These cumulants then obtain the desired values due to Proposition \ref{prop:vanishing-of-mixed-cumulants-with-mixed-bottom}, the fact that
	\[
		(B_\ell  \langle \vX \rangle, B_r \langle \vY\rangle) \qqand (C_\ell \langle \vX' \rangle, C_r\langle \vY'\rangle)
	\]
	are bi-free, and the properties of $\xi$.   Then, using linearity, continuity, and \cite{CNS2015-1}*{Theorem 9.1.5} to expand out cumulants of products, we see that any $(\ell, r)$-cumulants involving terms of the form $B_\ell$, $C_\ell$, $B_r$, $C_r$, $X_i + X'_i$, and $Y_j + Y'_j$ with a $P(\xi)$ at the end is the correct value for $P(\xi)$ to be  $\J_\ell(X_1+X'_1 : ((B_\ell \vee C_\ell) \langle\widehat{(\vX + \vX')}_1\rangle, (B_r \vee C_r) \langle \vY + \vY' \rangle ))$.
\end{proof}

Finally, we arrive at the following  generalization of \cite{V1998-2}*{Corollary 3.9} which enables us to guarantee the existence of bi-free conjugate variables (even if we are not in the tracially bi-partite setting) provided we perturb our variables by small multiplies of bi-free central limit distributions.
Although we state the result for a bi-free central limit system without covariance, one could just as easily perturb by a system of semicircular variables with any invertible covariance matrix and prove a similar result.
\begin{thm}
	\label{thm:conj-perturb-by-semis}
	Let $\vX, \vY$ be tuples of self-adjoint operators of length $n$ and $m$ respectively, let $(\{S_i\}^n_{i=1}, \{T_j\}^m_{j=1})$ be semicircular operators with variance one, and let $B_\ell$, $B_r$ be unital, self-adjoint subalgebras of a C$^*$-non-commutative probability space $(\A, \varphi)$ such that 
	\[
		(B_\ell  \langle \vX \rangle, B_r \langle \vY\rangle), \qquad \{(S_i, 1)\}^n_{i=1}, \qqand  \{(1, T_j)\}^m_{j=1}
	\]
	are bi-free and each pair contain no algebraic relations other than possibly elements of the left algebra commuting with elements of the right algebra.  If $P : L_2(\fA, \varphi) \to L_2((B_\ell \vee B_r)  \langle \vX + \sqrt{\epsilon} \vS, \vY + \sqrt{\epsilon} \vT \rangle, \varphi)$ is the orthogonal projection onto the codomain, then
	\[
		\xi = \J_\ell(X_1 + \sqrt{\epsilon} S_1 : (B_\ell  \langle \widehat{(\vX + \sqrt{\epsilon} \vS)}_1\rangle, B_r \langle  \vY + \sqrt{\epsilon} \vT  \rangle)) = \frac{1}{\sqrt{\epsilon}} P(S_1).
	\]
	Furthermore
	\[
		 \left\|\xi\right\|_2 \leq \frac{1}{\sqrt{\epsilon}}
	\]
	where the norms is computed in $L_2(\fA, \varphi)$.
\end{thm}

\begin{proof}
	Note
	\[
		\J_\ell(S_1 : (\bC\langle \hat{\vS}_1\rangle, \bC\langle \vT\rangle)) = S_1
	\]
	by Example \ref{exam:bi-free-conjugate-independence} and the free result \cite{V1998-2}*{Proposition 3.6}.
	From Lemma \ref{lem:deriv-conjugate-variable-scaling}, we have
	\begin{align*} 
		\eta &=J(\sqrt{\epsilon} S_1 : (\bC\langle \sqrt{\epsilon} \hat{\vS}_1\rangle, \bC\langle \sqrt{\epsilon}\vT\rangle)) = \frac{1}{\sqrt{\epsilon}} S_1.
	\end{align*}
	It then follows by Propositions \ref{prop:deriv-bi-free-affecting-conjugate-variables} and \ref{prop:deriv-sums-affecting-conjugate-variables} that
	\begin{align*}
		\xi &= P(\eta) = \frac{1}{\sqrt{\epsilon}} P(S_1),
	\end{align*}
	as desired.  The norm estimate then easily follow by inner product computations.
\end{proof}

\section{Relative Bi-Free Fisher Information}
\label{sec:Fisher}

We now extend the notion of Fisher information from \cite{V1998-2}*{Section 6} to the bi-free setting.  Due to the results of Section \ref{sec:Properties}, the results follow with nearly identical proofs.

\begin{defn}
	\label{defn:bi-fisher}
	Let $\vX, \vY$ be tuples of self-adjoint operators of length $n$ and $m$ respectively, and let $B_\ell$, $B_r$ be unital, self-adjoint subalgebras of a C$^*$-non-commutative probability space $(\fA, \varphi)$ such that $\vX, \vY, B_\ell$, and $B_r$ contain no algebraic relations other than the possibility that elements of $B_\ell \langle \vX\rangle$ commute with elements of $B_r \langle \vY\rangle$.

	For $i \in \{1, \ldots, n\}$ and $j \in \{1,\ldots, m\}$ let
	\[
		\xi_i = \J_\ell(X_i : (B_\ell\langle \hat{\vX}_i\rangle, B_r\langle \vY \rangle))
\qqand
		\eta_j = \J_r(Y_{j} : (B_\ell \langle \vX \rangle, B_r\langle \hat{\vY}_j\rangle))
	\]  
	provided these bi-free conjugate variables exist.
	The \emph{relative bi-free Fisher information of $\vX, \vY$ with respect to $(B_\ell, B_r)$} is
	\[
		\Phi^*(\vX \sqcup \vY : (B_\ell, B_r)) = \sum^n_{i=1} \left\|\xi_i\right\|_2^2 + \sum^m_{j=1} \left\|\eta_j\right\|_2^2
	\]
	if $\xi_1, \ldots, \xi_n, \eta_1, \ldots, \eta_m$ exist, and otherwise defined as $\infty$.

	If $B_\ell = B_r = \bC$, we call $\Phi^*(\vX \sqcup \vY : (B_\ell, B_r))$ \emph{the bi-free Fisher information of $\vX, \vY$} and denote it by
	\[
		\Phi^*(\vX \sqcup \vY)
	\]
	instead.
\end{defn}

\begin{exam}
	\label{exam:Fisher-two-semis?}
	Let $(S, T)$ be a self-adjoint bi-free central limit distribution with respect to a state $\varphi$ such that $\varphi(S^2) = \varphi(T^2) = 1$ and $\tau(ST) = \tau(TS) = c \in (-1,1)$.  By Example \ref{exam:conju-of-semis} 
	\[
		\J_\ell(S : (\bC, \bC\langle T\rangle)) = \frac{1}{1-c^2} (S - c T)  \qqand \J_r(T : (\bC \langle S \rangle, \bC)) = \frac{1}{1-c^2} (T - c S).
	\]
	Hence 
	\begin{align*}
		\Phi^*(S \sqcup T) &= \frac{1}{(1-c^2)^2} \left\|S - c T\right\|_2^2 + \frac{1}{(1-c^2)^2} \left\|T - c S\right\|_2^2 = \frac{2}{1-c^2}
	\end{align*}
	as 
	\[
	\varphi((S-cT)^2) = \varphi(S^2) - c \varphi(ST) - c \varphi(TS) + c^2 \varphi(T^2) = 1-c^2.
	\]
\end{exam}

\begin{exam}
	\label{exam:Fisher-bi-free-central}
	More generally, let $(\{S_k\}^{n}_{k=1}, \{S_k\}^{n+m}_{k=n+1})$ be a self-adjoint bi-free central limit distribution with respect to $\varphi$.  By \cite{V2014}*{Section 7} the joint distribution of $(\{S_k\}^{n}_{k=1}, \{S_k\}^{n+m}_{k=n+1})$ is completely determined by the matrix
	\[
		A = [a_{i,j}] = [\varphi(S_iS_j)] \in \M_{n+m}(\bR)_{\sa}.
	\]
	Furthermore, by \cite{V2014}*{Section 7}, $A$ is positive as we can represent this pair as left and right semicircular operators acting on a Fock space and thus $A = [\langle f_i, f_j\rangle]$ where $\{f_k\}^{n+m}_{k=1}$ are vectors in a Hilbert space.

	Suppose $A$ is invertible.  For $k \in \{1,\ldots, n\}$ let
	\[
		\xi_k = \J_\ell(S_k : (\bC \langle S_1, \ldots, S_{k-1}, S_{k+1}, \ldots, S_n\rangle, \bC \langle S_{n+1}, \ldots, S_{n+m}\rangle))
	\]
	and for $k \in \{n+1, \ldots, n+m\}$ let
	\[
		\xi_k = \J_r(S_k : (\bC \langle S_1, \ldots, S_n\rangle, \bC \langle S_{n+1}, \ldots, S_{k-1}, S_{k+1}, \ldots, S_{n+m}\rangle)).
	\]
	It is routine to verify using similar arguments to Example \ref{exam:conju-of-semis} that if $\{e_k\}^{n+m}_{k=1}$ denotes the standard basis of $\bR^{n+m}$, then
	\[
		\xi_k = b_{1,k} S_1 + \ldots + b_{n+m, k} S_{n+m}
	\]
	where
	\[
		A \begin{bmatrix}
			b_{1,k} \\ \vdots \\ b_{n+m, k}
		\end{bmatrix} = e_k.
	\]
	Therefore, if $B = [b_{i,j}] \in \M_{n+m}(\bR)$ then $AB = I_{n+m}$ so $B = A^{-1} \in \M_{n+m}(\bR)_{\sa}$.  Note as $A$ is self-adjoint that $B$ is self-adjoint.

	By Definition \ref{defn:bi-fisher}, we see that if $\Tr$ denotes the unnormalized trace on $\M_{n+m}(\bR)$ then
	\begin{align*}
		\Phi^*(S_1, \ldots, S_n \sqcup S_{n+1}, \ldots, S_{n+m}) &= \sum^{n+m}_{k=1} \left\|\xi_k\right\|^2_2 \\
		&= \sum^{n+m}_{k=1} \varphi\left( \left(\sum^{n+m}_{i=1} b_{i,k} S_i\right)\left(\sum^{n+m}_{j=1} b_{j,k} S_j\right) \right) \\
		&= \sum^{n+m}_{i,j,k =1}  b_{i,k} a_{i,j} b_{j,k} \\
		&= \Tr(B^*AB) = \Tr(B^*) = \Tr(A^{-1}).
	\end{align*}
	We will see later via  Example \ref{exam:entropy-bi-free-central} that if $A$ is not invertible, then the bi-free entropy is infinite and thus the bi-free Fisher information is infinite by Proposition \ref{prop:finite-fisher-implies-finite-entropy}.
\end{exam}

\begin{rem}
	\label{rem:remarks-about-fisher-info}
	We make the following observations.
	\begin{enumerate}
		\item First notice that if $m = 0$ and $B_r = \bC$ or $n = 0$, and $B_\ell = \bC$, then $\Phi^*(\vX \sqcup \vY : (B_\ell, B_r))$ is simply the relative free Fisher information of $\vX$ with respect to $B_\ell$ or of $\vY$ with respect to $B_r$ respectively. \label{part:bi-free-fisher-is-free-fisher-if-one-side-absent}
		\item \label{part:only-one-variable}
			As
			\begin{align*}
				\Phi^*( & \vX \sqcup \vY : (B_\ell, B_r))\ = \sum^n_{i=1} \Phi^*(X_i : (B_\ell\langle  \hat{\vX}_i\rangle, B_r \langle \vY \rangle))+ \sum^m_{j=1}  \Phi^*(Y_j : (B_\ell\langle \vX\rangle, B_r \langle \hat{\vY}_j \rangle)),
			\end{align*}
			many questions about the relative bi-free Fisher information reduce to the cases $(n,m) \in \{(1,0), (0, 1)\}$.
		\item 
			\label{part:scaling-fisher-info}
			Recall from Lemma \ref{lem:deriv-conjugate-variable-scaling} that for all $\lambda \in \bR \setminus \{0\}$
			\begin{align*}
				\J_\ell(\lambda X_i & : (B_\ell\langle\lambda \hat{\vX}_i\rangle, B_r \langle \lambda \vY \rangle)) = \frac{1}{\lambda} \J_\ell(  X_i : (B_\ell\langle  \hat{\vX}_i\rangle, B_r \langle  \vY \rangle)).
			\end{align*}
			As a similar result holds for the right bi-free conjugate variables, we see that
			\[
				\Phi^*(\lambda \vX \sqcup  \lambda\vY : (B_\ell, B_r)) = \frac{1}{\lambda^2} \Phi^*(\vX \sqcup \vY : (B_\ell, B_r)).
			\]
		\item 
			\label{part:increasing-algebra-in-fisher-info}
			Notice if $C_\ell \subseteq B_\ell$ and $C_r \subseteq B_r$ are unital, self-adjoint subalgebras, then by Lemma \ref{lem:deriv-conjugate-variable-projecting} the bi-free conjugate variables of
			\[
				(\vX \sqcup \vY : (C_\ell, C_r))
			\]
			are the projections of the bi-free conjugate variables of
			\[
				(\vX \sqcup \vY : (B_\ell, B_r))
			\]
			onto $L_2((C_\ell \vee C_r) \langle \vX, \vY\rangle, \varphi)$.  Therefore 
			\[
				\Phi^*(\vX \sqcup \vY : (C_\ell, C_r)) \leq \Phi^*(\vX \sqcup \vY : (B_\ell, B_r)).
			\]
		\item 
			\label{part:fisher-info-for-bi-free-things}
			Finally, if $(C_\ell, C_r)$ is a pair of unital, self-adjoint subalgebras of $\fA$ that is bi-free from
			\[
				\left( B_\ell \langle \vX\rangle, B_r \langle \vY\rangle\right)
			\]
			then it follows from Proposition \ref{prop:deriv-bi-free-affecting-conjugate-variables} that
			\[
				(\vX \sqcup \vY : (B_\ell, B_r)) \qqand (\vX \sqcup \vY : (B_\ell \vee C_\ell, B_r \vee C_r))
			\]
			have the same bi-free conjugate variables and thus
			\[
				\Phi^*(\vX \sqcup \vY : (B_\ell, B_r)) = \Phi^*(\vX \sqcup \vY : (B_\ell \vee C_\ell, B_r \vee C_r)).
			\]
	\end{enumerate}
\end{rem}

Furthermore, the bi-free Fisher information behaves well with respect to combining bi-free collections.

\begin{prop}
	\label{prop:fisher-info-with-bifree-things}
	Let $\vX, \vY, \vX', \vY'$ be tuples of self-adjoint operators of lengths $n$, $m$, $n'$, and $m'$ respectively, and let $B_\ell$, $B_r$, $C_\ell$, $C_r$ be unital, self-adjoint subalgebras of a C$^*$-non-commutative probability space $(\fA, \varphi)$ such that 
	\[
		(B_\ell  \langle \vX \rangle, B_r \langle \vY\rangle) \qqand (C_\ell \langle \vX' \rangle, C_r\langle  \vY'\rangle)
	\]
	are bi-free and the pairs have no algebraic relations other than possibly left operators commuting with right operators.  Then
	\begin{align*}
		\Phi^*&(\vX, \vX' \sqcup \vY,  \vY': (B_\ell \vee C_\ell, B_r \vee C_r))= \Phi^*(\vX   \sqcup \vY   : (B_\ell, B_r)) + \Phi^*( \vX' \sqcup  \vY' : (C_\ell, C_r)).
	\end{align*}
\end{prop}

\begin{proof}
	By Proposition \ref{prop:deriv-bi-free-affecting-conjugate-variables} 
	\begin{align*}
		\J_\ell(X_i : (B_\ell \langle \hat{\vX}_i \rangle, B_r\langle \vY \rangle)) = \J_\ell(X_i : ((B_\ell \vee C_\ell) \langle \hat{\vX}_i, \vX' \rangle, (B_r \vee C_r) \langle \vY, \vY' \rangle)).
	\end{align*}
	As a similar result holds for the right bi-free conjugate variables and for the $\vX'$s and $\vY'$s, the result easily follows.
\end{proof}

When pairs of operators are not bi-free, at least Proposition \ref{prop:fisher-info-with-bifree-things} holds upto an inequality.

\begin{prop}
	\label{prop:Fisher-supadditive}
	Let $\vX, \vY, \vX', \vY'$ be tuples of self-adjoint operators of lengths $n$, $m$, $n'$, and $m'$ respectively, and let $B_\ell$, $B_r$, $C_\ell$, $C_r$ be self-adjoint subalgebras of a C$^*$-non-commutative probability space $(\fA, \varphi)$ such that this collection has no algebraic relations other than possibly left operators commuting with right operators. 
	Then
	\begin{align*}
		\Phi^*&(\vX, \vX' \sqcup \vY, \vY' : (B_\ell \vee C_\ell, B_r \vee C_r)) \geq \Phi^*(\vX   \sqcup \vY  : (B_\ell, B_r)) + \Phi^*( \vX' \sqcup  \vY' : (C_\ell, C_r)).
	\end{align*}
\end{prop}

\begin{proof}
	By Remark \ref{rem:remarks-about-fisher-info} part (\ref{part:increasing-algebra-in-fisher-info})
	\begin{align*}
		 \Phi^*(X_i : ((B_\ell \vee C_\ell) \langle \hat{\vX}_i, \vX'\rangle, (B_r \vee C_r) \langle \vY, \vY'\rangle))  \geq \Phi^*(X_i : (B_\ell \langle \hat{\vX}_i \rangle, B_r \langle \vY \rangle)).
	\end{align*}
	As a similar result holds for the right bi-free conjugate variables and for the $\vX'$'s and $\vY'$'s, the result follows from Remark \ref{rem:remarks-about-fisher-info} part (\ref{part:only-one-variable}).
\end{proof}

Next we endeavour to obtain a bi-free analogue of the Stam Inequality.  To do so, we must first note the following.

\begin{lem}
	\label{lem:ortho-projections-orthogonal-bi-free}
	Let $\vX, \vY, \vX', \vY'$ be tuples of self-adjoint operators of length $n$, $m$, $n'$, and $m'$ respectively, and let $B_\ell$, $B_r$, $C_\ell$, $C_r$ be unital, self-adjoint subalgebras of a C$^*$-non-commutative probability space $(\fA, \varphi)$ such that 
	\[
		(B_\ell  \langle \vX \rangle, B_r \langle \vY\rangle) \qqand (C_\ell \langle \vX' \rangle, C_r\langle \vY'\rangle)
	\]
	are bi-free.  If
	\begin{align*}
		P_0 & : L_2(\A, \varphi) \to \bC1_{\A} \\
		P_1 & : L_2(\A, \varphi) \to L_2((B_\ell \vee B_r)\langle \vX, \vY \rangle, \varphi) \\
		P_2 & : L_2(\A, \varphi) \to L_2((C_\ell \vee C_r)\langle \vX', \vY' \rangle, \varphi)
	\end{align*}
	are the orthogonal projections onto their co-domains, then $P_1P_2 = P_2P_1 = P_0$.
\end{lem}

\begin{proof}
	First note that if
	\[
		Z \in (B_\ell \vee B_r)\langle \vX, \vY \rangle \qand Z' \in (C_\ell \vee C_r)\langle \vX', \vY' \rangle
	\]
	then bi-freeness implies
	\[
		\varphi(ZZ') = \varphi(Z'Z) = \varphi(Z) \varphi(Z').
	\]
	This can easily be seen via bi-non-crossing partitions as bi-freeness implies a cumulant of $ZZ'$ corresponding to a bi-non-crossing partition is non-zero if and only if it decomposes into a bi-non-crossing partition on $Z$ union a bi-non-crossing partition on $Z'$.

	As the above implies that
	\[
		L_2((B_\ell \vee B_r)\langle \vX, \vY \rangle, \varphi) \ominus L_2(\bC, \varphi) \qand L_2((C_\ell \vee C_r)\langle \vX', \vY' \rangle, \varphi)\ominus L_2(\bC, \varphi)
	\]
	are orthogonal subspaces by taking $L_2$-limits, the result follows.
\end{proof}

\begin{prop}[Bi-Free Stam Inequality]
	\label{prop:stam-inequality}
	Let $\vX, \vX'$ be $n$-tuples of self-adjoint operators, let $\vY, \vY'$ be $m$-tuples of self-adjoint operators, and let $B_\ell$, $B_r$, $C_\ell$, $C_r$ be unital, self-adjoint subalgebras of a C$^*$-non-commutative probability space $(\fA, \varphi)$ such that 
	\[
		(B_\ell  \langle \vX \rangle, B_r \langle \vY\rangle) \qqand (C_\ell \langle \vX' \rangle, C_r\langle \vY'\rangle)
	\]
	are bi-free  and the pairs have no algebraic relations other than possibly left operators commuting with right operators. Then
	\begin{align*}
		\left(\Phi^*( \vX+ \vX' \sqcup  \vY + \vY' : (B_\ell \vee C_\ell, B_r \vee C_r))   \right)^{-1}  \geq \left(\Phi^*(\vX   \sqcup \vY   : (B_\ell, B_r))\right)^{-1} + \left( \Phi^*( \vX' \sqcup  \vY' : (C_\ell, C_r)) \right)^{-1}.
	\end{align*}
\end{prop}

\begin{proof}
	If both of
	\[
		\Phi^*(\vX   \sqcup \vY  : (B_\ell, B_r)) \qand  \Phi^*( \vX' \sqcup  \vY' : (C_\ell, C_r)) 
	\]
	are infinite then the result is immediate.  If exactly one is infinite then the desired inequality is equivalent to
	\begin{align*}
		\Phi^*(\vX + \vX' \sqcup \vY + \vY' : (B_\ell \vee C_\ell, B_r \vee C_r)) \leq  \Phi^*(\vX   \sqcup \vY   : (B_\ell, B_r))
	\end{align*}
	(when $\Phi^*( \vX' \sqcup  \vY' : (C_\ell, C_r))  = \infty$) and thus easily follows from Proposition \ref{prop:deriv-sums-affecting-conjugate-variables} as a projection onto a subspace decreases the $L_2$-norm.  Thus we will assume that both relative bi-free Fisher informations are finite.

	Let $P_0, P_1,$ and $P_2$ be as in Lemma~\ref{lem:ortho-projections-orthogonal-bi-free}, and take $P_3$ to be the orthogonal projection onto the algebra generated by the sums of the variables:
	\begin{align*}
		P_0 & : L_2(\A, \varphi) \to \bC1_{\A} \\
		P_1 & : L_2(\A, \varphi) \to L_2((B_\ell \vee B_r)\langle \vX, \vY \rangle, \varphi) \\
		P_2 & : L_2(\A, \varphi) \to L_2((C_\ell \vee C_r)\langle \vX', \vY' \rangle, \varphi) \\
		P_3 & : L_2(\A, \varphi) \to L_2(((B_\ell \vee C_\ell) \vee (B_r  \vee C_r))\langle \vX + \vX', \vY + \vY' \rangle, \varphi).
	\end{align*}
	By Lemma \ref{lem:ortho-projections-orthogonal-bi-free}, $P_1P_2 = P_2P_1 = P_0$.

	For notational simplicity, let $X''_i = X_i + X'_i$ for all $i$, $Y''_j = Y_j + Y'_j$ for all $j$, and let
	\begin{align*}
		\xi_{1,i} &= \J_\ell(X_i : (B_\ell \langle \hat{\vX}_i\rangle, B_r\langle \vY \rangle)), \\
		\xi_{2,i} &= \J_\ell(X'_i : (C_\ell \langle \hat{\vX}'_i\rangle, C_r\langle \vY' \rangle)), \\
		\xi_{3,i} &= \J_\ell(X''_i : ((B_\ell\vee C_\ell) \langle \hat{\vX}''_i\rangle, (B_r\vee C_r)\langle \vY'' \rangle)), \\
		\eta_{1,j} &= \J_r(Y_j : (B_\ell \langle\vX\rangle, B_r\langle \hat{\vY}_j \rangle)), \\
		\eta_{2,j} &= \J_r(Y'_j : (C_\ell \langle \vX'\rangle, C_r\langle \hat{\vY}'_j \rangle)), \text{ and} \\
		\eta_{3,j} &= \J_r( Y''_j : ((B_\ell\vee C_\ell) \langle \vX''\rangle, (B_r\vee C_r)\langle \hat{\vY}''_j \rangle)).
	\end{align*}
	By Proposition \ref{prop:deriv-sums-affecting-conjugate-variables} we have that
	\[
		\xi_{3, i} = P_3(\xi_{1, i}) = P_3(\xi_{2, i}) \qqand \eta_{3, j} = P_3(\eta_{1, j}) = P_3(\eta_{2, j}).
	\]
	Since $P_1P_2 = P_2P_1 = P_0$, $\langle 1, \xi_{k,i}\rangle = 0 = \langle 1, \eta_{k,j}\rangle$, and $P_k(\xi_{k,i}) = \xi_{k,i}$ and $P_k(\eta_{k,j}) = \eta_{k,j}$ for all $k=1,2$, $1 \leq i \leq n$, and $1 \leq j \leq m$, we obtain that
	\[
		\langle \xi_{1, i}, \xi_{2,i} \rangle = 0 = \langle \eta_{1,j}, \eta_{2,j}\rangle
	\]
	for all $1 \leq i \leq n$, $1 \leq j \leq m$.

	Let $\zeta_{k,i} = \xi_{k,i} - \xi_{3,i} = (I - P_3)(\xi_{k,i})$ and $\theta_{k,j} = \eta_{k,j} - \eta_{3,j} = (I - P_3)(\eta_{k,j})$ for all $1 \leq i \leq n$, $1 \leq j \leq m$, and $k \in \{1,2\}$.  Clearly
	\[
		\xi_{k,i} = \xi_{3,i} + \zeta_{k,i}, \quad \xi_{3,i} \bot \zeta_{k,i}, \quad \eta_{k,j} = \eta_{3,j} + \theta_{k,j}, \qand \eta_{3,j} \bot \theta_{k,j}.
	\]
	Hence if for $k \in \{1,2, 3\}$ we define
	\[
		h_k = (\xi_{k,1}, \ldots, \xi_{k,n}, \eta_{k,1}, \ldots, \eta_{k,m}) \in (L_2(\A, \varphi))^{\oplus (n+m)}
	\]
	and for $k \in \{1,2\}$ we define
	\[
		f_k = (\zeta_{k,1}, \ldots, \zeta_{k,n}, \theta_{k,1}, \ldots, \theta_{k,m}) \in (L_2(\A, \varphi))^{\oplus (n+m)},
	\]
	then
	\[
		h_3 + f_1 = h_1, \quad h_3 + f_2 = h_2, \quad h_3 \bot f_1, \quad h_3 \bot f_2, \qand h_1 \bot h_2.
	\]
	Thus
	\[
		0 = \langle h_1, h_2 \rangle = \langle h_3, h_3\rangle + \langle f_1, f_2\rangle
	\]
	so that
	\begin{align*}
		\left\|h_3\right\|^4_2 &\leq \left\|f_1\right\|^2_2 \left\|f_2\right\|_2^2 \\
		&= \left( \left\|h_1\right\|^2_2 - \left\|h_3\right\|^2_2\right)\left( \left\|h_2\right\|^2_2 - \left\|h_3\right\|^2_2\right) \\
		&= \left\|h_1\right\|^2_2 \left\|h_2\right\|^2_2 - \left\|h_3\right\|^2_2 \left( \left\|h_1\right\|^2_2 + \left\|h_2\right\|^2_2\right) + \left\|h_3\right\|^4_2.
	\end{align*}
	This implies
	\[
		\left\|h_1\right\|^2_2 \left\|h_2\right\|^2 \geq \left\|h_3\right\|_2^2\left(\left\|h_1\right\|^2_2 + \left\|h_2\right\|^2_2\right).
	\]
	Hence
	\[
		\left( \left\|h_3\right\|^2_2\right)^{-1} \geq \left( \left\|h_2\right\|^2_2\right)^{-1} + \left( \left\|h_1\right\|^2_2\right)^{-1},
	\]
	which is the desired inequality.
\end{proof}

Next we note that the bi-free Fisher information behaves well with respect to specific transformations.
\begin{prop}
	\label{prop:fisher-information-unaffected-by-orthogonal-transform}
	Let $\vX, \vY$ be self-adjoint operators of length $n$ and $m$ respectively, and let $B_\ell$, $B_r$ be self-adjoint subalgebras of a C$^*$-non-commutative probability space $(\fA, \varphi)$ with no algebraic relations except for possibly left operators commuting with right operators.  Let $A = [a_{i,j}] \in \M_n(\bR)$ be an invertible matrix and for each $i \in \{1,\ldots, n\}$ let
	\[
		X'_i = \sum_{k=1}^n a_{i,k} X_k.
	\]
	Then for all $1 \leq k \leq n$,
	\begin{align*}
		 \J_\ell(X_k: (B_\ell \langle \hat{\vX}_k\rangle, B_r\langle \vY\rangle ))  = \sum_{i=1}^n a_{i,k} \J_\ell(X'_i : (B_\ell \langle \hat{\vX}'_i\rangle, B_r\langle \vY \rangle)).
	\end{align*}
	In particular, if $A$ is an orthogonal matrix then
	\[
		\Phi^*(\vX \sqcup \vY : (B_\ell, B_r)) = \Phi^*(\vX' \sqcup \vY : (B_\ell, B_r)).
	\]
	For general $A$, we have that
	\begin{align*}
		\left(\max\{\left\|A^{-1}\right\|, 1\}\right)^{-2} \Phi^*(\vX' \sqcup \vY : (B_\ell, B_r)) &\leq \Phi^*(\vX \sqcup \vY : (B_\ell, B_r)) \\
		& \leq \left(\max\{\left\|A\right\|, 1\}\right)^2 \Phi^*(\vX' \sqcup \vY : (B_\ell, B_r)).
	\end{align*}

	A similar result holds on the right.
\end{prop}

\begin{proof}
	As $A$ is an invertible matrix, we see that
	\[
		L_2((B_\ell \vee B_r)\langle \vX', \vY\rangle, \varphi) = L_2((B_\ell \vee B_r)\langle \vX, \vY\rangle, \varphi).
	\]
	The equation for the conjugate variables then follows by the linearity of the cumulants.  The remainder of the proof follows from easy $L_2$-norm computations.
\end{proof}

We note Proposition \ref{prop:fisher-information-unaffected-by-orthogonal-transform} only applies only to matrices acting on either just the left operators or just the right operators.
Due to the rigidity of the bi-free cumulants only accepting left operators in left entries and right operators in right entries (except for the final entry) it is unclear how such a transformation would affect the bi-free Fisher information.

Next we obtain a lower bound for the bi-free Fisher information based on the the variance of each operator.

\begin{prop}[Bi-Free Cramer-Rao Inequality]
	\label{prop:cramer-rao}
	Let $\vX, \vY$ be tuples of self-adjoint operators of length $n$ and $m$ respectively, and let $B_\ell$, $B_r$ be unital, self-adjoint subalgebras of a C$^*$-non-commutative probability space $(\fA, \varphi)$ with no algebraic relations except for possibly left operators commuting with right operators.  Then
	\[
		\Phi^*(\vX \sqcup \vY : (B_\ell, B_r)) \varphi\left(\sum^n_{i=1} X_i^2 + \sum^m_{j=1} Y_j^2 \right) \geq (n+m)^2.
	\]
	Moreover, equality holds if $\vX, \vY$ are centred semicircular distributions of the same variance and $\{(B_\ell, B_r)\} \cup \{(X_i, 1)\}^n_{i=1}\cup \{(1,Y_j)\}^m_{j=1}$ is bi-free.  The converse holds when $B_\ell = B_r = \bC$.
\end{prop}

\begin{proof}
	Let
	\begin{align*}
		B_{\ell, i} &= (B_\ell \langle \hat{\vX}_i \rangle, B_r\langle \vY\rangle)  \text{ and}\\
		B_{r, j} &= (B_\ell \langle \vX \rangle, B_r\langle \hat{\vY}_j\rangle).
	\end{align*}
	Then
	\begin{align*}
		&  \Phi^*(\vX \sqcup \vY : (B_\ell, B_r)) \varphi\left(\sum^n_{i=1} X_i^2 + \sum^m_{j=1} Y_j^2 \right)  \\
		& = \left(\sum^n_{i=1} \left\|\J_\ell(X_i : B_{\ell, i})\right\|^2_2 + \sum^m_{j=1} \left\|\J_r(Y_j : B_{r, j})\right\|^2_2 \right) \left(\sum^n_{i=1} \left\|X_i\right\|_2^2 + \sum^m_{j=1} \left\|Y_j\right\|_2^2 \right)  \\
		& \geq \left|\sum^n_{i=1} \varphi(X_i \J_\ell(X_i : B_{\ell, i})) + \sum^m_{j=1} \varphi(Y_j \J_r(Y_j : B_{r, j}))  \right|^2 = (n+m)^2
	\end{align*}
	by the Cauchy-Schwarz inequality.  Moreover, equality holds if and only if there exists a $\lambda \in \bR \setminus \{0\}$ such that
	\[
		\J_\ell(X_i : B_{\ell, i}) = \lambda X_i \qqand \J_r(Y_j : B_{r, j}) = \lambda Y_j
	\]
	for all $1 \leq i \leq n$ and $1 \leq j \leq m$.

	Suppose $\vX, \vY$ are centred semicircular distributions of the same variance, say $\lambda^{-1}$, and $\{(B_\ell, B_r)\}\cup \{(X_i, 1)\}^n_{i=1}\cup \{(1,Y_j)\}^m_{j=1}$ is bi-free. By Proposition \ref{prop:deriv-bi-free-affecting-conjugate-variables} and Lemma \ref{lem:deriv-conjugate-variable-scaling},
	\begin{align*}
		\J(X_i : B_{\ell, i}) &= \J_\ell(X_i : (\bC, \bC)) = \lambda^{\frac{1}{2}} \J_\ell(\lambda^{\frac{1}{2}} X_i : (\bC, \bC)) =  \lambda^{\frac{1}{2}}\left(\lambda^{\frac{1}{2}} X_i   \right) = \lambda X_i.
	\end{align*}
	Similarly $\J(Y_j : B_{r, j}) = \lambda Y_j$ so equality occurs in this case as desired.

	To see the converse if $B_\ell = B_r = \bC$, notice that if 
	\[
		\J(X_i : B_{\ell, i}) = \lambda X_i \qqand \J(Y_j : B_{r, j}) = \lambda Y_j
	\]
	for all $1 \leq i \leq n$ and $1 \leq j \leq m$, then the definition of the conjugate variables gives relations on the bi-free cumulants of $(\{X_i\}_{i=1}^n, \{Y_j\}^m_{j=1})$ which imply $\vX, \vY$ are centred semicircular distributions of the same variance $\lambda^{-1}$ and $ \{(X_i, 1)\}^n_{i=1}\cup \{(1,Y_j)\}^m_{j=1}$ is bi-free.
\end{proof}

\begin{rem}
	The reason that the converse of the last statement in Proposition~\ref{prop:cramer-rao} may fail when $B_\ell$ and $B_r$ are not both $\bC$ comes down to the fact that knowing the behaviour of conjugate variable does not tell us about bi-free cumulants with elements of $B_\ell$ or $B_r$ in the final entry.  In the free setting this difficulty is absent due to the traciality of the state.
\end{rem}

In order to perform many computations with the bi-free Fisher information, we require an understanding of some analytical aspects.  Thus we will prove the following.

\begin{prop} 
	\label{prop:fisher-strong-convergence-bounds}
	Let $\vX, \vY$ be tuples of self-adjoint operators of length $n$ and $m$ respectively, and let $B_\ell$, $B_r$ be unital, self-adjoint subalgebras of a C$^*$-non-commutative probability space $(\fA, \varphi)$ with no algebraic relations except for possibly left operators commuting with right operators.
	Suppose further that for each $k \in \bN$ that $\vX^{(k)}, \vY^{(k)}$ are tuples of self-adjoint elements in $\fA$ of length $n$ and $m$ respectively such that
	\begin{align*}
		& \limsup_{k \to \infty} \left\|X^{(k)}_i\right\| < \infty, \\
		& \limsup_{k \to \infty} \left\|Y^{(k)}_j\right\| < \infty, \\
		& s\text{-}\lim_{k \to \infty} X^{(k)}_i = X_i, \text{ and} \\
		& s\text{-}\lim_{k \to \infty} Y^{(k)}_j = Y_j
	\end{align*}
	for all $1 \leq i \leq n$ and $1 \leq j \leq m$ (where the strong limit is computed as bounded linear maps on $L_2(\fA, \varphi)$).  Then
	\begin{align*}
  \liminf_{k \to \infty}  \Phi^*\left(\vX^{(k)} \sqcup \vY^{(k)} : (B_\ell, B_r)\right)  \geq  \Phi^*(\vX \sqcup \vY : (B_\ell, B_r)) 
	\end{align*}
\end{prop}

The proof of Proposition \ref{prop:fisher-strong-convergence-bounds} first requires the following.
\begin{lem}
	\label{lem:fisher-strong-convergence-bounds}
	Under the assumptions of Proposition \ref{prop:fisher-strong-convergence-bounds} along with the additional assumptions that
	\[
		\xi_k = \J_\ell\left(X_1^{(k)} : (B_\ell \langle \hat{\vX}^{(k)}_1 \rangle, B_r\langle \vY^{(k)}\rangle)\right)
	\]
	exist and are bounded in $L_2$-norm by some constant $K > 0$, it follows that
	\[
		\xi = \J_\ell(X_1 : (B_\ell \langle \hat{\vX}_1\rangle, B_r\langle \vY\rangle))
	\]
	exists and is equal to
	\[
		w\text{-}\lim_{k \to \infty} P\left(\xi_k\right)
	\]
	where $P$ is the orthogonal projection of $L_2(\fA, \varphi)$ onto $L_2((B_\ell \vee B_r)\langle \vX, \vY\rangle, \varphi)$.

	If, in addition,
	\[
		\limsup_{k \to \infty} \left\|\xi_k\right\|_2 \leq \left\| \xi \right\|_2
	\]
	then
	\[
		\lim_{k \to \infty}  \left\|\xi_k- \xi\right\|_2 =0
	\]

	The same holds with $X_1$ replaced with $X_i$, and a similar result holds for the right.
\end{lem}

\begin{proof}
	First, as $(\xi_k)_{k\geq 1}$ is bounded in the $L_2$-norm, $(\xi_k)_{k\geq 1}$ has a subnet that converges in the weak topology.  If $\zeta$ is the limit of this net, we will show that $P(\zeta) = \xi$.  From this it follows that $(P(\xi_k))_{k\geq 1}$ converges in the weak topology to $\xi$ due to the uniqueness of the bi-free conjugate variables.
	Thus, for the purposes of that which follows, we will assume that $(\xi_k)_{k\geq 1}$ converges to $\zeta$ in the weak topology.

	For $q \geq 0$ fix a $\chi : \{1,\ldots, q+1\} \to \{\ell,r\}$ such that $\chi(q+1) = \ell$ and choose $Z_1, \ldots, Z_q \in \A$ such that $Z_p \in B_\ell \cup \{\vX\}$ if $\chi(p) = \ell$ and $Z_p \in B_r \cup \{\vY\}$ if $\chi(p) = r$.  For each $k \in \bN$, let
	\[
		Z^{(k)}_p = \begin{cases}
			Z_p & \text{if } Z_k \in B_\ell \cup B_r \\
			X^{(k)}_i & \text{if }Z_k = X_i \\
			Y^{(k)}_j & \text{if }Z_k = Y_j
		\end{cases}.
	\]
	Hence
	\begin{align}
		\label{eqn:fisherconvergelemmalineone}\kappa_\chi(Z_1, \ldots, Z_q, P(\zeta))
		& = \kappa_\chi(Z_1, \ldots, Z_q, \zeta) \\
		\label{eqn:fisherconvergelemmalinetwo}&= \lim_{k \to \infty} \kappa_\chi\left(Z^{(k)}_1, \ldots, Z^{(k)}_q, \xi_k\right)
	\end{align}
	where (\ref{eqn:fisherconvergelemmalineone}) follows from the fact that $Z_1, \ldots, Z_q \in P(L_2(\fA, \varphi))$ and (\ref{eqn:fisherconvergelemmalinetwo}) follows from the fact that the cumulants are sums of moments, we have weak convergence of $\xi_k$ to $\eta$, the $\xi_k$ are bounded in $L_2(\fA, \varphi)$, and strong convergence of non-commutative polynomials in $\vX^{(k)}, \vY^{(k)}, B_\ell, B_r$ to the corresponding polynomials in $\vX, \vY, B_\ell, B_r$ by the assumptions of Proposition \ref{prop:fisher-strong-convergence-bounds}.  Therefore, as 
	\[
		\kappa_\chi\left(Z^{(k)}_1, \ldots, Z^{(k)}_q, \xi_k\right)
	\]
	is either $0$ or $1$, we see that $\kappa_\chi(Z_1, \ldots, Z_q, P(\eta))$ obtains the appropriate values to be $\xi$.  Thus the first claim is proved.

	By the first claim we obtain that
	\[
		\liminf_{k \to \infty} \left\|\xi_k\right\|_2 \geq \left\|\xi\right\|_2.
	\]
	Thus the additional assumption
	\[
		\limsup_{k \to \infty} \left\|\xi_k\right\|_2 \leq \left\| \xi \right\|_2
	\]
	implies that
	\[
		\lim_{k \to \infty} \left\|\xi_k\right\|_2 = \left\|\xi\right\|_2.
	\]
	This together with the fact that $\xi$ is the weak limit of $(\xi_k)_{k\geq 1}$ implies that
	\[
		\lim_{k \to \infty}  \left\|\xi_k- \xi\right\|_2 =0
	\]
	as desired.
\end{proof}

\begin{proof}[Proof of Proposition \ref{prop:fisher-strong-convergence-bounds}]
	If 
	\[
		\liminf_{k \to \infty}  \Phi^*\left(\vX^{(k)} \sqcup \vY^{(k)} : (B_\ell, B_r)\right)   = \infty
	\]
	there is nothing to prove.  Otherwise, we may pass to subsequences to assume that
	\[
		\limsup_{k \to \infty}  \Phi^*\left(\vX^{(k)} \sqcup \vY^{(k)} : (B_\ell, B_r)\right)  < \infty.
	\]
	Combining part (\ref{part:only-one-variable}) of Remark \ref{rem:remarks-about-fisher-info} with Proposition \ref{prop:fisher-strong-convergence-bounds} then implies the result.
\end{proof}

The convergence properties obtained in Proposition \ref{prop:fisher-strong-convergence-bounds} allows for many analytical results pertaining to the bi-free Fisher information.

\begin{cor}
	\label{cor:fisher-limits-sum-tending-to-zero}
	Let $\vX, \vY$ be tuples of self-adjoint operators of length $n$ and $m$ respectively, and let $B_\ell$, $B_r$ be unital, self-adjoint subalgebras of a C$^*$-non-commutative probability space $(\fA, \varphi)$.
	Suppose further that for each $k \in \bN$ that $\vX^{(k)}, \vY^{(k)}$ are tuples of self-adjoint elements of length $n$ and $m$ respectively, and $C_\ell, C_r$ are unital, self-adjoint subalgebras of $\fA$ such that
	\[
		(B_\ell \langle \vX\rangle, B_r \langle \vY\rangle) \qqand \left(C_\ell \left\langle \vX^{(k)}\right\rangle, C_r \left\langle \vY^{(k)}\right\rangle\right)
	\]
	are bi-free, there are no algebraic relations other than possibly left operators commuting with right operators, and
	\[
		\lim_{k \to \infty} \left\|X^{(k)}_i\right\| = \lim_{k \to \infty} \left\|Y_j^{(k)}\right\| = 0
	\]
	for all $1 \leq i \leq n$ and $1 \leq j \leq m$.  Then
	\begin{align*}
		\lim_{k \to \infty} \Phi^*\left( \vX + \vX^{(k)}\sqcup \vY + \vY^{(k)} : (B_\ell \vee C_\ell, B_r \vee C_r)\right)=  \Phi^*(\vX \sqcup \vY : (B_\ell, B_r)).
	\end{align*}
	Furthermore, if $C_\ell =  C_r = \bC$, and
	\[
		\Phi^*(\vX \sqcup \vY : (B_\ell, B_r)) < \infty,
	\]
	then 
	\[
		\J_\ell\left(  X^{(k)}_i  :   \left( B_\ell \left\langle\widehat{(\vX + \vX^{(k)})}_i \right\rangle, B_r\left\langle \vY + \vY^{(k)} \right\rangle  \right)\right)
	\] 
	tends to 
	\[
		\J_\ell(X_i : (B_\ell\langle \hat{\vX}_i \rangle, B_r \langle \vY\rangle))
	\]
	in $L_2$-norm.  A similar result holds for right bi-free conjugate variables.
\end{cor}

\begin{proof}
	Proposition \ref{prop:fisher-strong-convergence-bounds} and part (\ref{part:fisher-info-for-bi-free-things}) of Remark \ref{rem:remarks-about-fisher-info} implies that
	\begin{align*}
		 \liminf_{k \to \infty} \Phi^*\left( \vX + \vX^{(k)} \sqcup  \vY + \vY^{(k)} : (B_\ell \vee C_\ell, B_r \vee C_r)\right)
		& \geq  \Phi^*(\vX \sqcup \vY : (B_\ell \vee C_\ell, B_r \vee C_r)) \\
		&=  \Phi^*(\vX \sqcup \vY : (B_\ell, B_r)).
	\end{align*}
	However, the bi-free Stam inequality (Proposition \ref{prop:stam-inequality}) implies that
	\begin{align*}
		 \Phi^*\left( \vX + \vX^{(k)} \sqcup  \vY + \vY^{(k)} : (B_\ell \vee C_\ell, B_r \vee C_r)\right)  \leq  \Phi^*(\vX \sqcup \vY : (B_\ell, B_r))
	\end{align*}
	for all $k$.  Hence the first claim follows.  The second claim now trivially follows from Lemma~\ref{lem:fisher-strong-convergence-bounds}.
\end{proof}

\begin{thm}
	\label{thm:fisher-info-after-perturbing-by-semis}
	Let $\vX, \vY$ be tuples of self-adjoint operators of lengths $n$ and $m$ respectively, and let $B_\ell$, $B_r$ be unital, self-adjoint subalgebras of a C$^*$-non-commutative probability space $(\fA, \varphi)$.  Suppose further that $S_1, \ldots, S_n, T_1, \ldots, T_m$ are $(0, 1)$ semicircular variables in $\fA$ such that 
	\[
		(B_\ell \langle \vX \rangle, B_r\langle \vY\rangle) \cup \{(S_i, 1)\}^n_{i=1}\cup \{(1, T_j)\}^m_{j=1}
	\]
	are bi-free and there are no algebraic relations other than possibly left operators commuting with right operators.  Then the map
	\[
		h : [0, \infty) \ni t \mapsto  \Phi^*(\vX + \sqrt{t} \vS \sqcup \vY + \sqrt{t} \vT : (B_\ell, B_r))
	\]
	is decreasing, right continuous, and
	\[
		\frac{(n+m)^2}{C^2 + (n+m)t} \leq  h(t) \leq \frac{n+m}{t}
	\]
	where
	\[
		C^2 = \varphi \left(\sum^n_{i=1} X_i^2 + \sum^m_{j=1} Y_j^2 \right).
	\]
	Moreover $h(t) = \frac{(n+m)^2}{C^2 + (n+m)t}$ for all $t$ if $\vX, \vY$ are centred semicircular distributions of the same variance and $\{(B_\ell, B_r)\} \cup \{(X_i, 1)\}^n_{i=1}\cup \{(1,Y_j)\}^m_{j=1}$ are bi-free.
	Finally, if $B_\ell = B_r = \bC$ and $h(t) = \frac{(n+m)^2}{C^2 + (n+m)t}$ for all $t$, then $\vX, \vY$ are centred semicircular distributions of the same variance such that $\{(X_i, 1)\}^n_{i=1}\cup \{(1,Y_j)\}^m_{j=1}$ is bi-free.
\end{thm}

\begin{proof}
	Let $S'_1, \ldots, S'_n, T'_1, \ldots, T'_m$ be $(0, 1)$ semicircular variables in $\fA$ (or a larger C$^*$-non-commutative probability space) such that 
	\[
		(B_\ell \langle \vX \rangle, B_r\langle \vY\rangle) \cup \{(S_i, 1)\}^n_{i=1}\cup \{(1, T_j)\}^m_{j=1}\cup \{(S'_i, 1)\}^n_{i=1}\cup \{(1, T'_j)\}^m_{j=1}
	\]
	are bi-free.  Then for all $\epsilon > 0$ we have that
	\begin{align*}
	 \Phi^*( \vX + \sqrt{t+\epsilon} \vS \sqcup \vY + \sqrt{t+\epsilon} \vY  : (B_\ell, B_r)) =\Phi^*( \vX + \sqrt{t} \vS + \sqrt{\epsilon} \vS'  \sqcup \vY + \sqrt{t} \vT + \sqrt{\epsilon} \vT' : (B_\ell, B_r)).
	\end{align*}
	It follows that the desired map is right continuous by Corollary \ref{cor:fisher-limits-sum-tending-to-zero} and decreasing from the bi-free Stam inequality (Proposition \ref{prop:stam-inequality}).  The lower bound follows from the bi-free Cramer-Rao inequality (Proposition \ref{prop:cramer-rao}) as
	\[
		\varphi\left(\left(X_i + \sqrt{t} S_i\right)^2\right) = \varphi(X_i^2) + t \qqand \varphi\left(\left(Y_j + \sqrt{t} T_j\right)^2\right) = \varphi(Y_i^2) + t
	\]
	whereas the upper bound follows from the bi-free Stam inequality (Proposition \ref{prop:stam-inequality}), which implies
	\begin{align*}
		  \Phi^*(\vX + \sqrt{t} \vS \sqcup \vY + \sqrt{t} \vT : (B_\ell, B_r))\leq  \Phi^*( \sqrt{t}\vS \sqcup  \sqrt{t} \vT : (B_\ell, B_r)) = \frac{n+m}{t}.
	\end{align*}
	The final claims follow from the equality portion of the bi-free Cramer-Rao inequality (Proposition \ref{prop:cramer-rao}) together with the fact that $\{(X_i + \sqrt{t} S_i, 1)\}^n_{i=1} \cup \{(1, Y_j + \sqrt{t} T_j)\}^m_{j=1}$ are bi-free centred semicircular distributions of the same variance if and only if $\{(X_i, 1)\}^n_{i=1} \cup \{(1, Y_j)\}^m_{j=1}$ are bi-free centred semicircular distributions of the same variance. This may be seen through examination of bi-free cumulants using the fact that 
	\[
		(B_\ell \langle \vX \rangle, B_r\langle \vY\rangle) \cup \{(S_i, 1)\}^n_{i=1}\cup \{(1, T_j)\}^m_{j=1}\cup \{(S'_i, 1)\}^n_{i=1}\cup \{(1, T'_j)\}^m_{j=1}
	\]
	are bi-free.
\end{proof}

\section{Non-Microstate Bi-Free Entropy}
\label{sec:Entropy}

In this section, we introduce the non-microstate bi-free entropy as follows.

\begin{defn}
	\label{defn:entropy}
	Let $\vX, \vY$ be tuples of self-adjoint operators of length $n$ and $m$ respectively, and let $B_\ell$, $B_r$ be unital, self-adjoint subalgebras of a C$^*$-non-commutative probability space $(\fA, \varphi)$.
	The \emph{relative bi-free entropy of $(\vX, \vY)$ with respect to $(B_\ell, B_r)$} is defined to be
	\begin{align*}
		\chi^* & (\vX \sqcup \vY : (B_\ell, B_r)) = \frac{n+m}{2} \log(2\pi e) + \frac{1}{2} \int^\infty_0 \left(\frac{n+m}{1+t} - \Phi^*( \vX + \sqrt{t} \vS \sqcup \vY + \sqrt{t} \vT : (B_\ell, B_r) ) \right) \, dt
	\end{align*}
	where $S_1, \ldots, S_n, T_1, \ldots, T_m$ are self-adjoint operators in (a larger) $\fA$ that have centred semicircular distributions with variance 1 such that 
	\[
		(B_\ell \langle \vX \rangle, B_r\langle \vY\rangle) \cup \{(S_i, 1)\}^n_{i=1}\cup \{(1, T_j)\}^m_{j=1}
	\]
	are bi-free.  

	In the case that $B_\ell = B_r = \bC$, the relative bi-free entropy of $\vX, \vY$ with respect to $(B_\ell, B_r)$ is called the \emph{non-microstate bi-free entropy of $(\vX, \vY)$} and is denoted $\chi^* (\vX \sqcup \vY)$.
\end{defn}

\begin{rem}
	We note that we have used a specific bi-free Brownian motion in Definition \ref{defn:entropy}, namely the one defined by completely independent bi-free central limit distributions.
	This appears to be the optimal choice as this choice of bi-free central limit distribution has the maximal microstate bi-free entropy among all bi-free central limit distributions (see \cite{CS2017}) and minimizes the inequality in the bi-free Cramer-Rao inequality (Proposition \ref{prop:cramer-rao}).
	We note other non-microstate bi-free entropies are possible by selecting different bi-free Brownian motions. 
\end{rem}

\begin{rem}
	By part (\ref{part:bi-free-fisher-is-free-fisher-if-one-side-absent}) of Remark \ref{rem:remarks-about-fisher-info}, it is easy to see that if $m = 0$ and $B_r = \bC$ then $\chi^* (\vX \sqcup \vY : (B_\ell, B_r))$ is the non-microstate free entropy of $\vX$ with respect to $B_\ell$, while if $n =0 $ and $B_ \ell = \bC$ then $\chi^* (\vX \sqcup \vY : (B_\ell, B_r))$ is the non-microstate free entropy of $\vY$ with respect to $B_\ell$.

	In addition, by Remark \ref{rem:conjugate-variables-to-free-conjugate}, it is elementary to see that
	\begin{align*}
		\Phi^*(\vX \sqcup \vY : (B_\ell, B_r))  \geq \Phi^*(\vX  : B_\ell) + \Phi^*(  \vY : B_r)
	\end{align*}
	for any $\vX, \vY$ and thus
	\[
		\chi^* (\vX \sqcup \vY : (B_\ell, B_r)) \leq \chi^* (\vX  : B_\ell) + \chi^* ( \vY : B_r) < \infty.
	\]
\end{rem}

\begin{exam}
	\label{exam:entropy-bi-free-central}
	Let $(\{S_k\}^{n}_{k=1}, \{S_k\}^{n+m}_{k=n+1})$ be a centred, self-adjoint bi-free central limit distribution with respect to a state $\varphi$. Recall the joint distribution of these operators is completely determined by the matrix
	\[
		A = [a_{i,j}] = [\varphi(S_iS_j)] \in \M_{n+m}(\bR)_{\sa}
	\]
	which is positive.

	Let $(\{T_k\}^{n}_{k=1}, \{T_k\}^{n+m}_{k=n+1})$ be a centred, bi-free central limit distribution with variance one and covariance zero that are bi-free from $(\{S_k\}^{n}_{k=1}, \{S_k\}^{n+m}_{k=n+1})$.  For each $t \in (0, \infty)$ let
	\[
		S_k(t) = S_k + \sqrt{t} T_k.
	\]
	Hence $(\{S_k(t)\}^{n}_{k=1}, \{S_k(t)\}^{n+m}_{k=n+1})$ is a centred, self-adjoint bi-free central limit distribution with covariance matrix
	\[
		A_t = [\varphi(S_i(t)S_j(t))] = t I_{n+m} + A.
	\]
	Therefore, since $A_t$ is invertible for all $t \in (0, \infty)$ as $A \geq 0$, we obtain from Example \ref{exam:Fisher-bi-free-central} that
	\[
		\Phi^*(S_1(t), \ldots, S_n(t) \sqcup S_{n+1}(t), \ldots, S_{n+m}(t)) = \Tr((t I_{n+m} + A)^{-1}).
	\]

	As $A$ is a self-adjoint matrix, there exists a unitary matrix $U$ and a diagonal matrix $D = \diag(\lambda_1, \ldots, \lambda_{n+m})$ such that $A = U^*DU$.  Hence it is easy to see that
	\[
		\Phi^*(S_1(t), \ldots, S_n(t) \sqcup S_{n+1}(t), \ldots, S_{n+m}(t)) = \Tr((t I_{n+m} + D)^{-1}) = \sum^{n+m}_{k=1} \frac{1}{t + \lambda_k}.
	\]
	Therefore, as $\prod^{n+m}_{k=1} \lambda_k = \det(D) = \det(A)$, we see that
	\begin{align*}
		\chi^* (S_1, \ldots, S_n \sqcup S_{n+1}, \ldots, S_{n+m}) & = \frac{n+m}{2}\log(2\pi e) + \frac{1}{2} \int^\infty_0 \frac{n+m}{1+t} - \sum^{n+m}_{k=1} \frac{1}{t + \lambda_k} \, dt \\
		&= \frac{n+m}{2} \log(2 \pi e) + \frac{1}{2} \left.\left(\log\left( \frac{(1+t)^{n+m}}{\prod^{n+m}_{k=1} (t + \lambda_k)}\right) \right)\right|^\infty_{t=0}\\
		&= \frac{n+m}{2} \log(2 \pi e) + \frac{1}{2}\log\left(\prod^{n+m}_{k=1} \lambda_k\right) \\
		&= \frac{n+m}{2} \log(2 \pi e) + \frac{1}{2}\log\left(\det(A)\right).
	\end{align*}
	Note this agrees with the microstate bi-free entropy of $(\{S_k\}^{n}_{k=1}, \{S_k\}^{n+m}_{k=n+1})$ obtained in \cite{CS2017} and that $\frac{n+m}{2} \log(2 \pi e)$ is $n+m$ times the free entropy of a single semicircular operator with variance one.
\end{exam}

To understand the non-microstate bi-free entropy, we first demonstrate an upper bound.

\begin{prop}
\label{prop:upper-bound-non-microstate-entropy-based-on-L2-norm}
	Let $\vX, \vY$ be tuples of self-adjoint operators of length $n$ and $m$ respectively, and let $B_\ell$, $B_r$ be unital, self-adjoint subalgebras of a C$^*$-non-commutative probability space $(\A, \varphi)$ with no algebraic relations other than possibly the commutation of left and right operators.  If 
	\[
		C^2 = \varphi\left(\sum^n_{i=1} X_i^2 + \sum^m_{j=1} Y_j^2 \right)
	\]
	then
	\[
		\chi^*(\vX \sqcup \vY : (B_\ell, B_r)) \leq \frac{n+m}{2} \log\left( \frac{2 \pi e }{n+m} C^2\right).
	\]
	Furthermore equality holds if $\vX, \vY$ are semicircular operators of the same variance such that $\{(X_i, 1)\}^n_{i=1} \cup \{(1, Y_j)\}^m_{j=1}$ are bi-free and, if $B_\ell = B_r = \bC$, the converse holds.
\end{prop}

\begin{proof}
	By Theorem \ref{thm:fisher-info-after-perturbing-by-semis}
	\[
		\Phi^*(\vX + \sqrt{t}\vS \sqcup \vY + \sqrt{t}\vT : (B_\ell, B_r)) \geq \frac{(n+m)^2}{C^2 + (n+m)t}.
	\]
	Hence
	\begin{align*}
		\chi^*(\vX \sqcup \vY : (B_\ell, B_r))& \leq \frac{n+m}{2} \left(\log(2\pi e) + \int^\infty_0 \frac{1}{1+t} - \frac{1}{t + (n+m)^{-1} C^2} \, dt \right) \\
		&=  \frac{n+m}{2} \left(\log(2\pi e) + \left.\left(\log\left(\frac{1+t}{t + (n+m)^{-1} C^2}\right) \right)\right|^\infty_{t=0} \right)\\
		&=  \frac{n+m}{2} \left(\log(2\pi e) - \log\left(\frac{1}{(n+m)^{-1} C^2}\right) \right) \\
		&= \frac{n+m}{2} \log\left( \frac{2 \pi e }{n+m} C^2\right).  
	\end{align*} 
	As equality holds if and only if the equality from Theorem \ref{thm:fisher-info-after-perturbing-by-semis} holds for almost every $t > 0$, the final claims follow as Theorem \ref{thm:fisher-info-after-perturbing-by-semis} specifies when the equality holds.
\end{proof}

Several other properties of the non-microstate bi-free entropy easily follow from our knowledge of bi-free Fisher information.

\begin{prop}
\label{prop:properties-of-bi-free-entropy}
	Let $\vX, \vY, \vX', \vY'$ be tuples of self-adjoint operators of lengths $n$, $m$, $n'$, and $m'$ respectively, and let $B_\ell$, $B_r$, $C_\ell$, $C_r$ be unital, self-adjoint subalgebras of a C$^*$-non-commutative probability space $(\fA, \varphi)$ with no algebraic relations other than possibly left and right operators commuting.  
	\begin{enumerate}
		\item We have
			\begin{align*}
				 \chi^*(\vX, \vX' \sqcup \vY, \vY' : (B_\ell\vee C_\ell, B_r \vee C_r)) \leq\chi^*(\vX \sqcup \vY : (B_\ell, B_r)) + \chi^*(\vX' \sqcup \vY' : (C_\ell, C_r)).
			\end{align*}
		\item If
			\[
				(B_\ell\langle \vX \rangle, B_r\langle \vY \rangle ) \qqand (C_\ell\langle \vX' \rangle, C_r\langle \vY' \rangle )
			\]
			are bi-free, then the inequality in part (1) is an equality.
		\item If $C_\ell \subseteq B_\ell$ and $C_r \subseteq B_r$, then 
			\[
				\chi^*(\vX \sqcup \vY : (B_\ell, B_r))  \leq \chi^*(\vX \sqcup \vY : (C_\ell, C_r)) .
			\]
		\item If
			\[
				(B_\ell\langle \vX \rangle, B_r\langle \vY \rangle ) \qqand (C_\ell , C_r  )
			\]
			are bi-free,  then
			\[
				\chi^*(\vX \sqcup \vY : (B_\ell, B_r)) = \chi^*(\vX \sqcup \vY : (B_\ell  \vee C_\ell, B_r \vee C_r)).
			\]
	\end{enumerate}
\end{prop}

\begin{proof}
	Part (1) follows from Proposition \ref{prop:Fisher-supadditive}, part (2) follows from Proposition \ref{prop:fisher-info-with-bifree-things}, part (3) follows from part (\ref{part:increasing-algebra-in-fisher-info}) of Remark \ref{rem:remarks-about-fisher-info}, and part (4) follows from part (\ref{part:fisher-info-for-bi-free-things}) of Remark \ref{rem:remarks-about-fisher-info}.
\end{proof}

Furthermore, the non-microstate bi-free entropy behaves well with respect to limits.

\begin{prop}
	Let $\vX, \vY$ be tuples of self-adjoint operators of lengths $n$ and $m$ respectively, and let $B_\ell$, $B_r$ be unital, self-adjoint subalgebras of a C$^*$-non-commutative probability space $(\fA, \varphi)$ with no algebraic relations other than possibly the commutation of left and right operators.
	Suppose further that for each $k \in \bN$ that $\vX^{(k)}, \vY^{(k)}$ are tuples of self-adjoint elements in $\fA$ of lengths $n$ and $m$ respectively such that
	\begin{align*}
		& \limsup_{k \to \infty} \left\|X^{(k)}_i\right\| < \infty, \\
		& \limsup_{k \to \infty} \left\|Y^{(k)}_j\right\| < \infty, \\
		& s\text{-}\lim_{k \to \infty} X^{(k)}_i = X_i, \text{ and} \\
		& s\text{-}\lim_{k \to \infty} Y^{(k)}_j = Y_j
	\end{align*}
	for all $1 \leq i \leq n$ and $1 \leq j \leq m$ (with the strong limit computed as bounded linear maps acting on $L_2(\fA, \varphi)$).  Then
	\begin{align*}
		\limsup_{k \to \infty} \chi^*\left(\vX^{(k)} \sqcup \vY^{(k)} : (B_\ell, B_r)\right) \leq \chi^*(\vX \sqcup \vY : (B_\ell, B_r)).
	\end{align*}
\end{prop}

\begin{proof}
	By assumption there exists a constant $C > 0$ such that
	\[
		C^2 \geq \varphi \left( \sum^n_{i=1}\left( X_i^{(k)}\right)^2 + \sum^m_{j=1} \left(Y_j^{(k)}\right)^2  \right)
	\]
	for all $k$ and 
	\[
		C^2 \geq \varphi \left( \sum^n_{i=1}  X_i^2  + \sum^m_{j=1}  Y_j^2 \right).
	\]
	By Theorem \ref{thm:fisher-info-after-perturbing-by-semis}, if $S^{(k)}_1, \ldots, S^{(k)}_n, T^{(k)}_1, \ldots, T^{(k)}_m$ are $(0, 1)$ semicircular variables such that 
	\[
		\left(B_\ell \left\langle \vX^{(k)} \right\rangle, B_r\left\langle \vY^{(k)}\right\rangle\right) \cup \left\{\left(S^{(k)}_i, 1\right)\right\}^n_{i=1}\cup \left\{\left(1, T^{(k)}_j\right)\right\}^m_{j=1}
	\]
	are bi-free, then
	\begin{align*}
		 \frac{n+m}{1+t} - \Phi^*\left( \vX^{(k)}  + \sqrt{t} \vS^{(k)}  \sqcup \vY^{(k)} + \sqrt{t} \vT^{(k)}  : (B_\ell, B_r)\right)  \leq \frac{n+m}{1+t} - \frac{n+m}{t + (n+m)^{-1} C^2}.
	\end{align*}
	Since $\frac{n+m}{1+t} - \frac{n+m}{t + (n+m)^{-1} C^2}$ is integrable and since
	\begin{align*}
		&\limsup_{k \to \infty} \frac{n+m}{1+t} - \Phi^*\left(\vX^{(k)}  + \sqrt{t} \vS^{(k)}  \sqcup \vY^{(k)} + \sqrt{t} \vT^{(k)} : (B_\ell, B_r)\right) \\
		& \leq \frac{n+m}{1+t} - \Phi^*\left(\vX  + \sqrt{t} \vS  \sqcup \vY + \sqrt{t} \vT : (B_\ell, B_r)\right)
	\end{align*}
	by Proposition \ref{prop:fisher-strong-convergence-bounds}, the result follows by the Dominated Convergence Theorem.
\end{proof}

\begin{prop}
	\label{prop:Fisher-is-the-derivative-of-entropy}
	Let $\vX, \vY$ be tuples of self-adjoint operators of lengths $n$ and $m$ respectively, and let $B_\ell$, $B_r$ be unital, self-adjoint subalgebras of a C$^*$-non-commutative probability space $(\fA, \varphi)$.  Suppose further that $S_1, \ldots, S_n, T_1, \ldots, T_m$ are $(0, 1)$ semicircular variables in $\fA$ such that 
	\[
		(B_\ell \langle \vX \rangle, B_r\langle \vY\rangle) \cup \{(S_i, 1)\}^n_{i=1}\cup \{(1, T_j)\}^m_{j=1}
	\]
	are bi-free and there are no algebraic relations other than possibly the commutation of left and right operators. 

	For $t \in [0, \infty)$, let
	\[
		g(t) = \chi^*(\vX + \sqrt{t} \vS \sqcup \vY + \sqrt{t} \vT : (B_\ell, B_r)).
	\]
	Then $g : [0, \infty) \to \bR \cup \{-\infty\}$ is a concave, continuous, increasing function such that $g(t) \geq \frac{n+m}{2} \log(2 \pi e t)$ and, when $g(t) \neq -\infty$, 
	\[
		\lim_{\epsilon \to 0+} \frac{1}{\epsilon} (g(t+\epsilon) - g(t)) = \frac{1}{2} \Phi^*(  \vX + \sqrt{t} \vS \sqcup \vY + \sqrt{t} \vT : (B_\ell, B_r) ).
	\]
\end{prop}

\begin{proof}
	Let $S'_1, \ldots, S'_n, T'_1, \ldots, T'_m$ be $(0, 1)$ semicircular variables in $\fA$ (or a larger C$^*$-non-commutative probability space) such that 
	\[
		(B_\ell \langle \vX \rangle, B_r\langle \vY\rangle) \cup \{(S_i, 1)\}^n_{i=1}\cup \{(1, T_j)\}^m_{j=1}\cup \{(S'_i, 1)\}^n_{i=1}\cup \{(1, T'_j)\}^m_{j=1}
	\]
	are bi-free.  Then for all $\epsilon > 0$ we have that
	\begin{align*}
		 \Phi^*(\vX  + \sqrt{t+ \epsilon} \vS \sqcup \vY  + \sqrt{t+ \epsilon} \vT : (B_\ell, B_r))
		&=\Phi^*(\vX + \sqrt{t} \vS  + \sqrt{\epsilon}\vS' \sqcup \vY  + \sqrt{t} \vT+ \sqrt{\epsilon}\vT' : (B_\ell, B_r)) \\
		& \geq \Phi^*(\vX + \sqrt{t} \vS \sqcup  \vY + \sqrt{t} \vT : (B_\ell, B_r))
	\end{align*}
	by Proposition \ref{prop:stam-inequality}.   Hence $g$ is increasing.

	If $t_0 \geq 0$, $\epsilon > 0$,  $g(t_0) \neq -\infty$, and 
	\[
		h(t) = \Phi^*(\vX + \sqrt{t} \vS \sqcup  \vY + \sqrt{t} \vT : (B_\ell, B_r))
	\]
	is as in Theorem \ref{thm:fisher-info-after-perturbing-by-semis}, the above computations show
	\begin{align*}
		g(t_0 + \epsilon) - g(t_0) 
		&=\frac{1}{2} \int^\infty_0  \frac{n+m}{1+t} - h(t + t_0 + \epsilon) \, dt - \frac{1}{2} \int^\infty_0  \frac{n+m}{1+t} - h(t + t_0 ) \, dt \\
		&=\frac{1}{2}\int^{t_0 + \epsilon}_{t_0} h(t) \, dt.
	\end{align*}
	Since $h(t)$ is right continuous and decreasing by Theorem \ref{thm:fisher-info-after-perturbing-by-semis}, we see that $g$ is concave, continuous, and
	\[
		\lim_{\epsilon \to 0+} \frac{1}{\epsilon} (g(t+\epsilon) - g(t)) =\frac{1}{2}  h(t) = \frac{1}{2} \Phi^*( \vX + \sqrt{t} \vS \sqcup  \vY + \sqrt{t} \vT : (B_\ell, B_r) ).
	\]
	Furthermore, by Theorem \ref{thm:fisher-info-after-perturbing-by-semis}
	\begin{align*}
		g(t_0) &\geq \frac{n+m}{2} \log(2 \pi e) + \frac{1}{2} \int^\infty_0 \frac{n+m}{1+t} - \frac{n+m}{t+t_0} \, dt  =\frac{n+m}{2}\log(2 \pi e t_0). \qedhere
	\end{align*}
\end{proof}

As it is unknown whether non-microstate free entropy behaves well with respect to all transformations performed on the variables, we prove only the following in the bi-free setting.
Again we are limited to transformations on only the left or only the right variables as per the comments after Proposition \ref{prop:fisher-information-unaffected-by-orthogonal-transform}.

\begin{prop}
\label{prop:unitary-conjugates-and-bi-free-entropy}
	Let $\vX, \vY$ be tuples of self-adjoint operators of lengths $n$ and $m$ respectively and let $B_\ell$, $B_r$ be unital, self-adjoint subalgebras of a C$^*$-non-commutative probability space $(\fA, \varphi)$ with no algebraic relations other than the possibility of left and right operators commuting.
	Let $U = [u_{i,j}]$ be an $n \times n$ unitary matrix with real entries.
	If for each $1 \leq i \leq n$ we define
	\[
	X'_i =  \sum^n_{k=1} u_{i,k} X_k,	
	\]
	then
	\[
		\chi^*\left(\vX' \sqcup \vY : (B_\ell, B_r)   \right) = \chi^*(\vX \sqcup \vY : (B_\ell, B_r))
	\]
	A similar holds for the right variables.
\end{prop}

\begin{proof}
	Let $S_1, \ldots, S_n, T_1, \ldots, T_m$ be $(0, 1)$ semicircular variables such that 
	\[
		(B_\ell \langle \vX \rangle, B_r\langle \vY\rangle) \cup \{(S_i, 1)\}^n_{i=1}\cup \{(1, T_j)\}^m_{j=1}
	\]
	are bi-free with no algebraic relations other than the possibility of left and right operators commuting.  If for each $1 \leq i \leq n$ we define
	\[
		S'_i =  \sum^n_{k=1} u_{i,k} S_k,
	\]
	then $S'_1, \ldots, S'_n, T_1, \ldots, T_m$ are $(0, 1)$ semicircular variables such that 
	\[
		\left(B_\ell \left\langle   \vX' \right\rangle, B_r\langle \vY\rangle \right) \cup \{(S'_i, 1)\}^n_{i=1}\cup \{(1, T_j)\}^m_{j=1}
	\]
	are bi-free.  By Proposition \ref{prop:fisher-information-unaffected-by-orthogonal-transform}, 
	\begin{align*}
		 \Phi^*(\vX + \sqrt{t} \vS \sqcup \vY + \sqrt{t} \vT : (B_\ell, B_r))= \Phi^*\left( \vX' + \sqrt{t} \vS' \sqcup \vY + \sqrt{t} \vT : (B_\ell, B_r)   \right)
	\end{align*}
	and hence the result follows.
\end{proof}

In the case of scaling transformations, we have the following.

\begin{prop}
\label{prop:non-microstate-entropy-and-scaling}
	Let $\vX, \vY$ be tuples of self-adjoint operators of lengths $n$ and $m$ respectively, and let $B_\ell$, $B_r$ be unital, self-adjoint subalgebras of a C$^*$-non-commutative probability space $(\fA, \varphi)$.  Let $\lambda \in \bR \setminus \{0\}$.  Then
	\[
		\chi^*(\lambda \vX \sqcup \lambda \vY : (B_\ell, B_r)) = (n+m)\log|\lambda| + \chi^*(\vX \sqcup \vY : (B_\ell, B_r)).
	\]
\end{prop}

\begin{proof}
	It suffices to prove the result for $\lambda > 1$ as this also implies the result for $\lambda < -1$ since $-I_n$ and $-I_m$ are unitary matrices so we can apply Proposition \ref{prop:unitary-conjugates-and-bi-free-entropy}.  Notice this then implies the result for $0 < |\lambda| \leq 1$ via using $\lambda^{-1}$.  For $\lambda > 1$, we see that
	\begin{align*}
		& \chi^*(\lambda \vX\sqcup \lambda \vY : (B_\ell, B_r)) \\
		& = \frac{1}{2} \int^\infty_0 \left(\frac{(n+m)\lambda^{-2}}{\lambda^{-2} + t\lambda^{-2}} - \lambda^{-2} \Phi^*(\vX + \sqrt{t\lambda^{-2}} \vS  \sqcup \vY  + \sqrt{t\lambda^{-2}} \vT  : (B_\ell, B_r) ) \right) \, dt + \frac{n+m}{2} \log (2\pi e) \\
		&= \frac{1}{2} \int^\infty_0 \left(\frac{n+m}{\lambda^{-2} + s} - \Phi^*(\vX  + \sqrt{s} \vS \sqcup \vY + \sqrt{s} \vT : (B_\ell, B_r) ) \right) \, ds + \frac{n+m}{2} \log (2\pi e) \\
		&=\chi^*(\vX \sqcup \vY : (B_\ell, B_r)) -\frac{1}{2} \int^{\lambda^{-2} - 1}_0 \frac{n+m}{1+s} \, ds \\
		&=  \chi^*(\vX \sqcup \vY : (B_\ell, B_r)) + (n+m) \log|\lambda|. \qedhere
	\end{align*}
\end{proof}

In the case of finite bi-free Fisher information, we have a lower bound on the non-microstate bi-free entropy.

\begin{prop}
	\label{prop:finite-fisher-implies-finite-entropy}
	Let $\vX, \vY$ be tuples of self-adjoint operators of lengths $n$ and $m$ respectively and let $B_\ell$, $B_r$ be unital, self-adjoint subalgebras of a C$^*$-non-commutative probability space $(\fA, \varphi)$ with no algebraic relations other than possibly left and right operators commuting.   If 
	\[
		\Phi^*(\vX \sqcup \vY : (B_\ell, B_r)) < \infty,
	\]
	then
	\begin{align*}
		\chi^*(\vX \sqcup \vY : (B_\ell, B_r))  \geq \frac{n+m}{2} \log\left(\frac{2\pi (n+m) e}{\Phi^*(\vX \sqcup \vY : (B_\ell, B_r))}\right)> - \infty.
	\end{align*}
\end{prop}

\begin{proof}
	Let $\lambda = \Phi^*(\vX \sqcup \vY : (B_\ell, B_r))$.

	Let $S_1, \ldots, S_n, T_1, \ldots, T_m$ be $(0, 1)$ semicircular variables such that 
	\[
		(B_\ell \langle \vX \rangle, B_r\langle \vY\rangle) \cup \{(S_i, 1)\}^n_{i=1}\cup \{(1, T_j)\}^m_{j=1}
	\]
	are bi-free.  By the bi-free Stam inequality (Proposition \ref{prop:stam-inequality}), we see for all $t \in (0, \infty)$ that
	\begin{align*}
		\Phi^*(\vX + \sqrt{t}\vS  \sqcup \vY + \sqrt{t}\vT : (B_\ell, B_r)) \leq \frac{1}{\frac{1}{\lambda} + \frac{t}{n+m}} = \frac{n+m}{\frac{n+m}{\lambda} + t}.
	\end{align*}
	Hence
	\begin{align*}
		\chi^*(\vX \sqcup \vY : (B_\ell, B_r))   & \geq \frac{n+m}{2} \log(2 \pi e) + \frac{1}{2} \int^\infty_0 \frac{n+m}{1+t} - \frac{n+m}{\frac{n+m}{\lambda} + t} \, dt \\
		&= \frac{n+m}{2} \log(2\pi e) + \frac{n+m}{2} \log\left( \frac{n+m}{\lambda}\right). \qedhere
	\end{align*}
\end{proof}

Additional lower bounds can be obtained in the tracially bi-partite setting using the non-microstate free entropy.

\begin{thm}
	\label{thm:non-micro-converting-rights-to-lefts}
	Let $\vX, \vY$ be tracially bi-partite tuples of self-adjoint operators of lengths $n$ and $m$ respectively in a C$^*$-non-commutative probability space $(\fA, \varphi)$.  Suppose there exists another C$^*$-non-commutative probability space $(\A_0, \tau_0)$  and tuples of self-adjoint operators $\vX', \vY' \in \A_0$ of lengths $n$ and $m$ respectively such that $\tau_0$ is tracial on $\A_0$ and
	\[
		\varphi(X_{i_1} \cdots X_{i_p} Y_{j_1} \cdots Y_{j_q}) = \tau_0(X'_{i_1} \cdots X'_{i_p} Y'_{j_q} \cdots Y'_{j_1})
	\]
	for all $p,q \in \bN \cup \{0\}$, $i_1, \ldots, i_p \in \{1,\ldots, n\}$, and $j_1, \ldots, j_q \in \{1, \ldots, m\}$. 
	Then
	\[
		\chi^*(\vX', \vY') \leq \chi^*(\vX \sqcup \vY).
	\]
\end{thm}

\begin{proof}
	First suppose that $\{(S_i, 1)\}^n_{i=1} \cup \{(1, T_j)\}^m_{j=1}$ have a bi-free central limit distribution that is bi-free from $(B_\ell \langle \vX \rangle, B_r\langle \vY\rangle)$ and that $\{S'_1, \ldots, S'_n, T'_1, \ldots, T'_m\}$ are free semicircular operators that are free from $\{\vX', \vY'\}$.
	It can be verified that for all $t \in (0, \infty)$, for all $p,q \in \bN \cup \{0\}$, and for all $i_1, \ldots, i_p \in \{1,\ldots, n\}$ and $j_1, \ldots, j_q \in \{1, \ldots, m\}$,
	\begin{align*}
		& \varphi((X_{i_1} + \sqrt{t} S_{i_1}) \cdots (X_{i_p} + \sqrt{t} S_{i_p}) (Y_{j_1} + \sqrt{t} T_{j_1}) \cdots (Y_{j_q} + \sqrt{t} T_{j_q})) \\
		&  = \tau_0((X'_{i_1} + \sqrt{t} S'_{i_1}) \cdots (X'_{i_p} + \sqrt{t} S'_{i_p}) (Y'_{j_q} + \sqrt{t} T'_{j_q}) \cdots (Y'_{j_1} + \sqrt{t} T'_{j_1})).
	\end{align*}
	Therefore, due to the definition of the free and bi-free entropies under consideration, it suffices to show that if 
	\begin{align*}
		\xi_i &= \J_\ell(X_i : (\bC \langle \hat{\vX}_i\rangle, \bC\langle \vY\rangle)), \\
	\xi'_i &= \J_\ell(X'_i : \bC \langle \hat{\vX}'_i, \vY'\rangle)), \\
	\eta_j &= \J_r(Y_j : (\bC \langle\vX\rangle, \bC\langle \hat{\vY}_j\rangle)), \text{ and} \\
\eta'_j &= \J_r(Y'_j : \bC \langle \hat{\vX}'_i,  \hat{\vY}'_j\rangle)) 
\end{align*}
all exist, then 
\[
	\sum^n_{i=1} \left\|\xi'_i\right\|^2_2 + \sum^m_{j=1} \left\|\eta'_j\right\|^2_2  \geq \sum^n_{i=1} \left\|\xi_i\right\|^2_2 + \sum^m_{j=1} \left\|\eta_j\right\|^2_2;
\]
this then passes to all times $t$ by applying the same but replacing $X_i$ by $X_i + \sqrt{t}S_i$, et cetera.
The existence follows from \cite{V1998-2}*{Corollary 3.9} and Theorem \ref{thm:conj-perturb-by-semis}.  The inequality then follows from Lemma \ref{lem:converting-rights-to-lefts}.
\end{proof}

\section{Non-Microstate Bi-Free Entropy Dimension}
\label{sec:Entropy-Dimension}

In this section, we extend the notion of non-microstate free entropy dimension to the bi-free setting and generalize the basic properties.

\begin{defn}
\label{defn:entropy-dimension}
Let $\vX, \vY$ be tuples of self-adjoint operators of length $n$ and $m$ respectively, and let $B_\ell$, $B_r$ be unital, self-adjoint subalgebras of a C$^*$-non-commutative probability space $(\fA, \varphi)$.  The \emph{$n$-left, $m$-right, non-microstate bi-free entropy dimension  of $(\vX, \vY)$ relative to $(B_\ell, B_r)$} is defined by
\[
	\delta^*(\vX \sqcup  \vY : (B_\ell, B_r)) = (n+m) + \limsup_{\epsilon \to 0^+} \frac{\chi^*(\vX + \sqrt{\epsilon} \vS \sqcup \vY + \sqrt{\epsilon} \vT : (B_\ell, B_r))}{|\log(\sqrt\epsilon)|}
\]
where $S_1, \ldots, S_n, T_1, \ldots, T_m$ are self-adjoint operators in (a larger) $\fA$ that have centred semicircular distributions with variance 1 such that 
	\[
		(B_\ell \langle \vX \rangle, B_r\langle \vY\rangle) \cup \{(S_i, 1)\}^n_{i=1}\cup \{(1, T_j)\}^m_{j=1}
	\]
are bi-free.

In the case that $B_\ell = B_r = \bC$, the non-microstate bi-free entropy dimension of $(\vX, \vY)$ relative to $(B_\ell, B_r)$ is called the \emph{non-microstate bi-free entropy of $(\vX, \vY)$} and is denoted $\delta^* (\vX \sqcup \vY)$.
\end{defn}

Clearly if $m = 0$ then $\delta^*(\vX \sqcup \vY)$ is the non-microstate free entropy dimension of $\vX$ and if $n = 0$ then $\delta^*(\vX \sqcup \vY)$ is the non-microstate free entropy dimension of $\vY$.  Consequently, the non-microstate bi-free entropy dimension is an extension of the non-microstate free entropy dimension.

To justify the terminology that non-microstate bi-free entropy dimension is a dimension, we note its value of bi-free central limit distributions.

\begin{thm}
Let $(\{S_k\}^{n}_{k=1}, \{S_k\}^{n+m}_{k=n+1})$ be a centred self-adjoint bi-free central limit distribution with respect to $\varphi$ with $\varphi(S^2_k) = 1$ for all $k$.  Recall that the joint distribution is completely determined by the positive matrix
\[
A = [a_{i,j}] = [\varphi(S_iS_j)] \in \M_n(\bR).
\]
Then
\[
\delta^*( S_1, \ldots, S_n \sqcup S_{n+1}, \ldots, S_{n+m} ) = \mathrm{rank}(A).
\]
\end{thm}
\begin{proof}
	Let $(\{T_k\}^{n}_{k=1}, \{T_k\}^{n+m}_{k=n+1})$ be a centred self-adjoint bi-free central limit distribution with respect to $\varphi$, bi-free from $(\set{S_k}_{k=1}^n, \set{S_k}_{k=n+1}^{n+m})$, with 
\[
\varphi(T_iT_j) = \begin{cases}
1 & \text{if }i = j \\
0 & \text{if }i \neq j
\end{cases}.
\]
If we define $Z_{k,\epsilon} = S_k + \sqrt\epsilon T_k$ for all $1 \leq k \leq n+m$, then $(\{Z_k\}^{n}_{k=1}, \{Z_k\}^{n+m}_{k=n+1})$ is a centred self-adjoint bi-free central limit distribution with respect to $\varphi$ with 
\[
\varphi(Z_{i,\epsilon}Z_{j,\epsilon}) = \begin{cases}
1 + \epsilon & \text{if }i = j \\
\varphi(S_iS_j) & \text{if }i \neq j
\end{cases}
\]
and
\[
\delta^*( S_1, \ldots, S_n \sqcup S_{n+1}, \ldots, S_{n+m} ) = (n+m) + \limsup_{\epsilon \to 0^+} \frac{\chi^*(Z_{1,\epsilon}, \ldots, Z_{n,\epsilon} \sqcup Z_{n+1,\epsilon}, \ldots, Z_{n+m,\epsilon})}{|\log(\sqrt\epsilon)|}.
\]
By applying Proposition \ref{prop:non-microstate-entropy-and-scaling} and Example \ref{exam:entropy-bi-free-central}, we see that
\begin{align*}
&\chi^*(Z_{1,\epsilon}, \ldots, Z_{n,\epsilon} \sqcup Z_{n+1,\epsilon}, \ldots, Z_{n+m,\epsilon})\\
&= (n+m) \log(\sqrt{1 + \epsilon}) + \chi^*\left(\frac{1}{\sqrt{1+\epsilon}}Z_{1,\epsilon}, \ldots, \frac{1}{\sqrt{1+\epsilon}}Z_{n,\epsilon} \sqcup \frac{1}{\sqrt{1+\epsilon}}Z_{n+1,\epsilon}, \ldots, \frac{1}{\sqrt{1+\epsilon}}Z_{n+m,\epsilon}\right) \\
&=  \frac{n+m}{2} \log(1 + \epsilon) + \frac{n+m}{2} \log(2\pi e)
	+ \frac{1}{2} \log\left(   \det\left( \left(1 - \frac{1}{1+\epsilon}\right) I_{n+m} +    \frac{1}{1 + \epsilon}A  \right)\right) \\
&=  \frac{n+m}{2} \log(2\pi e) + \frac{1}{2} \log\left(   \det\left( \epsilon I_{n+m} +     A  \right)\right).
\end{align*}
As $A$ is a positive matrix and thus diagonalizable, we know that
\[
\det\left( \epsilon I_{n+m} +     A  \right) = \epsilon^{\mathrm{nullity}(A)} p(\epsilon)
\]
where $p$ is a polynomial of degree $\mathrm{rank}(A)$ with real coefficients that does not vanish at 0.  Consequently, we obtain that
\begin{align*}
&\delta^*( S_1, \ldots, S_n \sqcup S_{n+1}, \ldots, S_{n+m} ) \\
&= (n+m) + \limsup_{\epsilon \to 0^+} \frac{\frac{n+m}{2} \log(2\pi e) + \frac{1}{2} \log(\epsilon^{\mathrm{nullity}(A)}p(\epsilon))}{|\log(\sqrt\epsilon)|} \\
&= (n+m) + \limsup_{\epsilon \to 0^+} \frac{\frac{n+m}{2} \log(2\pi e) + \frac12\mathrm{nullity}(A) \log(\epsilon) + \frac{1}{2}  \log(p(\epsilon))}{|\log(\sqrt\epsilon)|} \\
&= n+m - \mathrm{nullity}(A) = \mathrm{rank}(A) \qedhere
\end{align*}
\end{proof}

\begin{exam}
Let $(S, T)$ be a bi-free central limit distribution with variances 1 and covariance $c \in [-1,1]$.  Then
\[
\delta^*(S \sqcup T) = \begin{cases}
2 & \text{if } c \neq\pm 1 \\
1 & \text{if } c =\pm 1 \\
\end{cases}.
\]
In particular, the support of the joint distribution of $(S, T)$ has dimension $\delta^*(S\sqcup T)$: indeed, if $c \neq \pm 1$ then $(S, T)$ has joint distribution with support $[-2, 2]^2 \subset \bR^2$ by \cite{HW2016}, while otherwise it is supported on the line $y = cx$.
\end{exam}

Due to the previous results in this paper, the basic properties of non-microstate free entropy dimension carry-forward to the bi-free setting.

\begin{prop}
Let $\vX, \vY, \vX', \vY'$ be tuples of self-adjoint operators of lengths $n$, $m$, $n'$, and $m'$ respectively.
Let $B_\ell$, $B_r$, $C_\ell$, $C_r$ be unital, self-adjoint subalgebras of a C$^*$-non-commutative probability space $(\fA, \varphi)$ with no algebraic relations other than possibly left and right operators commuting.  
	\begin{enumerate}
		\item We have
			\begin{align*}
				 \delta^*(\vX, \vX' \sqcup \vY, \vY' : (B_\ell\vee C_\ell, B_r \vee C_r)) \leq\delta^*(\vX \sqcup \vY : (B_\ell, B_r)) + \delta^*(\vX' \sqcup \vY' : (C_\ell, C_r)).
			\end{align*}
		\item If
			\[
				(B_\ell\langle \vX \rangle, B_r\langle \vY \rangle ) \qqand (C_\ell\langle \vX' \rangle, C_r\langle \vY' \rangle )
			\]
			are bi-free, then the inequality in part (1) is an equality.
		\item If $C_\ell \subseteq B_\ell$ and $C_r \subseteq B_r$, then 
			\[
				\delta^*(\vX \sqcup \vY : (B_\ell, B_r))  \leq \delta^*(\vX \sqcup \vY : (C_\ell, C_r)) .
			\]
		\item If
			\[
				(B_\ell\langle \vX \rangle, B_r\langle \vY \rangle ) \qqand (C_\ell , C_r  )
			\]
			are bi-free,  then
			\[
				\delta^*(\vX \sqcup \vY : (B_\ell, B_r)) = \delta^*(\vX \sqcup \vY : (B_\ell  \vee C_\ell, B_r \vee C_r)).
			\]
	\end{enumerate}
\end{prop}
\begin{proof}
This result immediately follows from Definition \ref{defn:entropy-dimension}, Proposition \ref{prop:properties-of-bi-free-entropy}, and the fact that the semicircular perturbations have zero covariance.
\end{proof}

Moreover, we have an unsurprising upper bound for the non-microstate bi-free entropy dimension.

\begin{prop}
Let $\vX, \vY$ be tuples of self-adjoint operators of length $n$ and $m$ respectively, and let $B_\ell$, $B_r$ be unital, self-adjoint subalgebras of a C$^*$-non-commutative probability space $(\fA, \varphi)$.
Then
\[
\delta^*(\vX \sqcup \vY) \leq n+m.
\]
\end{prop}
\begin{proof}
If $S_1, \ldots, S_n, T_1, \ldots, T_m$ are self-adjoint operators in (a larger) $\fA$ that have centred semicircular distributions with variance 1 such that 
	\[
		(B_\ell \langle \vX \rangle, B_r\langle \vY\rangle) \cup \{(S_i, 1)\}^n_{i=1}\cup \{(1, T_j)\}^m_{j=1}
	\]
are bi-free, then using bi-freeness, we see that
\[
\varphi\left( \sum^n_{i=1} (X_i + \sqrt{\epsilon} S_i)^2 + \sum^m_{j=1} (Y_j + \sqrt{\epsilon} T_j)^2\right) = C^2 + (n+m)\epsilon.
\]
Therefore Proposition \ref{prop:upper-bound-non-microstate-entropy-based-on-L2-norm} implies that 
\begin{align*}
\delta^*(\vX \sqcup \vY) &= (n+m) + \limsup_{\epsilon \to 0^+} \frac{\chi^*(\vX + \sqrt{t} \vS \sqcup \vY + \sqrt{t} \vT : (B_\ell, B_r))}{|\log(\sqrt\epsilon)|} \\
& \leq (n+m) + \limsup_{\epsilon \to 0^+} \frac{\frac{n+m}{2} \log\left( \frac{2 \pi e }{n+m} (C^2 + (n+m)\epsilon)\right)}{|\log(\sqrt\epsilon)|} \\
& \leq (n+m) + \limsup_{\epsilon \to 0^+} \frac{\frac{n+m}{2} \log\left( \frac{2 \pi e }{n+m} (C^2 + (n+m))\right)}{|\log(\sqrt\epsilon)|} \\
&= n+m.\qedhere
\end{align*}
\end{proof}

Furthermore, a similar known lower bound for the non-microstate free entropy dimension extends to the bi-free setting.

\begin{prop}
\label{prop:lower-bound-for-entropy-dimension}
Let $\vX, \vY$ be tuples of self-adjoint operators of length $n$ and $m$ respectively, and let $B_\ell$, $B_r$ be unital, self-adjoint subalgebras of a C$^*$-non-commutative probability space $(\fA, \varphi)$.  Then
\[
\delta^*(\vX \sqcup \vY) \geq (n+m) - \limsup_{\epsilon \to 0^+} \epsilon \Phi^*(\vX + \sqrt{\epsilon} \vS \sqcup \vY + \sqrt{\epsilon} \vT : (B_\ell, B_r))
\]
where $S_1, \ldots, S_n, T_1, \ldots, T_m$ are self-adjoint operators in (a larger) $\fA$ that have centred semicircular distributions with variance 1 such that 
	\[
		(B_\ell \langle \vX \rangle, B_r\langle \vY\rangle) \cup \{(S_i, 1)\}^n_{i=1}\cup \{(1, T_j)\}^m_{j=1}
	\]
are bi-free.  Furthermore, if 
\[
\lim_{\epsilon \to 0^+} \epsilon \Phi^*(\vX + \sqrt{\epsilon} \vS \sqcup \vY + \sqrt{\epsilon} \vT : (B_\ell, B_r))
\]
exists, then the inequality becomes an equality.
\end{prop}
\begin{proof}
Let
\[
L = \limsup_{\epsilon \to 0^+} \epsilon \Phi^*(\vX + \sqrt{\epsilon} \vS \sqcup \vY + \sqrt{\epsilon} \vT : (B_\ell, B_r)).
\]
Given $\delta > 0$ there exists an $\epsilon_0 > 0$ such that 
\[
\Phi^*(\vX + \sqrt{\epsilon} \vS \sqcup \vY + \sqrt{\epsilon} \vT : (B_\ell, B_r)) \leq \frac{L + \delta}{\epsilon}
\]
for all $0 < \epsilon < \epsilon_0$.  Therefore, the same computation as used in Proposition \ref{prop:Fisher-is-the-derivative-of-entropy} implies for all $0 < \epsilon < \epsilon_0$ that
\begin{align*}
&\chi^*(\vX + \sqrt{\epsilon_0} \vS \sqcup \vY + \sqrt{\epsilon_0} \vT : (B_\ell, B_r)) - \chi^*(\vX + \sqrt{\epsilon} \vS \sqcup \vY + \sqrt{\epsilon} \vT : (B_\ell, B_r)) \\
&= \frac{1}{2} \int^{\epsilon_0}_{\epsilon} \Phi^*(\vX + \sqrt{t} \vS \sqcup \vY + \sqrt{t} \vT : (B_\ell, B_r)) \, dt \\
& \leq \frac{1}{2} \int^{\epsilon_0}_{\epsilon}  \frac{L + \delta}{t} \, dt \\
&= \frac{L+\delta}{2} \ln\left(\frac{\epsilon_0}{\epsilon}\right).
\end{align*}
Hence
\[
\chi^*(\vX + \sqrt{\epsilon} \vS \sqcup \vY + \sqrt{\epsilon} \vT : (B_\ell, B_r) \geq \chi^*(\vX + \sqrt{\epsilon_0} \vS \sqcup \vY + \sqrt{\epsilon_0} \vT : (B_\ell, B_r)) - \frac{L+\delta}{2} \ln(\epsilon_0) + \frac{L+\delta}{2} \ln(\epsilon)
\]
for all $0 < \epsilon < \epsilon_0$.  Therefore, since 
\[
\chi^*(\vX + \sqrt{\epsilon_0} \vS \sqcup \vY + \sqrt{\epsilon_0} \vT : (B_\ell, B_r))
\]
is finite (Proposition \ref{prop:upper-bound-non-microstate-entropy-based-on-L2-norm} gives an upper bound, while Theorem \ref{thm:fisher-info-after-perturbing-by-semis} and Proposition \ref{prop:finite-fisher-implies-finite-entropy} give a the lower bound), we obtain that
\[
\liminf_{\epsilon \to 0^+} \frac{\chi^*(\vX + \sqrt{\epsilon} \vS \sqcup \vY + \sqrt{\epsilon} \vT : (B_\ell, B_r))}{|\log(\sqrt\epsilon)|} \geq -\paren{L + \delta}
\]
for all $\delta > 0$.  
Hence
\begin{align*}
\delta^*(\vX \sqcup \vY : (B_\ell, B_r)) &= (n+m) + \limsup_{\epsilon \to 0^+} \frac{\chi^*(\vX + \sqrt{\epsilon} \vS \sqcup \vY + \sqrt{\epsilon} \vT : (B_\ell, B_r))}{|\log(\sqrt\epsilon)|} \\
& \geq (n+m) + \liminf_{\epsilon \to 0^+} \frac{\chi^*(\vX + \sqrt{\epsilon} \vS \sqcup \vY + \sqrt{\epsilon} \vT : (B_\ell, B_r))}{|\log(\sqrt\epsilon)|} \\
&\geq (n+m) - L
\end{align*}
as desired.

If 
\[
\lim_{\epsilon \to 0^+} \epsilon \Phi^*(\vX + \sqrt{\epsilon} \vS \sqcup \vY + \sqrt{\epsilon} \vT : (B_\ell, B_r))
\]
exists, given $\delta > 0$ there exists an $\epsilon_0 > 0$ such that 
\[
\Phi^*(\vX + \sqrt{\epsilon} \vS \sqcup \vY + \sqrt{\epsilon} \vT : (B_\ell, B_r)) \geq \frac{L - \delta}{\epsilon}
\]
for all $0 < \epsilon < \epsilon_0$.  By performing similar computations to those above with reversed inequalities, we obtain
\[
\delta^*(\vX \sqcup \vY) = (n+m) - \lim_{\epsilon \to 0^+} \epsilon \Phi^*(\vX + \sqrt{\epsilon} \vS \sqcup \vY + \sqrt{\epsilon} \vT : (B_\ell, B_r))
\]
as desired.
\end{proof}

The above lower bound in conjunction with previous results in this paper immediately give us the following.

\begin{cor}
Let $\vX, \vY$ be tuples of self-adjoint operators of length $n$ and $m$ respectively, and let $B_\ell$, $B_r$ be unital, self-adjoint subalgebras of a C$^*$-non-commutative probability space $(\fA, \varphi)$.  Then
\begin{enumerate}
\item $\delta^*(\vX \sqcup \vY) \geq 0$, and
\item if $\Phi^*(\vX \sqcup \vY) < \infty$, then $\delta^*(\vX \sqcup \vY) = n+m$.
\end{enumerate}
\end{cor}
\begin{proof}
Let $S_1, \ldots, S_n, T_1, \ldots, T_m$ be self-adjoint operators in (a larger) $\fA$ that have centred semicircular distributions with variance 1 such that 
	\[
		(B_\ell \langle \vX \rangle, B_r\langle \vY\rangle) \cup \{(S_i, 1)\}^n_{i=1}\cup \{(1, T_j)\}^m_{j=1}
	\]
are bi-free.  Since Theorem \ref{thm:fisher-info-after-perturbing-by-semis} implies that 
\[
0 \leq \limsup_{\epsilon \to 0^+} \epsilon \Phi^*(\vX + \sqrt{\epsilon} \vS \sqcup \vY + \sqrt{\epsilon} \vT : (B_\ell, B_r)) \leq \limsup_{\epsilon \to 0^+} \epsilon  \frac{n+m}{\epsilon} = n+m,
\]
we easily obtain that $\delta^*(\vX \sqcup \vY) \geq 0$ by Proposition \ref{prop:lower-bound-for-entropy-dimension}.
Furthermore, if $\lambda := \Phi^*(\vX \sqcup \vY) < \infty$, then by applying the bi-free Stam inequality (Proposition~\ref{prop:stam-inequality}) in the same manner as in Proposition \ref{prop:finite-fisher-implies-finite-entropy} we see
\[
0 \leq \limsup_{\epsilon \to 0^+} \epsilon \Phi^*(\vX + \sqrt{\epsilon} \vS \sqcup \vY + \sqrt{\epsilon} \vT : (B_\ell, B_r)) \leq \limsup_{\epsilon \to 0^+} \epsilon \frac{1}{\frac{1}{\lambda} + \frac{\epsilon}{n+m}} = 0.
\]
Hence Proposition \ref{prop:lower-bound-for-entropy-dimension} implies that $\delta^*(\vX \sqcup \vY) = n+m$, as desired.
\end{proof}

\section{Additivity of Bi-free Fisher Information}
\label{sec:Additive-Bi-Free-Fisher-Info}

By \cite{V1999} it is known that if $X_1, \ldots, X_n$ are self-adjoint operators such that
\[
\Phi^*(X_1, \ldots, X_n) = \Phi^*(X_1, \ldots, X_k) + \Phi^*(X_{k+1}, \ldots, X_n) < \infty,
\]
then $\{X_1, \ldots, X_k\}$ and $\{X_{k+1}, \ldots, X_n\}$ are freely independent. Thus it is natural to ask:

\begin{ques}
	\label{ques:Fisher}
	Is the converse to Proposition \ref{prop:fisher-info-with-bifree-things} true?  That is, if 
	\begin{align*}
		\Phi^*&(\vX, \vX' \sqcup \vY, \vY' : (B_\ell \vee C_\ell, B_r \vee C_r))= \Phi^*(\vX   \sqcup \vY   : (B_\ell, B_r)) + \Phi^*( \vX' \sqcup  \vY' : (C_\ell, C_r))
	\end{align*}
	and all terms are finite, is it the case that 
	\[
		(B_\ell  \langle \vX \rangle, B_r \langle \vY\rangle) \qqand (C_\ell \langle \vX' \rangle, C_r\langle \vY'\rangle)
	\]
	are bi-free?
\end{ques}
Question \ref{ques:Fisher} is of interest as verifying collections are bi-freely independent has been difficult so any equivalent characterizations would be exceptional.  In this section we illustrate some partial results towards such a characterization in the case that $B_\ell = B_r = C_\ell = C_r = \bC$ and $n=m=1$.  In this case, we are trying to demonstrate that if
\[
\Phi^*(X \sqcup Y) = \Phi^*(X) + \Phi^*(Y) < \infty,
\]
then $X$ and $Y$ are classically independent with respect to $\varphi$.  In particular, this would imply $X$ and $Y$ commute in distribution.

We begin with the following where we do not assume $X$ and $Y$ commute in distribution.

\begin{lem}
\label{lem:additive-bi-free-fisher-gives-info-about-conjugate-variables}
Let $(X, Y)$ be a pair of self-adjoint operators in a C$^*$-non-commutative probability space $(\fA, \varphi)$.  If $\Phi^*(X \sqcup Y) < \infty$ (so $\Phi^*(X), \Phi^*(Y) < \infty$ by Proposition \ref{prop:Fisher-supadditive}), then
\[
\Phi^*(X \sqcup Y) = \Phi^*(X) + \Phi^*(Y)
\]
if and only if
\[
\J_\ell(X : (\bC, \bC\langle Y \rangle)) = \J(X : \bC) \qqand \J_r(Y : (\bC\langle X \rangle, \bC)) = \J(Y : \bC).
\]
\end{lem}
\begin{proof}
Clearly if
\[
\J_\ell(X : (\bC, \bC\langle Y \rangle)) = \J(X : \bC) \qqand \J_r(Y : (\bC\langle X \rangle, \bC)) = \J(Y : \bC).
\]
then $\Phi^*(X \sqcup Y) = \Phi^*(X) + \Phi^*(Y)$.  

Conversely, let $\A = \bC \langle X, Y\rangle$, $\X = \bC\langle X \rangle$, $\Y = \bC\langle Y \rangle$, and let $P : L_2(\A, \varphi) \to L_2(\X, \varphi)$ and $Q : L_2(\A, \varphi) \to L_2(\Y, \varphi)$ be the orthogonal projections  onto their codomains.  Since $\Phi^*(X \sqcup Y) < \infty$, we know that
\[
\xi = \J_\ell(X : (\bC, \bC\langle Y \rangle)) \qqand \eta = \J_r(Y : (\bC \langle X \rangle, \bC))
\]
exist, and
\[
P(\xi) = \J(X : \bC) \qqand Q(\eta) = \J(Y : \bC)
\]
by Remark \ref{rem:conjugate-variables-to-free-conjugate}.  Therefore, as
\begin{align*}
\left\|P(\xi)\right\|^2 + \left\|Q(\eta)\right\|^2 &= \left\|\J(X : \bC)\right\|^2 + \left\|\J(Y : \bC)\right\|^2 \\
&= \Phi^*(X) + \Phi^*(Y)\\
&= \Phi^*(X \sqcup Y) \\
&= \left\|\xi\right\|^2 + \left\|\eta\right\|^2,
\end{align*}
it must be the case that $\xi = P(\xi)$ and $\eta = Q(\eta)$.
\end{proof}

To proceed, we recall the following result of Dabrowski \cite{D2010}*{Lemma 12}.
Suppose $\vX$ is an $n$-tuple of algebraically free self-adjoint operators that generate a tracial von Neumann algebra $(\fM, \tau)$.
If $\J(X_1 : \bC \langle\hat{\vX}_1 \rangle)$ exists, then the operator $(\tau \otimes 1) \circ \partial_{X_1} : \bC\ang{\vX} \to \bC\ang{\vX} $ extends to a bounded linear operator, which will also be denoted $(\tau \otimes 1) \circ \partial_{X_1}
$, from $\fM$ to $L_2(\fM, \tau)$.
Note that although the result is stated only for tuples with $n \geq 2$, it extends to the $n = 1$ case as well (by, for example, formally including a semi-circular variable free from $X_1$ and then restricting the resulting $(\tau\otimes1)\circ\partial_{X_1}$ to the $W^*$-algebra generated by $X_1$.
In fact, we will only use this result in the bi-free setting applied to a single left or a single right operator in which case traciality is trivial.

Using Dabrowski's result, we can state the following continuing on what was learned in Lemma \ref{lem:additive-bi-free-fisher-gives-info-about-conjugate-variables}.

\begin{lem}
\label{lem:bi-free-conjugate-equal-free-conjugate-implies-zero-expectations}
Let $(X, Y)$ be a pair of self-adjoint operators in a C$^*$-non-commutative probability space $(\fA, \varphi)$, let $\A = \bC\langle X, Y\rangle$, let $\X = \bC\langle X\rangle$, and let $P : L_2(\A, \varphi) \to L_2(\X, \varphi)$ be the orthogonal projection onto the codomain. 
Suppose the distribution of $X$ is absolutely continuous with respect to the Lebesgue measure with density $f_X$.
Suppose further that for each $m \in \bN$ there exists an element $E(Y^m) \in C^*(X)$ such that
\[
\langle E(Y^m), \zeta\rangle_{L_2(\X, \varphi)} = \langle Y^m, \zeta\rangle_{L_2(\A, \varphi)}
\]
for all $\zeta \in L_2(\X, \varphi) \subseteq L_2(\A, \varphi)$ (i.e. $E(Y^m) = P Y^m P \in C^*(X)$).  

If $\Phi^*(X \sqcup Y) < \infty$ and $\J_\ell(X : (\bC, \bC\langle Y \rangle)) = \J(X : \bC)$, then
\[
\left[(\varphi \otimes 1) \circ \partial_X\right](E(Y^m)) = 0
\]
for all $m \in \bN$.
\end{lem}

\begin{proof}
Since  $\J_\ell(X : (\bC, \bC\langle Y \rangle)) = \J(X : \bC)$ exists, we see for all $n \in \bN$ that
\begin{align*}
	\sum^{n}_{i=1} \varphi(Y^m X^{n-i}) \varphi(X^{i-1}) &= (\varphi \otimes \varphi)(\partial_{\ell, X}(Y^m X^n)) \\
	&= \langle Y^m X^n \J_\ell(X : (\bC, \bC\langle Y \rangle)) , 1\rangle_{L_2(\A, \varphi)} \\
	&= \langle  X^n \J(X : \bC) , Y^m\rangle_{L_2(\A, \varphi)} \\
	&= \langle  X^n \J(X : \bC) , E(Y^m)\rangle_{L_2(\X, \varphi)}.
\end{align*}
Since $E(Y^m) \in C^*(X)$, there exists a sequence of self-adjoint polynomials $(q_k(X))_{k \geq 1}$ from $\X$ such that $\lim_{k \to \infty} \left\|q_k(X) - E(Y^m)\right\| = 0$.  Hence, as this implies $\lim_{k \to \infty} \left\|q_k(X) - E(Y^m)\right\|_2 = 0$, we obtain that
\begin{align*}
	\sum^{n}_{i=1} \varphi(Y^m X^{n-i}) \varphi(X^{i-1})
	&= \lim_{k \to \infty}  \langle  X^n \J(X : \bC) , q_k(X)\rangle_{L_2(\X, \varphi)} \\
	&= \lim_{k \to \infty}  \langle  q_k(X) X^n \J(X : \bC) , 1\rangle_{L_2(\X, \varphi)}\\
	&= \lim_{k \to \infty} (\varphi \otimes \varphi)(\partial_X(q_k(X) X^n))\\
	&= \lim_{k \to \infty} (\varphi \otimes \varphi)\paren{\partial_X(q_k(X))(1 \otimes X^n) + (q_k(X) \otimes 1) \partial_X(X^n) } \\
	&= \lim_{k \to \infty} (\varphi \otimes \varphi)\paren{\partial_X(q_k(X))(1 \otimes X^n)}  + \sum^{n}_{i=1} \varphi(q_k(X)X^{n-i}) \varphi(X^{i-1}).
\end{align*}
Therefore, as
\[
	\lim_{k \to \infty} \sum^{n}_{i=1} \varphi(q_k(X)X^{n-i}) \varphi(X^{i-1}) = \sum^{n}_{i=1} \varphi(Y^m X^{n-i}) \varphi(X^{i-1})
\]
via inner product computations, we obtain that
\[
\lim_{k \to \infty} (\varphi \otimes \varphi)(\partial_X(q_k(X))(1 \otimes X^n)) = 0
\]
for all $n \in \bN$.  Hence
\[
\lim_{k \to \infty} (\varphi \otimes \varphi)(\partial_X(q_k(X))(1 \otimes r(X)))  =0
\]
for all $r(X) \in \bC \langle X \rangle$.

Fix $m \in \bN$, and let $Z_m := \sq{(\varphi\otimes1)\circ\partial_X}(E(Y^m))$.
Choose any polynomial $r(X) \in \X$.
Then, as $L_2(\X, \varphi)$ can be expressed as $L_2(\bR, f_X(x) \, dx)$, as $\lim_{k \to \infty} \left\|q_k(X) - E(Y^m)\right\| = 0$, and as $(\varphi \otimes 1) \circ \partial_X$ is norm continuous, we obtain that
\begin{align*}
\langle Z_m, r(X) \rangle_{L_2(\X, \varphi)} &= \int_{\bR} Z_m(x) \overline{r(x)} f_X(x) \, dx \\
&= \lim_{k \to \infty}  \int_{\bR} \left( [(\varphi \otimes 1) \circ \partial_X](q_k(X)) \right) (x) \overline{r(x)} f_X(x) \, dx \\
 &= \lim_{k \to \infty}  \int_{\bR} \int_\bR  \left(\partial_X(q_k(X)) \right) (y, x) \overline{r(x)} f_X(x)f_X(y) \, dy  \, dx \\
  &= \lim_{k \to \infty} (\varphi \otimes \varphi)(\partial_X(q_k(X))(1 \otimes r(X))) \\
  &= 0.
\end{align*}
It follows that $Z_m = 0$ since $\X$ is dense in $L_2(\X, \varphi)$.
\end{proof}

\begin{rem}
Unfortunately we cannot easily see how to replace the condition $E(Y^m) \in C^*(X)$ with $E(Y^m) \in W^*(X)$ as we only know operator norm continuity of $(\varphi \otimes 1) \circ \partial_X$.
\end{rem}

\begin{rem}
In the case $(X, Y)$ is bi-partite with joint distribution $f(x, y) \, d\lambda_2$, it is easy to compute $E(Y^m)$.  Indeed
\[
	E(Y^m)(x) = \int_\bR y^m \frac{f(x,y)}{f_X(x)} \, dy.
\]
Therefore, provided $f(x,y)$ is sufficiently nice, it is not too much to assume that $E(Y^m) \in C^*(X)$.  
\end{rem}

 In fact, in the bi-partite case, the converse of Lemma \ref{lem:bi-free-conjugate-equal-free-conjugate-implies-zero-expectations} holds.

\begin{lem}
\label{lem:zero-expectations-implies-bi-free-conjugate-variable-equals-free-conjugate-variable}
Under the assumptions of Lemma \ref{lem:bi-free-conjugate-equal-free-conjugate-implies-zero-expectations} together with the assumption that $(X, Y)$ is bi-partite, if $\Phi^*(X \sqcup Y) < \infty$ and 
\[
\left[(\varphi \otimes 1) \circ \partial_X\right](E(Y^m)) = 0
\]
for all $m \in \bN$, then $\J_\ell(X : (\bC, \bC\langle Y \rangle)) = \J(X : \bC)$.
\end{lem}

\begin{proof}
Since $E(Y^m) \in C^*(X)$, there exists a sequence of self-adjoint polynomials $(q_k(X))_{k \geq 1} \subseteq \X$ such that $\lim_{k \to \infty} \left\|q_k(X) - E(Y^m)\right\| = 0$.   Hence for all $r(x) \in \bC\langle X \rangle$ we have as in the proof of Lemma \ref{lem:bi-free-conjugate-equal-free-conjugate-implies-zero-expectations} that
\begin{align*}
0 = \langle Z_m, r(X) \rangle_{L_2(\X, \varphi)}  = \lim_{k \to \infty} (\varphi \otimes \varphi)(\partial_X(q_k(X))(1 \otimes r(X))).
\end{align*}  
Hence   
\[
\lim_{k \to \infty} (\varphi \otimes \varphi)(\partial_X(q_k(X))(1 \otimes X^n)) = 0
\]
for all $n \in \bN$.  Therefore
\begin{align*}
	(\varphi \otimes \varphi)(\partial_{\ell, X}(Y^m X^n))
	&= \sum^{n}_{i=1} \varphi(Y^m X^{n-i}) \varphi(X^{i-1}) \\
	&=\lim_{k \to \infty} (\varphi \otimes \varphi)(\partial_X(q_k(X))(1 \otimes X^n))  + \sum^{n}_{i=1} \varphi(q_k(X)X^{n-i}) \varphi(X^{i-1}) \\
	&= \lim_{k \to \infty} (\varphi \otimes \varphi)(\partial_X(q_k(X))(1 \otimes X^n) + (q_k(X) \otimes 1) \partial_X(X^n) )\\
	&= \lim_{k \to \infty} (\varphi \otimes \varphi)(\partial_X(q_k(X) X^n))\\
	&= \lim_{k \to \infty} \varphi(q_k(X) X^n \J(X : \bC))\\
	&= \varphi( Y^m X^n \J(X : \bC)).
\end{align*}
Therefore, as the above holds for all $m,n \in \bN$ and as $(X,Y)$ is bi-partite, we obtain that $\J_\ell(X : (\bC, \bC\langle Y \rangle)) = \J(X : \bC)$ as desired.
\end{proof}

\begin{rem}
Lemma \ref{lem:bi-free-conjugate-equal-free-conjugate-implies-zero-expectations} is useful in the context of Question \ref{ques:Fisher} as, by Lemma \ref{lem:additive-bi-free-fisher-gives-info-about-conjugate-variables}, $\Phi^*(X \sqcup Y) = \Phi^*(X) + \Phi^*(Y) < \infty$ implies $\J_\ell(X : (\bC, \bC\langle Y \rangle)) = \J(X : \bC)$ and thus Lemma \ref{lem:bi-free-conjugate-equal-free-conjugate-implies-zero-expectations} implies  $\left[(\varphi \otimes 1) \circ \partial_X\right](E(Y^m)) = 0$ for all $m \in \bN$.  This later condition often implies  $E(Y^m)$ is a scalar.  In this case, we must have $E(Y^m) = \varphi(Y^m)$ and that $X$ and $Y$ are independent.  

For an example where $\left[(\varphi \otimes 1) \circ \partial_X\right](E(Y^m)) = 0$ implies $E(Y^m)$ is scalar, consider the case that $X$ is a semicircular variable with variance 1.  Recall that if $U_0(X) = 1$, $U_1(X) = X$, and $U_n(X) = U_{N-1}(X)X - U_{n-2}(X)$, then $\{U_n(X)\}_{n \geq 0}$ form an orthonormal basis for $L_2(\X, \varphi)$.  If $T = (\varphi \otimes 1) \circ \partial_X$, then clearly $T(U_0(X)) = 0$, $T(U_1(X)) = U_0(X)$, and, by induction,
\begin{align*}
T(U_n(X)) &= T(U_{n-1}(X)X) - T(U_{n-2}(X)) \\
&= (\varphi \otimes 1)(\partial_X(U_{n-1}(X))(1 \otimes X) + (U_n(X) \otimes 1)) - U_{n-3}(X) \\
&= T(U_{n-1}(X)) X + \varphi(U_n(X)) - U_{n-3}(X) \\
&= U_{n-2}(X) X + 0 - U_{n-3}(X) \\
&= U_{n-1}(X).
\end{align*}
Hence $T$ is the annihilation operator on the Chebyshev polynomials so we easily see that if $\zeta \in W^*(X)$ has the property that $T(\zeta) = 0$ if and only if $\zeta = \lambda U_0(X) = \lambda$ for some $\lambda \in \bC$.
\end{rem}
Consequently, we have the following.

\begin{cor}
Under the assumptions of Lemma \ref{lem:bi-free-conjugate-equal-free-conjugate-implies-zero-expectations}, if $X$ is a semicircular operator and
\[
\Phi^*(X \sqcup Y) = \Phi^*(X) + \Phi^*(Y) < \infty,
\]
then $X$ and $Y$ are independent.
\end{cor}

Unfortunately, it is possible that the kernel of $(\varphi \otimes 1) \circ \partial_X$ contains more than just scalar operators.  For example, if we take
\[
f_X(x) =c \left( \sqrt{4-(x-4)^2} \chi_{[2, 6]} + \sqrt{4-(x+4)^2} \chi_{[-6,-2]}\right)
\]
where $c$ is a normalization constant to make $f$ a probability distribution, it is not too hard to see the free conjugate variable exists.  Moreover $X^{-1} \in C^*(X)$ and
\[
\left[(\varphi \otimes 1) \circ \partial_X\right](X^{-1}) = - \varphi(X^{-1}) X^{-1} = 0
\]
as $\varphi(X^{-1}) = 0$.  However, this does not immediately provide a counter example to Question \ref{ques:Fisher} as we do not know whether $E(Y^m) = X^{-1}$ is possible for some selection of $Y$ such that the joint density $f(x,y)$ satisfies all of the necessary properties.

\section{Open Questions}
\label{sec:Ques}

We conclude this paper with several important and interesting questions raised in this paper in addition to the question of whether results in bi-free probability may be applied to obtain results pertaining to von Neumann algebras.

To begin, recall the previous questions: Question \ref{ques:Fisher} and Question \ref{ques:domains}.
The interest in Question \ref{ques:Fisher} was discussed in Section \ref{sec:Additive-Bi-Free-Fisher-Info} and the importance of Question \ref{ques:domains} is that the free analogue is an essential fact in many works (e.g. \cites{CS2014, D2010, D2016, GS2014, MSW2017}).

One interest in regards to bi-freeness is the following.
\begin{ques}
	\label{ques:left-to-right-always-works}
	In the context of Theorem~\ref{thm:non-micro-converting-rights-to-lefts}, is the supremum of $\chi^*(\vX', \vY')$ over acceptable tuples $\vX', \vY'$ always equal to $\chi^*(\vX\sqcup\vY)$?
\end{ques}
It is worth pointing out that there are often choices of $\vX'$ and $\vY'$ for which equality is not attained; for example, if $\vX$ contains at least one variable and $\vY$ consists of a single variable, the tuples $\vX$ and $\vY$ themselves satisfy the conditions of $\vX'$ and $\vY'$, but $\chi^*(\vX, \vY) = -\infty$ regardless of $\chi^*(\vX \sqcup \vY)$ (since the algebraic relation $X_1Y = YX_1$ is satisfied).
The answer to Question \ref{ques:left-to-right-always-works} is affirmative for the bi-free central limit distributions and for independent distributions.
A general answer to Question \ref{ques:left-to-right-always-works} would be of interest as it directly relates the free and bi-free non-microstate entropies in the case that the bi-free entropy is tracially bi-partite.

One question related to Question \ref{ques:left-to-right-always-works} is the following.
\begin{ques}
	\label{ques:integration}
	Let $(X,Y)$ be a bi-partite pair with joint distribution $\mu$.  Is there an integration formula involving just $\mu$ to compute $\chi^*(X \sqcup Y)$?
\end{ques}
Question \ref{ques:integration} arises from the fact that \cite{V1998-2} demonstrated that if $X$ is a self-adjoint operator with distribution $\mu$, then the non-microstate free entropy of $X$ is
\[
	\chi^*(X) = \frac{1}{2} \log(2\pi) + \frac{3}{4} + \int_\bR \int_\bR \log|s-t| \, d\mu(s) \, d\mu(t).
\]
Of course, an affirmative answer to both Questions \ref{ques:left-to-right-always-works} and \ref{ques:integration} would enable the computation of the non-microstate free entropy of two self-adjoint operators via an integration formula.
Thus it is unlikely that both Question \ref{ques:left-to-right-always-works} and Question \ref{ques:integration} can be answered in the affirmative.
In addition, a negative answer to Question \ref{ques:integration} would give merit to the statement that bi-free probability is not a probability theory for measures on $\bR^2$ but completely a non-commutative probability theory.

In terms of the proof of this formula from \cite{V1998-2}, we appear to have all the necessary tools to prove a formula (if a formula exists at all).  Given a pair $(X,Y)$ of commuting self-adjoint operators and self adjoint operators $S, T$ with centred semicircular distribution with variance 1 such that $\{(X,Y), (S, 1), (1, T)\}$ are bi-free, for all $t \in \bR$ let $(X_t, Y_t) = (X + \sqrt{t} S, Y + \sqrt{t} T)$.  If $(X_t, Y_t)$ have distributions $f_{X_t}$ and $f_{Y_t}$ respectively and joint distribution $f_t(x,y)$, let
\begin{align*}
	h_{X,t}(x) &= \int_\bR \frac{f_{X_t}(s)}{x-s} \, ds, \\
	h_{Y,t}(y) &= \int_\bR \frac{f_{Y_t}(r)}{y-r} \, dr, \\
	H_{X,t}(x,y) &= \int_\bR \frac{f_t(s,y)}{x-s} \, ds, \text{ and} \\
	H_{Y,t}(x,y) &= \int_\bR \frac{f_t(x,r)}{y-r} \, dr.
\end{align*}
It is possible to show that
\[
	\frac{1}{\pi} \frac{d f_{X_t}}{dt}(x) = - h_{x,t}(x) \frac{df_{X_t}}{dx}(x) - f_{X_t}(x) \frac{d h_{x,t}}{dx}(x)
\]
and
\[
	\frac{1}{\pi} \frac{df_t}{dt}(x,y) = - h_{X,t}(x) \frac{df_t}{dx}(x,y) - f_{X_t}(x) \frac{d H_{X,t}}{dx}(x,y) - h_{Y,t}(y) \frac{df_t}{dy}(x,y)  - f_{Y_t}(y) \frac{dH_{Y,t}}{dy}(x,y).
\]
Using the integral formula from \cite{V1998-2} as the definition for $\chi^*(X_t)$ and the first differential equation, one shows that
\[
	\frac{d(\chi^*(X_t))}{dt} = \frac{1}{2}\Phi^*(X_t)
\]
from which the equivalence of definitions then follows.   For the bi-free side, we know from Proposition \ref{prop:Fisher-is-the-derivative-of-entropy} that
\[
	\frac{d(\chi^*(X_t \sqcup Y_t))}{dt} = \frac{1}{2}\Phi^*(X_t \sqcup Y_t).
\]
As Proposition \ref{prop:conjugate-variable-integral-description-bi-partite} gives a formula for $\Phi^*(X_t \sqcup Y_t)$ in terms of $f_t, f_{X_t}, f_{Y_t}, h_{X,t}, h_{Y,t}, H_{X,t}$, and $H_{Y,t}$, one needs `simply' modify the integral expression for $\Phi^*(X_t \sqcup Y_t)$ to invoke the above differential equations to obtain a $\frac{d}{dt}$ of a new expression which will be the formula for $\chi^*(X_t \sqcup Y_t)$.  Such a formula has remained elusive to us.

Of course, the most natural question is
\begin{ques}
	Does the microstate bi-free entropy from \cite{CS2017} agree with the above non-microstate bi-free entropy for  tracially bi-partite collections? 
\end{ques}

In the free setting, \cite{BCG2003} first showed that the microstate free entropy is always less than the non-microstate free entropy.  Thus perhaps a good starting point would be a bi-free version of \cite{BCG2003}.  Of course much progress was made towards the converse in \cite{D2016}.

\section*{Acknowledgements}
 The authors would like to thank Yoann Dabrowski for discussions related to $(\varphi \otimes 1) \circ \partial_X$.

\end{document}